\documentclass[preprint,1p,11pt]{elsarticle}
\usepackage{amsfonts, amssymb}
\usepackage{amsmath,amsthm}
\usepackage{graphicx}
\usepackage{tikz}

\usepackage{color} 
   \definecolor{cites}{rgb}{0.50 , 0.00 , 0.00}  
   \definecolor{urls} {rgb}{0.00 , 0.00 , 0.50}  
   \definecolor{links}{rgb}{0.00 , 0.00 , 0.50}   
\definecolor{sccol}{rgb}{0,0,0.5}

\usepackage{hyperref} 
\hypersetup{  
      colorlinks=true,   
      citecolor=cites,   
      urlcolor=urls,     
      linkcolor=links,   
      pdftitle={Spectral inclusions and approximations of finite and infinite banded matrices},  
      pdfauthor={Simon N. Chandler-Wilde, Marko Lindner}, 
      pdfstartview=FitH,       
      bookmarksopen=false      
}

\usepackage{soul}

\newcommand{\bl}[1]{\textcolor{blue}{#1}}   

\newcommand\A{{\mathcal A}}

\newcommand\C{{\mathbb C}}
\newcommand\D{{\mathbb D}}

\newcommand\N{{\mathbb N}}

\newcommand\I{{\mathbb I}}
\newcommand\J{{\mathbb J}}

\newcommand\T{{\mathbb T}}
\newcommand\R{{\mathbb R}}

\newcommand\Z{{\mathbb Z}}

\newcommand\eps{{\varepsilon}}
\newcommand\Spec{{\rm Spec}\,}  
\newcommand\Specn{{\rm Spec}}   
\newcommand\speps{{\rm spec}_\eps}
\newcommand\Speps{{\rm Spec}_\eps}

\newcommand\dist{{\rm dist}}
\newcommand\conv{{\rm conv}}

\newcommand\diag{{\rm diag}}

\newcommand\ri{{\rm i}}

\newcommand\supp{{\rm supp}}

\newcommand\Hto{
   \unitlength0.1ex
   \begin{picture}(30,15)
   \put(13,16){\makebox(0,0)[]{\tiny\rm H}}
   \put(15,5){\makebox(0,0)[]{$\to$}}
   \end{picture}
}

\newtheorem{theorem}{Theorem}[section]
\newtheorem{lemma}[theorem]{Lemma}
\newtheorem{corollary}[theorem]{Corollary}
\newtheorem{proposition}[theorem]{Proposition}

\newenvironment{remark}
 {\par\noindent\refstepcounter{theorem}{\bf Remark \thetheorem\ }}
 {\raisebox{1mm}{\framebox{}}\par\pagebreak[2]}

\newcommand\Proofend{\rule{2mm}{2mm}}

\numberwithin{equation}{section}   
\numberwithin{figure}{section}  

\newcounter{abccounter}

\renewcommand{\thefootnote}{\arabic{footnote}}
\newcommand{\symbolfootnote}[1]{\begingroup
\renewcommand{\thefootnote}{\fnsymbol{footnote}}\footnote{#1}\endgroup}

\begin{document}
\begin{frontmatter}


\title{Gershgorin-type spectral inclusions for matrices}

\author[label1]{Simon N. Chandler-Wilde \corref{cor1}}
\affiliation[label1]{organization={Department of Mathematics and Statistics, University of Reading},
          city={Reading},
                       postcode={RG6 6AX},
            country={UK}}

\author[label2]{Marko Lindner} 

\affiliation[label2]{organization={Institut Mathematik, TU Hamburg (TUHH)},
            city={Hamburg},
            postcode={D-21073}, 
            country={Germany}}

\cortext[cor1]{This pdf contains the research article as its first 40 pages, the pdf part of the Supplementary Materials as the remainder of the document. {\em E-mail addresses:} \href{mailto:s.n.chandler-wilde@reading.ac.uk}{\tt s.n.chandler-wilde@reading.ac.uk} (S.~N.~Chandler-Wilde), \href{mailto:lindner@tuhh.de}{\tt lindner@tuhh.de} (M.~Lindner)}

\begin{abstract}
In this paper we derive sequences of  Gershgorin-type inclusion sets for the spectra and pseudospectra of finite matrices. In common with previous generalisations of the classical Gershgorin bound for the spectrum, our inclusion sets are based on a block decomposition. In contrast to previous generalisations that treat the matrix as a perturbation of a block-diagonal submatrix, our arguments treat the matrix as a perturbation of a block-tridiagonal matrix, which can lead to sharp spectral bounds, as we show for the example of large Toeplitz matrices. Our inclusion sets, which take the form of unions of pseudospectra of square or rectangular submatrices,  build on our own recent work on inclusion sets for bi-infinite matrices
[Chandler-Wilde, Chonchaiya, Lindner, {\em J. Spectr. Theory} {\bf 14}, 719--804 (2024)].
\end{abstract}



\end{frontmatter}

\section{Introduction, our main results, and the related literature} \label{sec:intro}
Let $A\in \C^{M\times M}$ be an $M\times M$ matrix with complex entries, which we will write in block form as
\begin{equation} \label{eq:block}
A = [a_{ij}]_{i,j=1}^N
\end{equation}
where $a_{ij}\in \C^{m_i \times m_j}$ is a subblock of $A$, with $m_i\in \{1,\ldots,M-1\}$, for $i=1,\ldots,N$, so that $M\geq N>1$ and
$$
\sum_{i=1}^N m_i = M. 
$$ 
The aim of this paper is to derive new families of inclusion sets for the spectrum and pseudospectrum of $A$. Recall that $\Spec A$, the spectrum of $A$, is its set of eigenvalues, and, for $\eps>0$, 
\begin{equation} \label{eq:speceps}
\Speps A := \{\lambda\in \C: \|(A-\lambda I)^{-1}\|^{-1}\leq \eps\}, 
\end{equation}
is the (closed) $\eps$-pseudospectrum of $A$. Here $I$ is the identity matrix, we adopt the convention that $\|(A-\lambda I)^{-1}\|^{-1}:= 0$ when $A-\lambda I$ is singular, so that $\Spec A \subset \Speps A$, for $\eps>0$, and, except where explicitly indicated otherwise, our norms are $2$-norms, so that 
$$
\|(A-\lambda I)^{-1}\|^{-1} = s_{\rm min}(A-\lambda I), \qquad \lambda \in \C,
$$
where $s_{\rm min}(E)$ denotes the smallest singular value of a matrix $E$. With our convention the definition \eqref{eq:speceps} also makes sense when $\eps=0$, coinciding with the set of eigenvalues of $A$, i.e., $\Specn_0 A = \Spec A$.  

Using our own recent results in \cite{CW.Heng.ML:SpecIncl1}, where we derived analogous inclusion sets for the spectra and pseudospectra of bi-infinite matrices, we will derive two sequences of inclusion sets for the finite matrix case that we will term, following \cite{CW.Heng.ML:SpecIncl1}, the $\tau$ and $\tau_1$ methods. The starting point for the derivation of each of these methods is to establish inclusion sets for $B=[b_{ij}]_{i,j=1}^N$, defined by
$$
b_{ij} := \left\{\begin{array}{ll} a_{ij}, & |i-j|\leq 1,\\ 0, & \mbox{otherwise},\end{array}\right.
$$ 
which we term the {\it tridiagonal part} of $A$. We term $C:=A-B$ the {\em remaining part} of $A$, and set
\begin{eqnarray} \nonumber
r_L(A) &:=& \max\{\|a_{2,1}\|,\|a_{3,2}\|,\ldots,\|a_{N,N-1}\|\},\\ \label{eq:rdef}
 r_U(A)& :=& \max\{\|a_{1,2}\|,\|a_{2,3}\|,\ldots,\|a_{N-1,N}\|\},\\ \nonumber
r(A) &:=& r_L(A)+r_U(A),
\end{eqnarray}
so $r_L(A)$ and $r_U(A)$ are the maxima of the norms of the submatrices on the first subdiagonal and first superdiagonal of $A$, respectively, and $r(A)$ is an upper bound for both the maximum row sum and the maximum column sum of the norms of the block off-diagonal entries of $B$.

We will shortly, in Sections \ref{sec:tau}-\ref{sec:tau1}, define our sequences of spectral and pseudospectral inclusion sets for the $\tau$ and $\tau_1$ methods, respectively. Our main results are Theorems \ref{thm:tau} and \ref{thm:tau1}, statements that the sets we define are, indeed, inclusion sets. We prove these theorems in \S\ref{sec:main}. Preceding \S\ref{sec:main} we introduce in \S\ref{sec:tools} concepts  and ideas that we need for our arguments. Our main tool is corresponding inclusion sets for bi-infinite matrices  recently established in \cite{CW.Heng.ML:SpecIncl1}. We recall the results we need from \cite{CW.Heng.ML:SpecIncl1} in \S\ref{sec:bi}, where we make, necessary for our later arguments, some generalisation of our earlier results as Theorems \ref{cor:nubound} and \ref{thm:taubi}.

In \S\ref{sec:examples} we apply our spectral inclusion sets to particular matrices. The examples we choose are mainly Toeplitz, this an attractive class because much is known about the asymptotics of spectra and pseudospectra for large matrix order (e.g.,  \cite{ReichelTref,BoeGru,BoeSi2,TrefEmbBook}). In \S\ref{sec:disLap} we treat two canonical tridiagonal examples, the non-normal case where $A$ is a Jordan block, and the real symmetric case where $A$ is a discrete Laplacian. A message from these two simple examples is that for large Toeplitz matrices our sequences of inclusion sets provide accurate approximations to the pseudospectrum, indeed  also to the spectrum when the matrix is Hermitian. This message is confirmed by analysis of the general Toeplitz case in \S\ref{sec:bandT} where we show in Theorem \ref{thm:Tgen} that our $\tau$-method sequence of inclusion sets accurately approximates the pseudospectra of large banded Toeplitz matrices, indeed (Theorem \ref{thm:Herm}) also the spectrum if the matrix is Hermitian. In Theorem  \ref{thm:final} we show the same result for the $\tau_1$-method inclusion set sequence for Toeplitz matrices that 
 have symbol in the so-called Wiener algebra class.

Our $\tau$-method sequence of spectral inclusion sets can be seen as an extension to a family, indexed by a parameter $n\in \N$, of previous block-matrix versions of the Gershgorin theorem  (see \cite{Ostrowski,Feingold,Fiedler,Varga}). These are close to the $\tau$ method in the simplest case $n=1$, as we explain in Remark \ref{rem:case1}. In \S\ref{sec:disLap2}, and in a final \S\ref{sec:2T} studying a family of $2$-Toeplitz tridiagonal Hermitian matrices, we demonstrate that our new sequences of inclusion sets can give much sharper estimates for the spectrum than these previous block-matrix extensions (see Remark \ref{rem:bmG} and \S\ref{sec:T2G}), indeed that these block-matrix extensions can lead to larger inclusion sets than classical Gershgorin.
 In \S\ref{sec:disLap2}, \S\ref{sec:banded}, and \S\ref{sec:T2G} we also make comparison in the banded Hermitian case between our  $\tau$-method inclusion sets for the spectrum and $\R\cap G$, where $G$ is the classical Gershgorin inclusion set. We show, for the families of matrices we study including the large Hermitian Toeplitz class, that, while $\R\cap G$ is a sharp inclusion set in particular cases, our $\tau$-method sequence provides sharp inclusions in every case; see Remark \ref{rem:BandG}, \S\ref{sec:T2G}, and Figure \ref{fig:Toep2}.

In \S\ref{sec:concl} we make concluding remarks and indicate directions for futher work. 

In supplementary materials we provide the Matlab codes used to produce Figures \ref{fig:jordan}-\ref{fig:Toep2}. These materials also contain Matlab code for computing the $\tau$-method eigenvalue inclusion set $\R\cap \Sigma_0^n(A)$ whenever $A$ is Hermitian, and application of this code to check the accuracy of formulae \eqref{eq:Rcap} and \eqref{eq:Rcap2p} below.

\begin{figure}[t]
\begin{center}
\[
B\ =\
\begin{pmatrix}~
\begin{tikzpicture}[scale=0.42]
\foreach \k in {0,0.5,...,4.5}
  \fill[gray!50] (\k,6-\k) rectangle ++(1.5,-1.5);
\draw[gray!30,very thin,step=0.5cm] (0,0) grid (6,6);
\end{tikzpicture}
~\end{pmatrix}
\ =\
\begin{pmatrix}~
\begin{tikzpicture}[scale=0.42]
\foreach \k in {0,0.5,...,4.5}
  \fill[gray!50] (\k,6-\k) rectangle ++(1.5,-1.5);
\draw[gray!30,very thin,step=0.5cm] (0,0) grid (6,6);
\draw[black,line width=0.5mm,step=1.5cm] (0,0) grid (6,6);
\foreach \i in {1,4}
  {
  \foreach \j in {1,4}
     \node at ({1.5*\j-0.75},{6.75-1.5*\i}) {$b_{{\i},{\j}}$};
  }
\end{tikzpicture}
~\end{pmatrix}
\ =\
\begin{pmatrix}~
\begin{tikzpicture}[scale=0.42]
\draw[gray!10,very thin,step=0.5cm] (0,0) grid (6,6);
\draw[gray,line width=0.3mm,step=1.5cm] (0,0) grid (6,6);
\foreach \k [evaluate=\k as \l using int(80-20*\k)] in {1,2,3}
  {
    \fill[gray!\l] ({1.5*(\k-1)},{7.5-1.5*\k}) rectangle ++(3,-3);
    \draw[black,line width=0.7mm] ({1.5*(\k-1)},{7.5-1.5*\k}) rectangle ++(3,-3);
    \pgfmathtruncatemacro{\sub}{\k-1}
    \node at ({1.5*\k-0.5},{6.5-1.5*\k}) {$B_{n,{\sub}}$};
  }
\end{tikzpicture}
~\end{pmatrix}
\]
\end{center}
\caption{Illustration, for the case that $M=12$, $N=4$, $m_i=3$, $i=1,\ldots,N$, and $n=2$, of the part $B$ of the matrix $A$, of its block tridiagonal representation, and of the sub-matrices $B_{n,k}$, for $k=0,\ldots,N-n$, that arise in the $\tau$ method.} \label{fig:matrices}
\end{figure}

\subsection{The $\tau$ method and earlier generalisations of the Gershgorin theorem} \label{sec:tau}
For $n=1,\ldots,N$ and $k=0,\ldots,N-n$, let 
$$
B_{n,k}:=[b_{ij}]_{i,j=k+1}^{k+n},
$$ 
so that $B_{n,k}$ is the $n\times n$ block-tridiagonal matrix consisting of the blocks of $B$ whose rows and columns are in the range $k+1$ through $k+n$. The first sequence  of inclusion sets we propose, termed the $\tau$ method ($\tau$ for truncation), is based on these {\it truncations} or {\it finite sections} $B_{n,k}$ of the block-tridiagonal matrix $B$; see Figure \ref{fig:matrices}. To define these inclusion sets,  for $n\in \N$ let
\begin{equation} \label{eq:epsdef}
\eps_n(A) := 2r(A) \sin(\theta_n(A)/2) + \|C\| \leq 2r(A) \sin(\pi/(2n+4)) + \|C\|\leq \frac{\pi r(A)}{n+2} + \|C\|,
\end{equation}
where $\theta_n(A)$ is the unique solution $t$ in the range
$\left[\frac{\pi}{2n+1}\,,\,\frac{\pi}{n+2}\right]$
of the equation
\begin{equation} \label{eq:theta}
2\sin\Big(\frac{t}{2}\Big)\cos\Big(\Big(n+\frac{1}{2}\Big)t\Big)\ +\ \frac{r_L(A)r_U(A)}{(r(A))^2} \sin\left(\left(n-1\right) t\right)\ =\ 0.
\end{equation}
Note that, in particular, 
$$
\eps_1(A)= r(A)+\|C\|.
$$
For $\eps\geq 0$ and $n=1,\ldots,N-1$ (with $N> n\geq 3$ for the second definition), let
\begin{eqnarray*}
\sigma^n_{\eps}(A) &:=& \bigcup_{k=0}^{N-n} \Specn_{\eps+\eps_n(A)} B_{n,k} \cup \bigcup_{m=1}^{n-1} \left(\Specn_{\eps+\eps_n(A)} B_{m,0}\cup \Specn_{\eps+\eps_n(A)} B_{m,N-m}\right),\\
\widehat \sigma^n_{\eps}(A) &:=&\bigcup_{k=0}^{N-n} \Specn_{\eps+\eps_{n-2}(A)} B_{n,k},
\end{eqnarray*}
and define the {\it $\tau$-method inclusion set $\Sigma^n_{\eps}(A)$} by
\begin{equation} \label{eq:Sigdef}
\Sigma^n_{\eps}(A) := \left\{\begin{array}{ll}\sigma^n_{\eps}(A), & \mbox{if } n\leq 2,\\
\sigma^n_{\eps}(A) \cap \widehat \sigma^n_{\eps}(A),& \mbox{if } n> 2.\end{array} \right.
\end{equation}
The following is our main result for the $\tau$ method.
\begin{theorem}[\bf The $\tau$ method inclusion sets] \label{thm:tau} For $\eps\geq 0$ and $n=1,\ldots,N-1$,
$$
\Speps A \subset \Sigma^n_{\eps}(A), \qquad \mbox{in particular} \qquad \Spec A \subset  \Sigma^n_{0}(A).
$$
\end{theorem}

\noindent To aid comprehension and use of the above result we make a number of remarks.

\vspace{0.7ex} 

\begin{remark}{\bf (The case $n=1$ and previous generalisations of Gershgorin).}  \label{rem:case1} In the case that $n=1$ and $\eps=0$ the above result is close to previous generalisations of the classical Gershgorin theorem  to block matrices. For in that case the bound on $\Spec A$ in Theorem \ref{thm:tau} reduces to
\begin{equation} \label{eq:n=1}
\Spec A  \subset  \Sigma^1_{0}(A) =  \sigma^1_0(A) = \bigcup_{k=0}^{N-1} \Specn_{r(A)+\|C\|} B_{1,k}= \bigcup_{k=1}^{N} \Specn_{r(A)+\|C\|} a_{k,k}. 
\end{equation}
This is reminiscent of block matrix versions of the Gershgorin theorem, discovered independently by Ostrowski \cite{Ostrowski}, Feingold \& Varga \cite{Feingold}, and Fiedler \& Pt\'ak \cite{Fiedler} (and see \cite[Chapter 6]{Varga}). These read that (e.g., \cite[Theorem 6.3]{Varga})
\begin{equation} \label{eq:salas}
\Spec A \subset \bigcup_{k=1}^N  G_k, 
\end{equation}
where
$$
G_k := \{\lambda\in \C: \|(a_{k,k}-\lambda I)^{-1}\|^{-1} \leq r_k\}, \quad r_k := \sum_{j=1, \,j\neq k}^N \|a_{k,j}\|,
$$
and the norms here are any consistent set of matrix norms (see \cite[\S6.1]{Varga} for detail). Of course, the set $G_k$ is precisely an $\eps$-pseudospectrum of $a_{k,k}$, defined with respect to the chosen norm, with $\eps=r_k$. If we take each of the matrix norms in the definition of $G_k$ in \eqref{eq:salas} to be the $2$-norm, then $G_k$ is our $2$-norm pseudospectrum and \eqref{eq:salas} reads
\begin{equation} \label{eq:salas2}
\Spec A \subset \bigcup_{k=1}^N \Specn_{r_k} a_{k,k}.
\end{equation}
In the case that $N=M$, so that each $a_{k,k}\in \C$ is a scalar, \eqref{eq:salas} and \eqref{eq:salas2} reduce to the classical Gershgorin theorem \cite{Gershgorin,Varga}, that
\begin{equation} \label{eq:ger}
\Spec A \subset G:=\bigcup_{k=1}^N  G_k, \, \mbox{ where } \, G_k = \{\lambda\in \C: |a_{k,k}-\lambda| \leq r_k\}, \;\; r_k = \sum_{j=1, \,j\neq k}^N |a_{k,j}|,
\end{equation}
with $G_k$ commonly termed the $k$th {\it Gershgorin disc}.

To make the connection between \eqref{eq:salas2} and \eqref{eq:n=1}, note that, where $D=\diag(a_{1,1},\ldots,a_{N,N})$ is the main diagonal of $A$, each $r_k$ in \eqref{eq:salas2} is a lower bound for the mixed infinity and 2-norm of $A-D$, defined by
$$
\|A-D\|_{\infty,2} := \max_{k\in \{1,\ldots,N\}} \sum_{j=1, \,j\neq k}^N \|a_{k,j}\|,
$$
with equality for some $k$, while $r(A)+\|C\|$ in \eqref{eq:n=1} is an upper bound for $\|A-D\|$, the $2$-norm of $A-D$, with equality, for example, when $B$ is diagonal so that $r(A)=0$ and $A-D=C$. 
\end{remark}

\vspace{1ex}

\begin{remark}{\bf (The case $n\geq 2$ and an example).} Theorem \ref{thm:tau} is equivalent to a statement that $\Speps A\subset  \sigma^n_{\eps}(A)$ for $n\in \N$ and also  $\Spec A\subset  \widehat \sigma^n_{\eps}(A)$ for $n>2$. Both $\sigma^n_{\eps}(A)$ and $\widehat \sigma^n_{\eps}(A)$ are unions of the $\eta$-pseudospectra of finitely many of the submatrices $B_{m,k}$, with $m=n$ for $\widehat \sigma^n_{\eps}(A)$ and $1\leq m\leq n$ for $\sigma^n_{\eps}(A)$, and with 
$\eta \leq \eps +\pi r(A)/n + \|C\|$
for each pseudospectrum. For the example 
 in Figure \ref{fig:matrices} with $N=4$,  and taking $n=2$ as in the figure caption, the Theorem \ref{thm:tau} bound on $\Spec A$ is
\begin{eqnarray*}
\Spec A &\subset  &\Sigma^2_{0}(A) = \sigma^{2}_{0}(A)\\
& = & \Specn_{\eps_2(A)} B_{2,0}\ \cup\ \Specn_{\eps_2(A)} B_{2,1}\ \cup\ \Specn_{\eps_2(A)} B_{2,2}\\ 
& & \qquad \qquad \cup\  \Specn_{\eps_2(A)} B_{1,0}\ \cup\  \Specn_{\eps_2(A)} B_{1,3}\\
& = & \Specn_{\eps_2(A)} B_{2,0}\ \cup\ \Specn_{\eps_2(A)} B_{2,1}\ \cup\ \Specn_{\eps_2(A)} B_{2,2}\\ 
& & \qquad \qquad \cup\  \Specn_{\eps_2(A)} a_{1,1}\ \cup\  \Specn_{\eps_2(A)} a_{4,4},
\end{eqnarray*}
since $B_{1,0}=b_{1,1}=a_{1,1}$ and $B_{1,3}=b_{4,4}=a_{4,4}$, and note that  $b_{1,1}$, $b_{4,4}$, $B_{2,0}$, $B_{2,1}$, and $B_{2,2}$ are displayed in Figure \ref{fig:matrices} and that $\eps_2(A) \leq 2r(A) \sin(\pi/8)+\|C\|<0.77 r(A) + \|C\|$.
\end{remark}

\vspace{1ex}

\begin{remark}{\bf (The case that $A$ is banded)} \label{rem:band} In the case that the $M\times M$ matrix $A$ is banded with band-width $w$, the block-matrix representation  \eqref{eq:block} of $A$ is tridiagonal if the block dimensions are large enough, in particular if $m_i\geq w$, for $i=1,\ldots,N$, in which case $B=A$ and $C=0$. (In Figure \ref{fig:matrices} the matrix $B$ is banded with order $M=12$ and band-width $w=2$, so that, taking $N=4$ and $m_i=3$ for $i=1,\ldots N$, as in Figure \ref{fig:matrices}, its block representation is tridiagonal.) We show examples in \S\ref{sec:disLap} and \S\ref{sec:banded} that illustrate that, for large matrices $A$, the inclusion set $\Sigma^{n}_{\eps}(A)$ for $\Speps A$ can be sharp in such cases if $n$ is large enough ($1\ll n\ll N$), so that
each $\eta$-pseudospectrum of $B_{m,k}$ in the definition of $\Sigma^{n}_{0}(A)$ has a small value of $\eta\leq \pi r(A)/n$.
\end{remark}

\vspace{1ex}

\begin{remark}{\bf (Computation)} \label{rem:comp} 
Computation of $\Sigma^{n}_{\eps}(A)$ needs computation of the pseudospectra of finitely many square matrices  and the computation of $\|C\|$. We envisage that our new inclusion sets will typically be used in cases where the order, $\# B_{n,k}$, of each $B_{n,k}$, namely
$
\# B_{n,k} = \sum_{i=k+1}^{k+n} m_i,
$
 is in the range $[1,1000]$, so that precise computation of pseudospectra of $B_{n,k}$ is feasible (see, e.g., \cite[\S IX]{TrefEmbBook}). On the other hand we expect use in cases where $M$ is so large that exact computation of the $2$-norm of $C$ may be unfeasible, in which case $\|C\|$ can be replaced in the definition   \eqref{eq:epsdef}, and so in the definition of $\Sigma^n_{\eps}(A)$, by some more computationally feasible upper bound, at the cost of an increase in the size of the inclusion set, e.g.,
$$
\|C\| \leq \|C\|_F \qquad \mbox{or} \qquad \|C\| \leq \sqrt{\|C\|_1\|C\|_\infty\,},
$$ 
where $\|\cdot\|_F$ denotes the Frobenius norm. In the special case that $A$ is Hermitian, also $A_{n,k}$ is Hermitian for each $n$ and $k$ so that, where $\D:=\{z\in \C:|z|<1\}$,
$$
\Speps A_{n,k} = \Spec A_{n,k} + \eps \overline{\D} = \bigcup_{\lambda \in \Spec A_{n,k}} (\lambda +  \eps \overline{\D})
$$
is the closed $\eps$-neighbourhood of the eigenvalues of $A_{n,k}$, for $\eps>0$ (see \S\ref{sec:tools}).
\end{remark}

\begin{figure}
\[
\begin{tabular}{cc}
$
\begin{pmatrix}~
\begin{tikzpicture}[scale=0.45]
\fill[gray!20] (1.5,4.5) rectangle ++(3,-3);
\foreach \k in {0,0.5,...,5}
  \fill[gray!50] (\k,6-\k) rectangle ++(1.0,-1.0);
\draw[gray!30,very thin,step=0.5cm] (0,0) grid (6,6);
\draw[black,line width=0.7mm] (1.5,4.5) rectangle ++(3,-3);
\node at (3,3) {\Large $B_{n,k}$};
\end{tikzpicture}
~\end{pmatrix}
$
&
$
\begin{pmatrix}~
\begin{tikzpicture}[scale=0.45]
\fill[gray!20] (1.5,5) rectangle ++(3,-4);
\foreach \k in {0,0.5,...,5}
  \fill[gray!50] (\k,6-\k) rectangle ++(1.0,-1.0);
\draw[gray!30,very thin,step=0.5cm] (0,0) grid (6,6);
\draw[black,line width=0.7mm] (1.5,5) rectangle ++(3,-4);
\node at (3,3) {\Large $B^+_{n,k}$};
\end{tikzpicture}
~\end{pmatrix}
$
\\
the $\tau$-method
&
the $\tau_1$-method
\end{tabular}
\]

\vspace{-2ex}

\caption{Typical submatrices $B_{n,k}$ and $B^+_{n,k}$ of the block-tridiagonal matrix $B$.} \label{fig:three}
\end{figure}

\subsection{The $\tau_1$ method: rectangular submatrices} \label{sec:tau1}

Let $b_{0,1}$ and $b_{N+1,N}$ denote zero matrices of dimensions $1\times m_1$, and $1\times m_N$, respectively. 
For $n=1,\ldots,N$ and $k=0,\ldots,N-n$, let 
$B^+_{n,k}$ 
be the $(n+2)\times n$ block-tridiagonal matrix (see Figure \ref{fig:three})
\begin{equation} \label{eq:Bn+}
B^+_{n,k} := \begin{pmatrix}
b_{k,k+1} & 0&\cdots&0& 0\\\hline
& & B_{n,k}\\\hline
0& 0 &\cdots&0 & b_{k+n+1,k+n}
\end{pmatrix}, 
\end{equation}
with the understanding that the $0$'s are zero matrices of dimensions consistent with the block structure of $B^+_{n,k}$. 
For $k=1,\ldots,N-n-1$,  $B_{n,k}^+$ 
consists precisely of the blocks of $B$ whose rows are in the range $k$ through $k+n+1$ and columns in the range $k+1$ to $k+n$. For all $k=0,\ldots,N-n$, $B^+_{n,k}$ contains all the non-zero blocks of $B$ in columns $k+1$ through $k+n$.  

Our second sequence of inclusion sets, termed the $\tau_1$ method ($\tau_1$ for one-sided truncation), is based on pseudospectra of these {\it one-sided} or {\it rectangular truncations}, $B^+_{n,k}$, of the block-tridiagonal matrix $B$. Our notion of the pseudospectrum of an $(n+2)\times n$ rectangular matrix is essentially that of \cite{WrightTrefethen,TrefEmbBook}. Let $E^+_n$ and $I_n^+$ be block matrices of the same dimensions and with the same block sizes as $B^+_{n,k}$, with $I_n^+$ taking the form 
\begin{equation} \label{eq:In+}
 I^+_n := \begin{pmatrix}
0&\cdots&0\\\hline
&I_n &\\\hline
0&\cdots&0
\end{pmatrix},
\end{equation}
where $I_n$ is an order $n$ block identity matrix. Then we define the $\eps$-pseudospectrum of the rectangular matrix $E^+_n$ by
\begin{equation} \label{eq:recteps}
\Speps E^+_n := \{\lambda\in \C: s_{\min}(E^+_n-\lambda I^+_n)\leq \eps\}.
\end{equation}

To define the $\tau_1$ inclusion sets,  for $n\in \N$ let
\begin{equation} \label{eq:epstau1def}
\eps''_n(A) := 2r(A) \sin(\pi/(2n+2)) + \|C\| \leq \frac{\pi r(A)}{n+1} + \|C\|.
\end{equation}
For $\eps\geq 0$ and $n=1,\ldots,N-1$, define the {\it $\tau_1$-method inclusion set $\Gamma^{n}_{\eps}(A)$} by
\begin{equation} \label{eq:Gamdef}
 \Gamma_{\eps}^{n}(A)\ :=\ \bigcup_{k=0}^{N-n}
\Specn_{\eps+\eps'_n(A)} B^{+}_{n,k}.
\end{equation}
The following is our main result on the $\tau_1$ method. Note that the inclusions here are two-sided, in contrast to Theorem \ref{thm:tau}. Obvious variants of Remarks \ref{rem:band} and \ref{rem:comp} apply in this $\tau_1$ method case.
\begin{theorem}[\bf The $\tau_1$ method inclusion sets] \label{thm:tau1} For $n=1,\ldots,N-1$ and  $\eps\geq 0$,
$$
\Speps A \subset \Gamma^{n}_{\eps}(A) \subset \Specn_{\eps+\eps'_n(A)+2\|C\|}A, 
$$
in particular
$$\Spec A \subset  \Gamma^{n}_{0}(A)  \subset \Specn_{\eps'_n(A)+2\|C\|}A.$$
\end{theorem}
 
\section{Notations and tools} \label{sec:tools}
Throughout, $\N=\{1,2,\ldots\}$, $\Z$, $\R$, and $\C$ denote the sets of natural, integer, real, and complex numbers, respectively,
$\D:=\{z\in\C:|z|<1\}$, and $\T:=\{z\in\C:|z|=1\}$. $\overline S$ denotes the closure of a set $S\subset\C$ and $\bar z$ denotes the complex conjugate of $z\in \C$. For Hilbert spaces $H_1$ and $H_2$, $L(H_1,H_2)$ 
 denotes the space of bounded linear operators from $H_1$ to $H_2$, with $L(H_1,H_1)$ abbreviated as $L(H_1)$, and $H'_1$ denotes the dual space of $H_1$ (space of continuous linear functionals). For $E\in L(H_1)$ the spectrum of $H_1$, $\Spec H_1$, is the set of $\lambda\in \C$ such that $E-\lambda I$ is not invertible, where $I$ is the identity operator. If $H_1$ is finite-dimensional, this is the set of eigenvalues of $E$. If $E$ is a  complex-valued matrix, $E^T$ denotes its transpose and $E^H$ its conjugate transpose.

\paragraph{Our Hilbert space $Y_\I$} Our arguments will feature Hilbert spaces that are finite- or infinite-dimensional, constructed in the following manner. Choose $\I\subset \Z$, finite or infinite. For each $i\in \I$, let $X_i$ denote a complex Hilbert space and set 
\begin{equation} \label{eq:YI}
Y_\I := \bigoplus_{i\in \I}X_i,
\end{equation}
which (see, e.g., \cite[\S1.6]{ConwayFA}) is a Hilbert space with the norm $\|\cdot\|$ given by
$$
\|x\|^2 = \sum_{i\in \I} \|x_i\|^2,  \quad \mbox{for} \quad x=(x_i)_{i\in \I}\in Y_\I.
$$
 In the case $X_i=X$, for $i\in \I$ and some Hilbert space $X$,  we have $Y_\I = \ell^2(\I,X)$.  

\paragraph{Band operators}  Let $E\in L(Y_\I)$. Then (e.g., \cite{LiBook}) $E$ has a matrix representation $[E]=[e_{ij}]_{i,j\in \I}$, denoted again by $E$, with all $e_{ij}\in L(X_j,X_i)$, such that, for every $x=(x_j)_{j\in \I}\in Y_\I$ with finitely many non-zero entries, $Ex = y=(y_i)_{i\in \I}$, where 
\begin{equation} \label{eq:matmult}
y_i = \sum_{j\in \I} e_{ij}x_j, \quad i\in \I.
\end{equation}

Let $BO(Y_\I)$ denote the linear subspace of those  $E\in L(Y_\I)$ whose matrix representation is {\em banded} with some {\em band-width $w$}, meaning that $e_{ij}=0$ for $|i-j|>w$. If $E\in BO(Y_\I)$, then \eqref{eq:matmult} holds (with the sum finite) for all $x\in Y_\I$. Conversely, if $[E]=[e_{ij}]_{i,j\in \I}$ is a banded matrix of operators  $e_{ij}\in L(X_j,X_i)$ that satisfies 
\begin{equation} \label{eq:eijbound}
\sup_{i,j\in \I}\|e_{ij}\| < \infty,
\end{equation}
then the mapping $E:Y_\I\to Y_\I$, $x\mapsto y$, with $y$ given by \eqref{eq:matmult}, satisfies $E\in BO(Y_\I)$, and $[E]$ is the matrix representation of  $E$. As a special case of this observation the block matrix $A$ given by \eqref{eq:block} can be identified, in the usual way, with an operator in $L(Y_\I)$, denoted also by $A$, with   $\I=\{1,\ldots,N\}$ and $X_i = \C^{m_i\times 1}$, $i=1,\ldots,N$, and with the mapping $x\mapsto Ax$ given by
$$
(Ax)_i = \sum_{j=1}^N a_{ij} x_j, \quad \mbox{for} \quad x=(x_1,\ldots,x_N) \in Y_\I.
$$
Note that  $L(Y_\I)=BO(Y_\I)$ if $\I$ is finite.


\paragraph{The dual space $(Y_\I)'$ and adjoint operators} For $\I\subset\Z$ and $E\in L(Y_\I)$, we denote by $E^*\in L(Y_\I)$ and $E'\in L((Y_\I)')$ the Hilbert space and Banach space adjoints of $E$, respectively. We will identify the dual space $(Y_\I)'$ with $Y_{\I}$ through the mapping $Y_\I\to (Y_\I)'$, $x=(x_i)_{i\in \I}\mapsto \hat x$, where $\hat x(y) := (y,\bar x)$,
for $y\in Y_\I$, where $(\cdot,\cdot)$ is the inner product on $Y_\I$, $\bar x=(\bar x_i)_{i\in \Z}$, and, for $i\in \I$, $\bar x_i := J_ix_i$, where $J_i:X_i \to X_i$ is any fixed anti-linear isometric involution on $X_i$ (sometimes called a conjugate map, and easily constructed using an orthonormal basis for $X_i$, e.g., \cite[Conclusion 2.1.18]{SaSc:11}). In the case that each $X_i=\C^{n_i}=\C^{n_i\times 1}$, for some $n_i\in \N$, we choose $\bar x_i = J_i x_i$ to be simply the complex conjugate of $x_i$. If $[e_{ij}]_{i,j\in\I}$ is the matrix representation of $E$, then $[e^*_{ji}]_{i,j\in \I}$ is the matrix representation of $E^*$ and, with the above identification of $(Y_\I)'$ with $Y_\I$, $[e'_{ji}]_{i,j\in \I}$ is the matrix representation of $E'$. In particular,  if each $X_i=\C^{n_i}=\C^{n_i\times 1}$, for some $n_i\in \N$, so  $e_{ij}\in \C^{n_i\times n_j}$, for $i,j\in \I$, we have $e^*_{ij}=e^H_{ij}$ and $e'_{ij}=e_{ij}^T$, for $i,j\in \I$. Further, $E$ is a block matrix with complex-valued matrix blocks  and $E^*=E^H$, $E'=E^T$.

\paragraph{The lower norm, spectrum, and pseudospectrum}
For $\I\subset\Z$, a set $T\subset\I$ and $E\in L(Y_\I)$, put
\begin{equation} \label{eq:nuT}
\nu_T(E)\ :=\ \inf\{\|Ex\|:\|x\|=1,\ \supp(x)\subset T\},
\end{equation}
where, for $x=(x_i)_{i\in\I}$, $\supp(x):=\{i\in\I:x_i\ne 0\}$. Clearly, it holds that
\begin{equation}\label{eq:numon}
\nu_{T}(E)\ \ge\ \nu_{U}(E)
\quad\text{if}\quad
T\subset U\subset\I.
\end{equation}
Another basic result (e.g., \cite[Lemma 2.38]{LiBook}) is that, if $E,F\in L(Y_\I)$,
\begin{equation} \label{eq:Lipschitz}
|\nu_T(E)-\nu_T(F)|\ \le\ \|E-F\|, \quad \mbox{ for all } \quad T\subset \I.
\end{equation}
Note also that if  $Y_\I$ is finite-dimensional, so that $\{x\in Y_\I:\|x\|=1\}$ is compact, then the $\inf$ in \eqref{eq:nuT} can be replaced by $\min$.

 Abbreviate $\nu_\I(E)=:\nu(E)$. By analogy to the norm of $E$, i.e.\ to $\|E\|:= \sup\{\|Ex\|:\|x\|=1\}$, $\nu(E)$ is sometimes  (by abuse of notation) called the {\sl lower norm} 
 of $E$. Since $Y_\I$ is a Hilbert space, and where $s_{\rm max}(E)$ and $s_{\rm min}(E)$ denote the largest and smallest singular values\footnote{Recall that the singular values of a bounded linear operator $F$ on a Hilbert space $Y$ are the points in the spectrum of $(F^*F)^{1/2}$.} of $E$,
$$
\|E\|=s_{\rm max}(E) \quad \mbox{ and } \quad \nu(E)=s_{\rm \min}(E).
$$
Another key property of $\nu(\cdot)$ is that
\begin{equation} \label{eq:invnu}
\|E^{-1}\|^{-1}\ =\ \min\big(\,\nu(E),\nu(E')\,\big)\ =:\ \mu(E),
\end{equation}
where  $\|E^{-1}\|^{-1}:= 0$ if $E$ is not invertible. In particular, $E$ is invertible if and only if $\nu(E)$ and $\nu(E')$ are both nonzero, i.e., if and only if $\mu(E)\neq 0$, in which case
$\nu(E)=\nu(E')=\mu(E)$. Further, if $E$ is Fredholm of index zero, in particular if $Y_\I$ is finite-dimensional, then $\nu(E)=0$ if and only if $\nu(E')=0$, so that 
\begin{equation} \label{eq:munu}
\mu(E)=\nu(E)=s_{\min}(E).
\end{equation}
It follows from \eqref{eq:Lipschitz} (since $\mu(E)=\nu(E)$ if $\mu(E)\neq 0$) that, for $E,F\in L(Y_\I)$, also
\begin{equation} \label{eq:lnPro2}
|\mu(E)-\mu(F)|\leq \|E-F\|.
\end{equation}

We have recalled the definition of $\Spec E$ above. By \eqref{eq:invnu},
\[
\Spec E\ =\ \{\lambda\in\C\ :\ E-\lambda I \mbox{ is not invertible}\}\ =\ \{\lambda\in\C\ :\ \mu(E-\lambda I)=0\}.
\]
Generalising \eqref{eq:speceps}, $\Speps E$,   the {\em closed $\eps$-pseudospectrum of $E$}, is defined for $\eps\geq 0$ by
\begin{equation} \label{eq:Speps}
\Speps E := \{\lambda\in \C: \|(E-\lambda I)^{-1}\|^{-1}\leq \eps\} =\{\lambda\in\C\ :\ \mu(E-\lambda I)\le\eps\}, 
\end{equation}
by \eqref{eq:invnu}.  Note that $\Specn_0 E = \Spec E$. We will occasionally mention the  {\sl open $\eps$-pseudospectrum of $E$}, defined for $\eps>0$ by
\begin{equation} \label{eq:speps}
\speps E := \{\lambda\in \C: \|(E-\lambda I)^{-1}\|^{-1}< \eps\} =\{\lambda\in\C\ :\ \mu(E-\lambda I)<\eps\}.
\end{equation}
Using \eqref{eq:lnPro2} it is easy to see that the sets $\speps E$ and $\Speps E$ are open and closed, respectively. Indeed (see, e.g., \cite{Shargo08}), since $Y_\I$ is a Hilbert space and $E\in L(Y_\I)$,
\begin{equation} \label{eq:clop}
\Speps E = \overline{\speps E} \quad \mbox{and} \quad \speps E = {\rm int}(\Speps E),
\end{equation}
the interior of $\Speps E$. Suppose that $\I=\cup_{j\in J}\J_j$, with $\J_i\cap \J_j=\emptyset$, for $i\neq j$. In the case that $E\in L(Y_\I)$ is the direct sum $E=\oplus_{j\in J} E_j$, where $E_j\in L(Y_{\J_j})$ for $j\in J$,  a useful result, that can be seen, for example, via \eqref{eq:speps} and \eqref{eq:clop}, is that 
\begin{equation} \label{eq:ds2}
\Speps E = \overline{\bigcup_{j\in J} \Speps E_j}, \qquad \mbox{for} \quad \eps\geq 0.
\end{equation}

An informative characterisation of $\Speps E$ (which holds since $Y_\I$ is a Hilbert space; see \cite[Theorem 3.27]{HaRoSi2}) 
is that
\begin{equation} \label{eq:spepspert2}
\Speps E\ = \bigcup_{\|F\|\leq\eps} \Spec (E+F), \qquad \eps>0.
\end{equation}
It follows from \eqref{eq:spepspert2} that, for $\eps>0$,
$\Spec E + \eps \overline{\D}\subset \Speps E$, with  (see \cite[p.~247]{Davies2007:Book}) equality if $E$ is normal,
and that
\begin{equation} \label{eq:PseudInc2}
\Speps(B+T) \subset \Specn_{\eps+\delta} B \quad \mbox{if} \quad \|T\|\leq \delta.
\end{equation}

We will use the above definitions of $\Speps E$ and $\speps E$ also when $E$ is a square complex-valued matrix (cf.~\eqref{eq:speceps}, our norms then the matrix $2$-norm), and note that \eqref{eq:Speps}-\eqref{eq:PseudInc2} hold equally in that case. In particular, in the case that $E=\diag(E_1,E_2)$ is the direct sum of complex-valued square matrices $E_1$ and $E_2$, it follows from \eqref{eq:ds2} (cf.~\cite[Thm.~2.4]{TrefEmbBook}) that
\begin{equation} \label{eq:ds}
\Speps E = \Speps E_1 \cup \Speps E_2, \quad \mbox{for} \quad \eps\geq 0.
\end{equation}

In the case when $E$ is an $(n+2)\times n$ matrix, for some $n\in \N$ (with scalar, block, or operator-valued entries), we define the closed $\eps$-pseudospectrum of $E$, for $\eps\geq 0$, by \eqref{eq:recteps}. 
 Since $s_{\min}(E-\lambda I_n)=\nu(E-\lambda I_n)$,
\begin{equation} \label{eq:speps.rect}
\begin{array}{ccc}
\Speps E\ =\ \{\lambda\in\C\ :\ \nu(E-\lambda I_n)\le\eps\},&&\eps\ge 0.
\end{array}
\end{equation}
The example \eqref{eq:chara} below illustrates that 
$\Speps E$ can be empty when $E$ is rectangular, in contrast to the case when $E$ is square when each pseudospectrum contains a neighbourhood of the spectrum.

\section{Spectral inclusions for bi-infinite matrices} \label{sec:bi}

Our inclusion set results for finite matrices, Theorems \ref{thm:tau} and \ref{thm:tau1}, are derived from generalisations, which we present in this section,  of analogous inclusion sets for the spectra and pseudospectra of bi-infinite matrices from \cite{CW.Heng.ML:SpecIncl1}. 

Suppose that $\I=\Z$ and that $E\in L(Y_\I)$, where $Y_\I$ is given by \eqref{eq:YI}. It is enough for our purposes to suppose that $E$ has a matrix representation $[e_{ij}]_{i,j\in \Z}$ that is tridiagonal, in which case, since $E\in L(Y_I)$, the coefficients $e_{ij}$ satisfy \eqref{eq:eijbound}. Analogously to \eqref{eq:rdef}, let
\begin{equation}  \label{eq:rdef2}
\begin{aligned}
r_L(E) := \sup\{\|e_{i+1,i}\|:i\in \Z\}, &\quad r_U(E) := \sup\{\|e_{i,i+1}\|:i\in \Z\},\\
r(E) & := r_L(E)+r_U(E).
\end{aligned}
\end{equation}
For $k\in \Z$ and $n\in \N$, let
$$
E_{n,k}:=[e_{ij}]_{i,j=k+1}^{k+n},
$$ 
and (cf.~\eqref{eq:Bn+}) let
\begin{equation} \label{eq:En+}
\begin{aligned}
E^+_{n,k} &:= \begin{pmatrix}
e_{k,k+1} & 0&\cdots&0& 0\\\hline
& & E_{n,k}\\\hline
0& 0 &\cdots&0 & e_{k+n+1,k+n}
\end{pmatrix},\\ 
(E')^+_{n,k} &:= \begin{pmatrix}
e'_{k+1,k} & 0&\cdots&0& 0\\\hline
& & E'_{n,k}\\\hline
0& 0 &\cdots&0 & e'_{k+n,k+n+1}
\end{pmatrix}.
\end{aligned}
\end{equation}
For $w\in \R^n\setminus \{0\}$, and setting $w_0:=0$ and $w_{n+1} := 0$, let
\begin{equation*} 
S_{n}(w)\ :=\ \sum_{j=1}^{n}w_{j}^{2}, \quad T_{n}^{-}(w)\ :=\
\sum_{j=1}^{n}\left(w_{j-1}-w_{j}\right)^{2},
\end{equation*}
\[
T_{n}^{+}(w)\ :=\ \sum_{j=1}^{n}\left(w_{j+1}-w_{j}\right)^{2}, \quad T_{n}(w)\ :=\
w_{1}^{2}+w_{n}^{2}+\sum_{j=1}^{n-1}\left(w_{i+1}-w_{i}\right)^{2}.
\]
For $w\in \R^n\setminus \{0\}$ and tridiagonal $E\in L(Y_\I)$, let
\begin{equation} \label{eq:epsn1}
\eta_n(E,w)\ :=\
r_L(E)\sqrt{\frac{T_{n}^{-}(w)}{S_{n}(w)}}\ +\
r_U(E)\sqrt{\frac{T_{n}^{+}(w)}{S_{n}(w)}}\,,
\end{equation}
\begin{equation} \label{eq:epsn2}
\eta'_n(E,w)\ :=\ r(E)\sqrt{\displaystyle\frac{T_{n}(w)}{S_{n}(w)}}.
\end{equation}

The following is the main result for bi-infinite matrices that we need. The inequalities \eqref{eq:infon} and  \eqref{eq:infonG} are results regarding the approximation of the lower norm of a bi-infinite matrix by the lower norm of a finite submatrix. (The inequality \eqref{eq:infonG} provides an upper as well as a lower bound, the upper bound an immediate consequence of \eqref{eq:numon}.)  These inequalities generalise  results in   \cite{CW.Heng.ML:SpecIncl1} for the case that, for some Banach space $X$,  $X_i=X$, for $i\in \Z$, so that $Y_\Z=\ell^2(\Z,X)$, but the proofs apply to the case we need that $Y_\I$ has the form \eqref{eq:YI}, with each $X_i$ a Hilbert space. (The proofs of \eqref{eq:akbk1} and \eqref{eq:infon} are those of \cite[Prop.~3.3, Cor.~3.4]{CW.Heng.ML:SpecIncl1}, of \eqref{eq:akbk2} and \eqref{eq:infonG} are those of \cite[Prop.~5.1, Cor.~5.2]{CW.Heng.ML:SpecIncl1}.)

\begin{theorem}[{\bf Approximation of $\nu(E)$}] \label{cor:nubound}
Let $n\in\N$, $w=(w_1,\ldots,w_n)\in\R^n\setminus\{0\}$,   and  $x=(x_i)_{i\in \Z}\in Y_\Z\setminus\{0\}$, suppose $E\in L(Y_\Z)$ is tridiagonal, and let $x_{n,j} := (w_1 x_{j+1},\ldots, w_n x_{j+n})^T$, for $j\in \Z$. Then, for some $k,\ell,m\in \Z$, $x_{n,k}$, $x_{n,\ell}$, and $x_{n,m}$ are non-zero with
\begin{eqnarray} \label{eq:akbk1}
\frac{\|E_{n,k} x_{n,k}\|}{\|x_{n,k}\|} &\leq& \frac{\|Ex\|}{\|x\|}+\eta_{n}(E,w),\\  \label{eq:akbk2}
\frac{\|E^+_{n,\ell} x_{n,\ell}\|}{\|x_{n,\ell}\|} &\leq & \frac{\|Ex\|}{\|x\|}+\eta'_n(E,w).
\end{eqnarray}
As a consequence,
\begin{eqnarray} \label{eq:infon}
\inf_{k\in \Z} \nu(E_{n,k}) & \leq & \nu(E)+\eta_n(E,w),\\
 \label{eq:infonG}
\inf_{k\in \Z} \nu(E^+_{n,k}) & \leq & \nu(E)+\eta'_n(E,w)\ \leq\ \inf_{k\in \Z} \nu(E^+_{n,k}) +\eta'_n(E,w).
\end{eqnarray}
\end{theorem}

We think of the terms $\eta_n(E,w)$ and $\eta'_n(E,w)$ in \eqref{eq:infon} and \eqref{eq:infonG} as {\em penalty terms} accounting for the truncation of $E$ to $E_{n,k}$ and $E^+_{n,k}$, respectively. The following result from \cite{CW.Heng.ML:SpecIncl1} minimises these penalty terms as a function of  $w\in\R^n\setminus\{0\}$. The formulae \eqref{eq:wtauinf} and \eqref{eq:wtau1inf} are proved in  \cite[Corollaries 3.8 and 5.4]{CW.Heng.ML:SpecIncl1}. 

\begin{theorem}[{\bf Minimisation of the penalty terms}] \label{thm:inf} For each $n\in \N$ and each tridiagonal $E\in L(Y_Z)$, the infimum of each of $\eta_n(E,w)$ and $\eta'_n(E,w)$, as a function of $w\in\R^n\setminus\{0\}$, is achieved for some $w$ with $\|w\|=1$, and
\begin{equation} \label{eq:wtauinf}
\min_{w\in \R^n, \, \|w\|=1} \eta_n(E,w)\ \leq\ \eta_n(E)\ :=\ 2r(E) \sin(\theta_n(E)/2),
\end{equation}
with equality if $n=1$ or $r_L(E)r_U(E)=0$,
where $\theta_n(E)$ is the unique solution $t$ in the range
$\left[\frac{\pi}{2n+1}\,,\,\frac{\pi}{n+2}\right]$
of equation \eqref{eq:theta}, and
\begin{eqnarray}  \label{eq:wtau1inf}
\min_{w\in \R^n, \, \|w\|=1} \eta'_n(E,w) &=& \eta'_n(E)\ :=\ 2r(E) \sin(\pi/(2n+2)).
\end{eqnarray}
\end{theorem}

Inclusion set results, analogous to Theorems \ref{thm:tau} and \ref{thm:tau1}, follow by application of the above results.  Theorem \ref{thm:taubi} below, that follows from the above theorems and that we use to prove our main results in \S\ref{sec:main}, is a variant on results in \cite{CW.Heng.ML:SpecIncl1} for the case that $Y_\Z=\ell^2(\Z,X)$, for some Banach space $X$. To state Theorem \ref{thm:taubi} we introduce additional notation. For tridiagonal $E\in L(Y_\Z)$ and $n\in \N$, let
\begin{eqnarray} \label{eq:mudag}
\mu_n(E) & :=& \inf_{k\in \Z} \, \mu(E_{n,k})  = \inf_{k\in \Z} \, \min\Big(\nu\left(E_{n,k}\right),\,\nu\left(E'_{n,k}\right)\Big),\\ \label{eq:mudag2}
\mu^+_n(E)&  := & \inf_{k\in \Z} \min\Big(\nu\left(E^+_{n,k}\right),\,\nu\left((E')^+_{n,k}\right)\Big),
\end{eqnarray}
and, for $\eps\geq 0$ (cf.~\eqref{eq:Sigdef} and \eqref{eq:Gamdef}), let
\begin{eqnarray} \label{eq:sigmaAAlt}
\Sigma^n_{\eps}(E) & := &  \{\lambda\in \C:\mu_n(E-\lambda I)\leq \eps+\eta_n(E)\},\\ \label{eq:gamdef}
\Gamma^n_{\eps}(E) &:=  &\left\{\lambda\in \C : \mu^+_n(E-\lambda I)\leq \eps+\eta'_n(E)\right\},
\end{eqnarray} 
 where $\eta_n(E)$ and $\eta'_n(E)$ are as defined in Theorem \ref{thm:inf}. The following theorem is a bi-infinite version of Theorems \ref{thm:tau} and \ref{thm:tau1}. We omit the proof which is identical to the  proofs of this result (in \cite[Thm.~3.5, Cor.~3.8]{CW.Heng.ML:SpecIncl1} for \eqref{eq:sigmaAAlt} and  \cite[Thm.~5.3, Cor.~5.4]{CW.Heng.ML:SpecIncl1} for \eqref{eq:gamdef}) for the special case that $X_i=X$, $i\in \Z$.

\begin{theorem}[\bf The $\tau$ and $\tau_1$ method inclusion sets: bi-infinite matrices] \label{thm:taubi} Suppose $E\in$
$L(Y_\Z)$ is tridiagonal. For $\eps\geq 0$ and $n\in \N$,
$$
\Speps E\ \subset\ \Sigma^n_\eps(E) \qquad \mbox{and} \qquad \Speps E\ \subset\ \Gamma^n_\eps(E) \subset\ \Specn_{\eps+\eta'_n(E)}.
$$
\end{theorem}

\section{The proofs of our main results} \label{sec:main}
We turn now to the proofs of our main results, the inclusion set results for finite matrices, Theorems \ref{thm:tau} and \ref{thm:tau1}. Throughout this section $A$ is the finite complex-valued matrix with block representation \eqref{eq:block}.

\begin{theorem} \label{thm:main} Suppose that $A$ is block-tridiagonal. Then
\begin{eqnarray} \label{eq:nu1}
\min_{0\leq k\leq N-n} \nu(A^+_{n,k}) &\leq & \nu(A)+\eps'_{n}(A)\ \leq\  \min_{0\leq k\leq N-n} \nu(A^+_{n,k}) + \eps'_{n}(A), \hspace*{4ex}
\end{eqnarray}
for $n=1,\ldots,N-1$, and
\begin{eqnarray} \label{eq:nu2}
\min_{0\leq k\leq N-n} \nu(A_{n,k}) &\leq & \nu(A)+\eps_{n-2}(A), \qquad \mbox{for} \quad n>2.
\end{eqnarray} 
\end{theorem}
\begin{proof}
Let $X_i = \C^{m_i\times 1}$,  $i=1,\ldots,N$, so that $A$ given by \eqref{eq:block} is the matrix representation of a linear operator on $Y_\I$ with $\I=\{1,\ldots,N\}$. Extend the definition of $X_i$ to $i\in \Z$ by periodicity, i.e.\ so that $X_{i+N}=X_i$, $i\in \Z$.  Define $E=[e_{ij}]_{i,j\in \Z}$ so that $e_{ij}:=a_{ij}$, for $1\leq i,j\leq N$,  and $e_{ij}:=0$, otherwise (see Figure \ref{fig:cases}). 
Since $A$ is tridiagonal, $E\in L(Y_\Z)$ is tridiagonal. 
Note that, for $n\in \N$,
\begin{equation} \label{eq:epseta}
\eps_n(A)=\eta_n(E), \quad \mbox{and} \quad \eps'_n(A)=\eta'_n(E).  
\end{equation}

Suppose now that $n\in \I$. Since $Y_\I$ is finite-dimensional we can choose $\tilde x\in Y_\I$ with $\|\tilde x\|=1$ such that $\nu(A) = \|A\tilde x\|$. Extend $\tilde x$ by zeros to be an element $x=(x_i)_{i\in \Z}$ of  $Y_\Z$. Then $x_i=\tilde x_i$, for $i\in \I$, $x_i=0$, otherwise, $\|x\|=1$, and $\|Ex\|=\|A\tilde x\|=\nu(A)$.  For $k\in \Z$ let $\widehat E_{n,k}$ denote $E_{n,k}$ or $E_{n,k}^+$, and let $\widehat A_{n,k}$ denote $A_{n,k}$ or $A_{n,k}^+$, respectively (cf.~Figure 1.2), for $0\leq k\leq N-n$. For any $w\in \R^n\setminus\{0\}$, and where $x_{n,j} := (w_1 x_{j+1},\ldots, w_n x_{j+n})^T$ for $j\in \Z$, it follows by Theorem \ref{cor:nubound} that there exists  $k\in \Z$ such that $x_{n,k}\neq 0$ and
\begin{equation} \label{eq:akbkg}
\nu(\widehat E_{n,k}) \leq \frac{\|\widehat E_{n,k} x_{n,k}\|}{\|x_{n,k}\|}  \leq \nu(A)+\hat\eta_{n}(E,w),
\end{equation}
where $\hat\eta_n(E,w)$ denotes $\eta_n(E,w)$ or $\eta'_n(E,w)$, in the respective cases   $\widehat E_{n,k}=E_{n,k}$ or $E_{n,k}^+$. Since $x_{n,k}=0$ if $k+n<1$ or $k+1>N$, it follows that $-n< k< N$. To complete the argument we consider the cases $-n< k<0$, $0\leq k \leq N-n$, and $N-n<k<N$ separately (see Figure \ref{fig:cases}).

\begin{figure}[t]
\[
\begin{tabular}{cp{3mm}cp{3mm}c}
$
\begin{pmatrix}~
\begin{tikzpicture}[scale=0.4]
\fill[gray!20] (1.5,4.5) rectangle ++(3,-3);
\foreach \k in {1.5,2,...,3.5}
  \fill[gray!60] (\k,6-\k) rectangle ++(1.0,-1.0);
\draw[gray!40,very thin,step=0.5cm] (0,0) grid (6,6);
\draw[black!30,line width=0.5mm] (1.5,4.5) rectangle ++(3,-3);
\end{tikzpicture}
~\end{pmatrix}
$
&&
$
\begin{pmatrix}~
\begin{tikzpicture}[scale=0.4]
\fill[gray!20] (1.5,4.5) rectangle ++(3,-3);
\fill[gray!30] (3.5,2.5) rectangle ++(2,-2);
\foreach \k in {1.5,2,...,3.5}
  \fill[gray!60] (\k,6-\k) rectangle ++(1.0,-1.0);
\draw[gray!40,very thin,step=0.5cm] (0,0) grid (6,6);
\draw[black!30,line width=0.5mm] (1.5,4.5) rectangle ++(3,-3);
\draw[black,line width=0.7mm] (3.5,2.5) rectangle ++(2,-2);
\end{tikzpicture}
~\end{pmatrix}
$
&&
$
\begin{pmatrix}~
\begin{tikzpicture}[scale=0.4]
\fill[gray!20] (1.5,4.5) rectangle ++(3,-3);
\fill[gray!30] (0.5,5.5) rectangle ++(2,-2);
\foreach \k in {1.5,2,...,3.5}
  \fill[gray!60] (\k,6-\k) rectangle ++(1.0,-1.0);
\draw[gray!40,very thin,step=0.5cm] (0,0) grid (6,6);
\draw[black!30,line width=0.5mm] (1.5,4.5) rectangle ++(3,-3);
\draw[black,line width=0.7mm] (0.5,5.5) rectangle ++(2,-2);
\end{tikzpicture}
~\end{pmatrix}
$
\\
$A$ and its extension $E$ &&
Case 2: $N-n<k<N$ && Case 3: $-n<k<0$
\end{tabular}
\]

\vspace{-3ex}

\caption{The bi-infinite matrix $E\in L(Y_\Z)$ that is the extension of $A\in L(Y_\I)$, and the matrices $E_{n,k}$, outlined in bold, in Cases 2 and 3.}\label{fig:cases}
\end{figure}

{\em Case 1: $0\leq k\leq N-n$.} If $k$ lies in this range then  $\widehat E_{n,k}=\widehat A_{n,k}$, so that \eqref{eq:akbkg} implies that 
\begin{equation} \label{eq:wideA}
\nu(\widehat A_{n,k})\leq \nu(A)+\hat \eta_n(E,w). 
\end{equation}

{\em Case 2: $N-n<k<N$.} If $k$ lies in this range put 
$$
x^-_{n,k} := (w_1 x_{k+1},\ldots, w_{N-k} x_{N})^T\in \bigoplus_{i=k+1}^N X_{i}
$$
and
$$
\hat x_{n,k} := \left(\begin{array}{c}0_{n+k-N}\\x^-_{n,k}\end{array}\right) \in \bigoplus_{i=N+1-n}^N X_{i},
$$
where $0_{\ell}$ denotes an appropriate null vector of length $\ell$, and put 
$$
\alpha := a_{k,k+1} \quad and \quad  \beta :=w_1\alpha x_{k+1}.
$$
Then $\|x_{n,k}\|=\|x^-_{n,k}\|= \|\hat x_{n,k}\|$ and, in the cases $\widehat E_{n,k}=E_{n,k}$ and $E_{n,k}^+$ we have
\begin{eqnarray*}
\widehat E_{n,k} x_{n,k} &=& \left(\begin{array}{c}A_{N-k,N+1-n} \,x^-_{n,k}\\0_{n+k-N}\end{array}\right) \mbox{\; and \;}  \left(\begin{array}{c}A^+_{N-k,N+1-n} \,x^-_{n,k}\\0_{n+k-N}\end{array}\right),
\end{eqnarray*}
respectively, and
\begin{eqnarray*}
\widehat A_{n,N-n} \,\hat x_{n,k} &=& \left(\begin{array}{c} 0_{n+k-N-1}\\  w_1 \alpha x_{k+1}\\ A_{N-k,N+1-n} \,x^-_{n,k}\end{array}\right) \mbox{\; and \;}
\left(\begin{array}{c} 0_{n+k-N}\\A^+_{N-k,N+1-n} \,x^-_{n,k}\end{array}\right),
\end{eqnarray*}
respectively. Thus, provided that $w_1=0$  in the case that $\widehat E_{n,k}=E_{n,k}$ and $\widehat A_{n,k}=A_{n,k}$ (which implies that $n>1$), we have that $\|\widehat E_{n,k} x_{n,k}\|=\|\widehat A_{n,N-n} \, \hat x_{n,k}\|$, so that \eqref{eq:akbkg} implies that \eqref{eq:wideA} holds for $k=N-n$.

{\em Case 3: $-n<k<0$.} If $k$ lies in this range a similar argument to that of Case 2 shows 
 that   \eqref{eq:wideA} holds for $k=0$, provided that $w_n=0$  in the case that $\widehat E_{n,k}=E_{n,k}$ and $\widehat A_{n,k}=A_{n,k}$.

Thus, in every case, \eqref{eq:wideA} holds for some $k$ in the range $0\leq k\leq N-n$, provided that $w_1=w_n=0$  in the case that $\widehat A_{n,k}=A_{n,k}$ (which implies that $n>2$). Thus
$$
\min_{0\leq k\leq N-n} \nu(\widehat A_{n,k})\leq \nu(A)+\hat \eta_n(E,w). 
$$
 Taking the infimum of $\hat \eta_n(E,w)$ in the above equation over all $w\in \R^n\setminus\{0\}$ (over all $w$ with $w_1=w_n=0$ in the case that $n>2$ and  $\widehat A_{n,k}=A_{n,k}$, noting that this infimum is the same as the infimum of $\eta_{n-2}(E,\tilde w)$ over all $\tilde w\in \R^{n-2}\setminus\{0\}$), we see that \eqref{eq:nu1} and \eqref{eq:nu2} follow from  Theorem \ref{thm:inf} and \eqref{eq:epseta} and, in the case of \eqref{eq:nu1}, by application of \eqref{eq:numon}.
\end{proof}

We come now to the proofs of our main results, establishing inclusion sets for $\Spec A$ and $\Speps A$. Many authors, for example Davies \cite{Davies2007:Book} and Trefethen \& Embree \cite{TrefEmbBook}, prefer to work with open pseudospectra, defined by \eqref{eq:speps}, rather than closed pseudospectra given by \eqref{eq:Speps}. We can derive $\tau$-method inclusion sets for the open pseudospectra $\speps A$ by combining  Theorem  \ref{thm:tau} with \eqref{eq:clop}.

\begin{proof}[{\bf Proof of Theorem  \ref{thm:tau1}}.] Note that it is enough to consider the case that $A$ is block-tridiagonal. For if $n\in\{1,\ldots,N-1\}$ 
and $\Speps A \subset \Gamma^{n}_{\eps}(A)\subset \Specn_{\eps+\eps'_n(A)}A$  whenever $A$ is block-tridiagonal and $\eps\geq 0$, then, in the case that $A$ is not block-tridiagonal, $\Speps B \subset    \Gamma^{n}_{\eps}(B)\subset \Specn_{\eps+\eps'_n(A)}B$, where $B$ is the tridiagonal part of $A$, so that, by \eqref{eq:PseudInc2}, 
$$
\Speps A \subset \Specn_{\eps+\|C\|}(B) \subset \Gamma^{n}_{\eps+\|C\|}(B)=\Gamma^{n}_\eps(A),
$$ 
and also
$$
\Gamma^{n}_\eps(A)= \Gamma^{n}_{\eps+\|C\|}(B)\subset \Specn_{\eps+\|C\|+\eps'_n(A)}B \subset \Specn_{\eps+2\|C\|+\eps'_n(A)}A.
$$

So suppose that $A$ is block-tridiagonal, so that $A$ coincides with $B$, its tridiagonal part, and its remaining part is $C=0$. Let $1\leq n<N$ and  $\eps\geq 0$, and
 suppose that $\lambda\in \Speps A$. Then, by \eqref{eq:Speps} and \eqref{eq:munu}, $\nu(A-\lambda I)\leq \eps$. But this implies, by Theorem \ref{thm:main}, that, for some $k\in \{0,\ldots,N-n\}$,  $\nu(A^{+}_{n,k}-\lambda I) \leq \eps+\eps'_{n}(A)$,  so that, by \eqref{eq:speps.rect},  $\lambda\in \Specn_{\eps+\eps'_{n}(A)} A^{+}_{n,k} \subset \Gamma^{n}_\eps(A)$. On the other hand, if $\lambda\in \Gamma^{n}_\eps(A)$, then 
 $\lambda\in \Specn_{\eps+\eps'_{n}(A)} A^{+}_{n,k}$, for some $k\in \{0,\ldots,N-n\}$, so that $\nu(A^+_{n,k}-\lambda I) \leq \eps+\eps'_{n}(A)$, so that, by \eqref{eq:numon}, $\nu(A-\lambda I)\leq \eps+\eps'_{n}(A)$, so that $\lambda \in \Specn_{\eps+\eps'_{n}(A)}(A)$.
\end{proof}

\begin{proof}[{\bf Proof of Theorem  \ref{thm:tau}.}] As in the above proof it is enough to consider the case that $A$ is block-tridiagonal. 

Let $X_i = \C^{m_i\times 1}$, for $i=1,\ldots,N$, so that $A$ given by \eqref{eq:block} is the matrix representation of a linear operator on $Y_\I$ with $\I=\{1,\ldots,N\}$. Extend the definition of $X_i$ to $i\in \Z$ by periodicity, i.e.\ so that $X_{i+N}=X_i$, $i\in \Z$. Suppose that $A$ is block-tridiagonal, so that $A$ coincides with $B$, its tridiagonal part, and its remaining part is $C=0$. Let $1\leq n< N$ and $\eps\geq 0$. To show that $\Speps A \subset \Sigma^n_\eps(A)$, we need to show that $\Speps A \subset \sigma^n_\eps(A)$ and that, for $n>2$, $\Speps A \subset \hat \sigma^n_\eps(A)$.

To see that $\Speps A \subset \sigma^n_\eps(A)$, define $E=[e_{ij}]_{i,j\in \Z}=\diag(\ldots,A,A,\ldots)$, with $e_{1,1}=a_{1,1}$. Note that $E\in L(Y_\Z)$, with $Y_\Z$ given by \eqref{eq:YI}, $E$ is tridiagonal, and $\Speps E = \Speps A$ by \eqref{eq:ds2}. Thus, for $\eps\geq 0$ and $n=1,\ldots,N$, $\Speps A = \Speps E \subset \Sigma^n_\eps(E)$ by Theorem \ref{thm:taubi}. With our definition of $E$, $E_{n,N+k}=E_{n,k}$, for $k\in \Z$. Note also that  $\eta_n(E)=\eps_n(A)$, since $\|C\|=0$. Thus 
$$
\Sigma^n_\eps(E) = \{\lambda\in \C:\mu_n(E-\lambda I)\leq \eps+\eps_n(A)\}, \quad \mbox{where} \quad \mu_n(E)  := \min_{0\leq k\leq N-1} \, \mu(E_{n,k}),
$$
so that, by \eqref{eq:Speps},
$$
\Sigma^n_\eps(E) = \bigcup_{k=0}^{N-1} \Specn_{\eps+\eps_n(A)} E_{n,k}.
$$
Now $E_{n,k}=A_{n,k}$, for $k=0,\ldots,N-n$, Further, if $n>1$, then, for $k=N-n+1,\ldots,N-1$, $E_{n,k} = \diag(A_{N-k,k},A_{k+n-N,0})$, so that
$$
\Specn_{\eps+\eps_n(A)} E_{n,k} = \Specn_{\eps+\eps_n(A)} A_{N-k,k} \cup  \Specn_{\eps+\eps_n(A)}A_{k+n-N,0},
$$
by \eqref{eq:ds}, so that $\Sigma^n_\eps(E) = \sigma^n_{\eps}(A)$. Thus $\Speps A = \Speps E \subset \sigma^n_\eps(E)$.

It follows that also  $\Speps A \subset \hat \sigma^n_\eps(A)$, if $N> n\geq 3$, by arguing by application of Theorem \ref{thm:main} (i.e., \eqref{eq:nu2})  as in the last part of the proof of Theorem \ref{thm:tau1}.
\end{proof}

\section{Examples} \label{sec:examples}

In this section we illustrate the inclusion set bounds that we have proposed in Theorems \ref{thm:tau} and \ref{thm:tau1}, with some  focus on cases where the matrix $A$ is banded, so that Remark \ref{rem:band} applies.

\subsection{The discrete Laplacian and Jordan block} \label{sec:disLap}

We start with two simple tridiagonal examples, one real symmetric ($A=L_M$), the other non-normal ($A=V_M$), both Toeplitz so that we know the spectrum explicitly. Here $L_M$ and $V_M$, both order $M$, are  a discrete Laplacian and a Jordan block, respectively. Precisely, $L_1=V_1=0$ and, for $M>1$,
$$
L_M :=\ \left(\begin{array}{ccccc}0&1\\
1&0&1\\
&\ddots&\ddots&\ddots\\
&&1&0&1\\
&&&1&0
\end{array}\right)_{M\times M}  
\quad \mbox{and} \quad
V_M :=\ \left(\begin{array}{cccc}0&1\\
&\ddots&\ddots\\
&&0&1\\
&&&0
\end{array}\right)_{M\times M} .
$$
For $M\in \N$, $\Spec V_M=\{0\}$ and (e.g., \cite{BoeGru}) 
\begin{equation} \label{eq:DLS}
\Spec L_M= \{2\cos(j\pi/(M+1)):j\in \{1,\ldots,M\}\}.
\end{equation}

\subsubsection{The Jordan block} \label{sec:jordan}
Let's write down the $\tau$ and $\tau_1$ inclusion sets of \S\ref{sec:intro} for the Jordan block case  $A=V_M$, choosing for simplicity $N=M$, so that the block structure \eqref{eq:block} is trivial, each block a single complex number. These inclusion sets are illustrated, for particular values of $n$ and $\eps$, in Figure \ref{fig:jordan}.
In the case $A=V_N$ Theorem \ref{thm:tau} tells us that, for $\eps\geq 0$ and $n<N$,
\begin{equation} \label{eq:SpepsVN}
\Speps V_N\subset \Sigma_\eps^n(V_N) = \left\{\begin{array}{ll}\bigcup_{m=1}^n  \Specn_{\eps+ \eps_{n}}V_m, & n=1,2,\\
\Specn_{\eps+ \eps_{n-2}} V_n \cap \bigcup_{m=1}^n  \Specn_{\eps+ \eps_{n}}V_m, & n>2,\end{array}\right. 
\end{equation}
where
$$
\eps_n := \eps_n(V_N) = 2\sin(\pi/(4n+2)).
$$

Let $V_1^+ := (1,0,0)^T$, $\widetilde V_1^+ := (0,0,0)^T$, and, for $n>1$, let
$$
V^+_{n} := \begin{pmatrix}
1 & 0&\cdots&0\\\hline
& & V_{n}\\\hline
0& 0 &\cdots&0
\end{pmatrix}_{(n+2)\times n}, \quad \widetilde V^+_{n} := \begin{pmatrix}
0&\cdots&0\\\hline
 & V_{n}\\\hline
 0 &\cdots&0
\end{pmatrix}_{(n+2)\times n}.
$$
Then $\{B^{+}_{n,k}: 0\leq k\leq N-n\} = \{V^+_n,\widetilde V^+_n\}$ so that,
by Theorem \ref{thm:tau1},
\begin{equation} \label{SpepsVN2}
\Speps V_N \subset \Gamma^n_\eps(V_N) = \Specn_{\eps+\eps'_n} V^+_n \cup \Specn_{\eps+\eps'_n} \widetilde V^{+}_n \subset \Specn_{\eps+\eps'_n} V_N, 
\end{equation}
for $\eps\geq 0$ and $1\leq n< N$, where
$$
\eps'_n := \eps'_n(V_N)  = 2\sin(\pi/(2n+2)). 
$$

\begin{figure}[t]
\[
\begin{tabular}{cc}
\includegraphics[width=57mm]{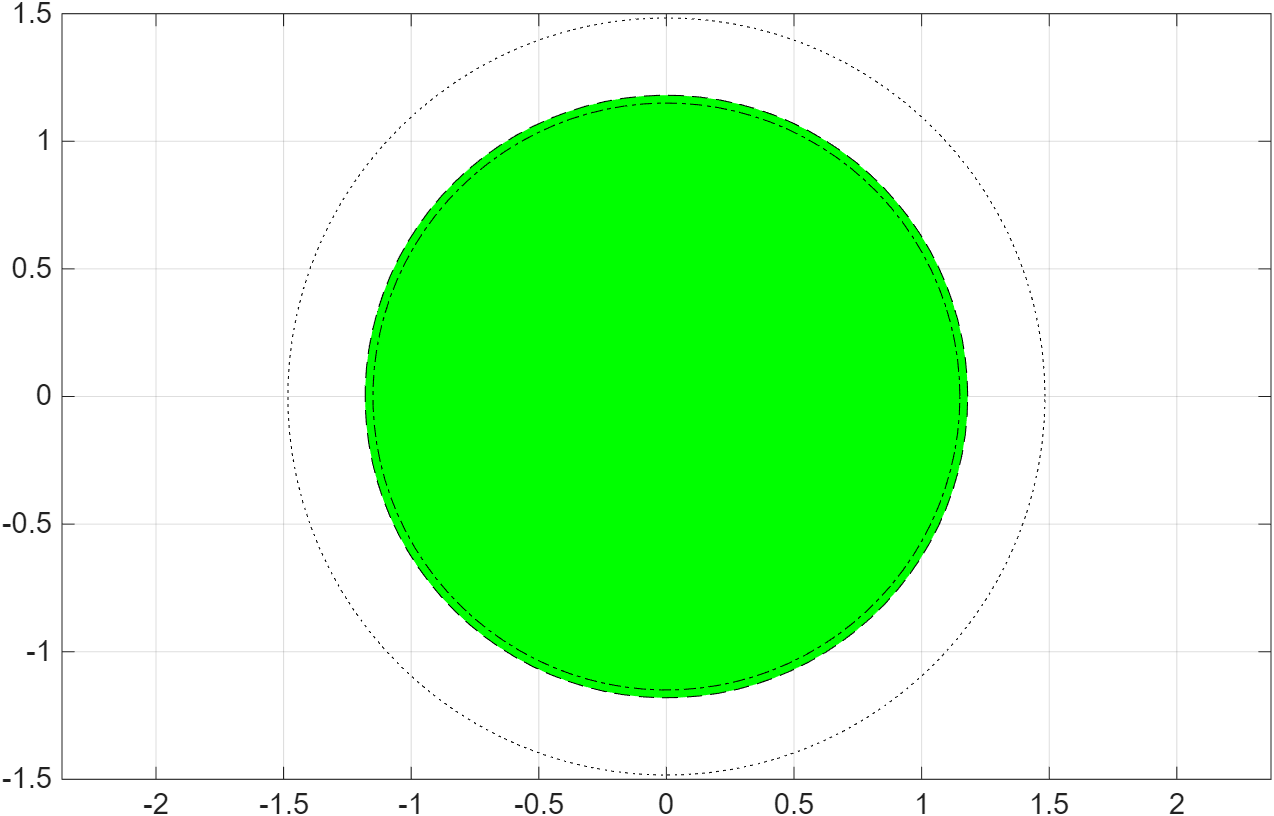} &
\includegraphics[width=57mm]{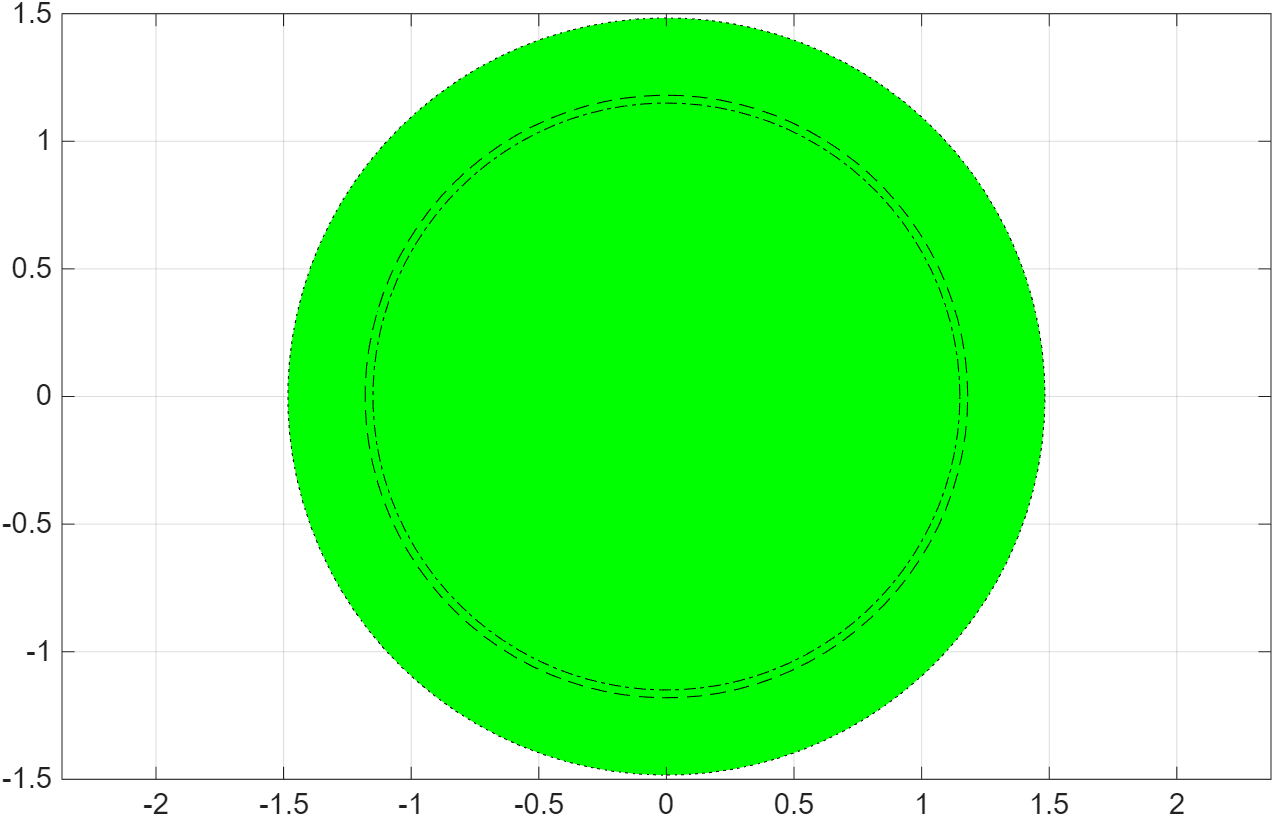}\\
$\Sigma^n_{\eps}(V_N)$ & $\Gamma^n_{\eps}(V_N)$
\end{tabular}
\]

\vspace{-2ex}

\caption{The $\tau$ and $\tau_1$ inclusion sets, $\Sigma^n_{\eps}(V_N)$ and $\Gamma^n_{\eps}(V_N)$, for $n=4$ and $\eps=0.15$.
Each is an inclusion set for $\Speps V_N$ if $N\geq n$. Also shown in each panel are circles of radii $1+\eps=1.15$ ($- .$), $\alpha_n(\eps_n+\eps)\approx 1.18$ ($-\,-$), and $\alpha_n(\eps'_n+\eps)\approx 1.48$ ($\cdots$). The circles of radii $\alpha_n(\eps_n+\eps)$ and $\alpha_n(\eps'_n+\eps)$ are the boundaries of $\Sigma^n_{\eps}(V_N)$ and $\Gamma^n_{\eps}(V_N)$, respectively, by \eqref{eq:inclus2}, \eqref{eq:inclus}, and \eqref{eq:PiCh}. In the supplementary materials we provide Matlab codes to reproduce this figure and Figures \ref{fig:penta} and \ref{fig:Toep2} below, and to produce similar figures for other values of the parameters $n$ and $\eps$.} \label{fig:jordan}
\end{figure}

Our next proposition characterises the inclusion sets \eqref{eq:SpepsVN} and \eqref{SpepsVN2} more precisely, and uses the notations $v_n$ and $v_n^+$, where
$$
v_n(s) := s_{\min}(V_n-sI_n), \qquad v^+_n(s) := s_{\min}(V^+_n-sI^+_n), \qquad \mbox{for } s\geq 0, \quad n\in \N.
$$
It is shown in \cite[\S8.1]{CW.Heng.ML:SpecIncl1} that
$$
s_{\min}(V_n-\lambda I_n) = v_n(|\lambda|) \quad \mbox{and} \quad s_{\min}(V^+_n-\lambda I^+_n) = v^+_n(|\lambda|), \qquad \lambda\in \C,
$$ 
that
$v_n$ is continuous and strictly monotonically increasing on $[0,\infty)$, with $v_n(0)=0$, $v_n(1)=\eps_n$, and $v_n(s)\to+\infty$ as $s\to +\infty$, and that
\begin{equation} \label{eq:vn+}
v_n^+(s) = \sqrt{1+s^2-2s c_n}, \quad s\geq 0, \; \mbox{where} \; c_n := \cos(\pi/(n+1))=1-(\eps'_n)^2/2.
\end{equation}
Noting that $v_n^+$ has a unique minimum on $(0,\infty)$ at $s=c_n$, with  $v_n^+(c_n) = \sin(\pi/(n+1))$, it follows that, for $\eps\geq 0$, 
\begin{equation} \label{eq:chara}
\begin{aligned}
\Speps V^+_n &= \left\{\begin{array}{ll}\emptyset, & 0\leq \eps < \sin(\pi/(n+1)),\\
\{\lambda\in \C:a^-_n(\eps) \leq |\lambda| \leq a^+_n(\eps)\}, & \eps \geq \sin(\pi/(n+1)), \end{array}\right.\\ 
\Speps V_n &= \alpha_n(\eps) \overline{\D},
\end{aligned}
\end{equation}
where 
$$
a^\pm_n(\eps) := \cos(\pi/(n+1))\pm \sqrt{\eps^2-\sin^2(\pi/(n+1))},
$$
are the solutions of $v_n^+(s)=\eps$, and $\alpha_n(\eps)$ is the unique solution of $v_n(s)=\eps$. We will show in the proof of the following proposition that, analogously to \eqref{eq:vn+},
 \begin{equation} \label{eq:vn}
v_n(s) = \sqrt{1+s^2-2s \cos(\phi_n(s))}, \quad s\geq 1, 
\end{equation}
where $\phi_n(s)$ is the unique solution $t$ in the range $[\pi/(2n+1),\pi/(n+1))$ of the equation
\begin{equation} \label{eq:phi}
s\sin((n+1)t)=\sin(nt).
\end{equation}

\begin{proposition} \label{prop:VN}
For $n\in \N$ and $\eps\geq 0$,
\begin{equation} \label{eq:PiCh}
\Specn_{\eps} \widetilde V^{+}_n = \Specn_{\eps}  V_n  = \alpha_n(\eps)\overline{\D}.
\end{equation}
Further, for $n\in \N$, $\alpha_n(\eps) <\alpha_n(\eps')$, for $0\leq \eps<\eps'$, 
\begin{equation} \label{eq:maxan}
\max\{a_n^+(\eps_n+\eps),\alpha_j(\eps_n+\eps)\} <  \alpha_n(\eps_n+\eps),  \quad \mbox{for}\quad \eps\geq 0 \;\mbox{ and }\; 1\leq j<n,
\end{equation}
and 
\begin{equation} \label{eq:bounds}
1+\eps\leq \alpha_n(\eps_n+\eps) \leq 1+\eps+\min\{\eps_n,\sqrt{2\eps_n\eps}\}, \quad \mbox{for}\quad \eps\geq 0.
\end{equation}
Thus, for $N>1$, $n=1,\ldots,N-1$, and $\eps\geq0$,
\begin{eqnarray} \nonumber
(1+\eps)\overline{\D} &\subset &\Sigma_\eps^n(V_N) = \Specn_{\eps+\eps_n} V_n  = \alpha_n(\eps+\eps_n)\overline{\D}\\  \label{eq:inclus2}
&\subset &\left(1+\eps+\min\{\eps_n,\sqrt{2\eps_n\eps}\}\right) \overline{\D}, 
\end{eqnarray}
in particular $\Sigma_0^n(V_N)=\overline{\D}$, and
\begin{equation} \label{eq:inclus}
\Sigma_\eps^n(V_N) \subsetneq \Gamma_\eps^n(V_N) =  \Specn_{\eps+\eps'_n} V_n  = \alpha_n(\eps+\eps'_n)\overline{\D}.
\end{equation}
\end{proposition}

We prove this proposition below. But let us first make two observations:
\begin{enumerate}
\item By \eqref{eq:inclus2} and \eqref{eq:inclus}, the $\tau$-method inclusions sets $\Sigma_\eps^n(V_N)$ are, for this example, sharper than the $\tau_1$-method inclusion sets $\Gamma_\eps^n(V_N)$. Both are closed discs centred on the origin, the $\tau_1$ discs having  larger radius, since $\eps_n'>\eps_n$ so that $\alpha_n(\eps+\eps'_n)>\alpha_n(\eps+\eps_n)$.  See Figure \ref{fig:jordan}, which provides an illustration of \eqref{eq:inclus2} and \eqref{eq:inclus} for $n=4$ and $\eps=0.15$.
\item \label{point3} The $\tau$-method inclusion set for $\Spec V_N$ is $\Sigma_0^n(V_N)=\overline{\D}$, the same inclusion set as provided by the Gerschgorin theorem \eqref{eq:ger}.
This is a poor approximation to $\Spec_0 V_N = \Spec V_N=\{0\}$. On the other hand, for every $\eps>0$,  $\Sigma_\eps^n(V_N)$ and $\Gamma_\eps^n(V_N)$ are both good approximations to $\Speps V_N$ if $N$ and $n$ are sufficiently large, in the sense that, where $S_n$ denotes  $\Sigma_\eps^n(V_N)$ or $\Gamma_\eps^n(V_N)$,  the Hausdorff distance between $S_n$ and $\Speps V_N$ (see \eqref{eq:HD} below) satisfies $d_H(S_n,\Speps V_N)\to 0$ as $n,N\to \infty$ with $n< N$.
\end{enumerate}
 
The last sentence of \ref{point3} is the content of Corollary \ref{cor:convVN},  
a special case of general results
for 
Toeplitz matrices, Theorems \ref{thm:Tgen} and \ref{thm:final} below. This corollary (and the later theorems) reference the standard  Hausdorff distance between compact subsets $T_1,T_2 \subset \C$, defined by
\begin{equation} \label{eq:HD}
d_H(T_1,T_2)\ :=\ \max\left\{ \sup_{t_1\in T_1} \dist(t_1,T_2)\,,\, \sup_{t_2\in T_2}\dist(t_2,T_1)\right\}.
\end{equation} 
For a sequence $(T_n)_{n\in \N}$ of compact subsets of $\C$ and a compact $T\subset \C$ we will write $T_n\Hto T$ as $n\to\infty$ if $d_H(T_n,T)\to 0$ as $n\to\infty$. We recall (e.g., \cite{HaRoSi2})  that $d_H(\cdot,\cdot)$ is a metric on the set of compact subsets of $\C$.

\begin{corollary} \label{cor:convVN}
Suppose that $\eps>0$. As $N\to\infty$, $\Speps V_N\Hto (1+\eps)\overline{\D}$. Further, where $S_n$ denotes  $\Sigma_\eps^n(V_N)$ or $\Gamma_\eps^n(V_N)$,  $S_n\Hto (1+\eps)\overline{\D}$, as $n,N\to\infty$ with $n< N$, so that $d_H(S_n,\Speps V_N)\to 0$ as $n,N\to\infty$ with $n<N$. 
\end{corollary}
\begin{proof}
By \eqref{eq:inclus2}, if $N$ is large enough so that $\eps_N<\eps$, then $(1+\eps-\eps_N)\overline{\D} \subset \Speps V_N \subset (1+\eps)\overline{\D}$, so that $\Speps V_N\Hto (1+\eps)\overline{\D}$ as $N\to\infty$. The rest of the corollary follows from \eqref{eq:inclus2} and \eqref{eq:inclus}.
\end{proof}

\begin{proof}[{\bf Proof of Proposition \ref{prop:VN}.}] Let us first show the claimed \eqref{eq:vn}. For $n\in \N$ and $s\geq 0$, $(V_n-sI_n)^H(V_n-sI_n))=(1+s^2)I_n-s(V_n+V_n^T)-O_n$, where $O_n=[o_{ij}]_{i,j=1}^n$ and $o_{ij}=\delta_{1,i}\delta_{1,j}$, $i,j=1,\ldots,n$, so that
\begin{eqnarray*}
[s_{\min}(V_n-sI_n)]^2 &=& \min \Spec((V_n-sI_n)^H(V_n-sI_n))\\
&=&1+s^2-s\min \Spec(-V_n-V_n^T-s^{-1}O_n),
\end{eqnarray*}
for $s>0$. Further, for $s\geq 1$, $\Spec(-V_n-V_n^T-s^{-1}O_n)\subset [-2,2]$ by \eqref{eq:ger}. Thus,  arguing, e.g., as in \cite[\S8.1]{CW.Heng.ML:SpecIncl1}, if $s\geq 1$, $\lambda\in \Spec(-V_n-V_n^T-s^{-1}O_n)$ is an eigenvalue with corresponding eigenvector $(v_n,\ldots,v_1)^T$ if and only if $\lambda= -2\cos(t)$, for some $t\in (0,\pi)$, with $v_j=\sin(jt)$, $j=1,\ldots,n$, and $t$ satisfying \eqref{eq:phi}. It is easy to see that the smallest positive solution of \eqref{eq:phi} is the unique solution in $[\pi/(2n+1),\pi/(n+1))$, so that $1+s^2-s\min \Spec(-V_n-V_n^T-s^{-1}O_n)=[v_n(s)]^2$, and \eqref{eq:vn} follows.
Suppose now that $n\in \N$ and $\eps\geq 0$. It is clear that, for $\lambda\in \C$, $\nu(\widetilde V^+_n-\lambda I_n^+)=\nu(V_n-\lambda I_n)$, so that $\Specn_{\eps} \widetilde V^{+}_n = \Specn_{\eps}  V_n$, by \eqref{eq:munu}, \eqref{eq:Speps}, and \eqref{eq:speps.rect}. Further, $\Speps V_n = \alpha_n(\eps)\overline{\D}$ by \eqref{eq:chara}, so \eqref{eq:PiCh} follows.

For $n\in \N$, $\alpha_n(\eps_1) <\alpha_n(\eps_2)$, for $0\leq \eps_1<\eps_2$, since $v_n$ is strictly monotonic increasing. Let us show that, for $s\geq 1$, $\phi_n(s)>\phi_{n+1}(s)$, for $n\in \N$. This is clear for $s=1$ as $\phi_n(1)= \pi/(2n+1)$, $n\in \N$. To see that this holds also for $s>1$, let $\psi(s,p)$ denote the unique solution $t$ of $s\sin((p+1)t)=\sin(pt)$ in $(\pi/(2p+1),\pi/(p+1))$, for $s>1$, $p\geq 1$, so that $\psi(s,n)=\phi_n(s)$, $n\in \N$. Differentiating this equation with respect to $p$, we see that $\partial_p\psi(s,p) < 0$, for $s>1$, $p\geq 1$, so that $\psi(s,p)$ is strictly decreasing as a function of $p$, and the required result follows. Thus, for $s\geq 1$, $n\in \N$, and $1\leq j<n$, $v_j(s) > v_{n}(s)$, and note that also $v_n^+(s) > v_n(s)$. Since, moreover, $v_n(1)=\eps_n$, $n\in \N$, so that $\alpha_n(\eps_n+\eps) \geq 1$, for $\eps\geq 0$, it follows that $\alpha_j(\eps_n+\eps) < \alpha_n(\eps_n+\eps)$, for $1\leq j<n$ and $\eps\geq 0$, and also $a^+_n(\eps_n+\eps) < \alpha_n(\eps_n+\eps)$, so that \eqref{eq:maxan} holds.

Supose that $n\in \N$ and $\eps\geq 0$. To see \eqref{eq:bounds}, note first that $v_n(s)=\nu(V_n-sI_n)$, for $s\geq 0$, so that$|v_n(1+\eps)-v_n(1)| \leq \eps$, by \eqref{eq:Lipschitz}. Since also $v_n(1)=\eps_n$, it follows that $v_n(1+\eps) \leq \eps_n+\eps$, so $1+\eps\leq \alpha_n(\eps_n+\eps)$. On the other hand, let $s=\alpha_n(\eps_n+\eps)\geq 1$. Then $v_n(s) = \eps_n + \eps$ and, since $\phi_n(s) \geq \pi/(2n+1)$ and $\cos(\pi/(2n+1))=1-\eps_n^2/2$, it follows from \eqref{eq:vn}  that
\begin{eqnarray*}
1+s^2-2s + \eps_n^2 s \leq (\eps_n+\eps)^2,
\end{eqnarray*}
so that, since $\sqrt{a^2+b^2} \leq a+b$, for $a,b\in \R$,
\begin{eqnarray*}
s &\leq& 1-\frac{\eps_n^2}{2} + \sqrt{\frac{\eps_n^4}{4} + 2\eps_n\eps + \eps^2}
 \leq  1 + \sqrt{2\eps_n\eps + \eps^2}\\
& = & 1 + \eps + \frac{2\eps_n\eps}{\eps +  \sqrt{2\eps_n\eps + \eps^2}} \leq  1 + \eps + \min\{\eps_n,\sqrt{2\eps_n\eps}\},
\end{eqnarray*}
establishing \eqref{eq:bounds}.

Since $\alpha_n(\eps_1) <\alpha_n(\eps_2)$, for $0\leq \eps_1<\eps_2$ and $n\in \N$, it follows from \eqref{eq:PiCh} that $\Specn_{\eps+\eps_{n-2}}V_n\cap \Specn_{\eps+\eps_{n}}V_n= \Specn_{\eps+\eps_{n}}V_n$, for $n>2$, and that 
\begin{equation} \label{eq:incluss}
\Specn_{\eps+\eps_{n}}V_n\subsetneq \Specn_{\eps+\eps'_n}V_n. 
\end{equation}
Further, \eqref{eq:maxan} implies that $\Specn_{\eps+\eps_n} V_m \subset \Specn_{\eps+\eps_n} V_n$, for $m=1,\ldots,n$ and that $\Specn_{\eps+\eps'_n} V^+_n \subset \Specn_{\eps+\eps'_n} V_n$. Thus, and by \eqref{eq:SpepsVN}, \eqref{SpepsVN2}, and \eqref{eq:PiCh},
\begin{equation} \label{eq:ff}
\Sigma_\eps^n(V_N) = \Specn_{\eps+\eps_n} V_n \quad \mbox{and} \quad
\Gamma_\eps^n(V_N) =  \Specn_{\eps+\eps'_n} V_n.
\end{equation}
Equation \eqref{eq:inclus2} follows from \eqref{eq:ff}, \eqref{eq:PiCh}, and \eqref{eq:bounds}. The inclusions \eqref{eq:inclus} follow from \eqref{eq:ff}, \eqref{eq:PiCh}, \eqref{eq:incluss}, and \eqref{eq:inclus}.
\end{proof}

\subsubsection{The discrete Laplacian} \label{sec:disLap2}

Let us turn now to the case $A=L_M$, again choosing for simplicity $N=M$ so that the block structure \eqref{eq:block} is trivial.
By Theorem \ref{thm:tau} (cf.~\eqref{eq:SpepsVN}) the $\tau$-method inclusion set is, for $\eps\geq 0$ and $n<N$,
\begin{equation} \label{eq:SpepsLN}
\Speps L_N\subset \Sigma_\eps^n(L_N) = \left\{\begin{array}{ll}\bigcup_{m=1}^n  \Specn_{\eps+ \eps_{n}}L_m, & n=1,2,\\
\Specn_{\eps+ \eps_{n-2}} L_n \cap \bigcup_{m=1}^n  \Specn_{\eps+ \eps_{n}}L_m, &n>2,\end{array}\right. 
\end{equation}
with 
$$
\eps_n := \eps_n(L_N) = 4\sin(\theta_n^{*}/2),
$$
where $\theta_n^{*} = \theta_n(L_N)$ denotes the unique solution in the range $[\pi/(2n+1),\pi/(n+2)]$ of the equation 
$$
8\sin(t/2)\cos((n+1/2)t) +\sin((n-1)t)=0.
$$
Equivalently (see \cite[\S3.2]{CW.Heng.ML:SpecIncl1}), $\theta_n^{*}$ is the unique solution in the range $(\pi/(n+3),\pi/(n+2)]$ of the equation
\begin{equation}\label{equi1s}
2\cos((n+1)t/2) = \cos((n-1)t/2),
\end{equation}
in particular, $\theta^{*}_1=\pi/3$ and $\theta^{*}_2=2\cos^{-1}(\sqrt{7/8})$.
Thus $\eps_1 =2$, $\eps_2=\sqrt{2}$, and
$$
4\sin(\pi/(2n+6)) < \eps_n < 4\sin(\pi/(2n+4)), \qquad n>1,
$$
so that $\eps_n\sim 2\pi/n$ as $n\to\infty$. As noted in Remark \ref{rem:comp}, since each $L_n$ is Hermitian,
\begin{equation} \label{eq:SpepsLn}
\Speps L_n = \Spec L_n + \eps \overline{\D}, \qquad n\in \N, \quad \eps>0,
\end{equation}
with $\Spec L_n$ given by \eqref{eq:DLS}. Note that $\Sigma_0^1(L_N)=\Specn_{\eps_1}L_1=2\overline{\D}$, the same inclusion set G as provided by Gershgorin's theorem \eqref{eq:ger} in the case that $N>2$. 

In this example, in contrast to the Jordan block example (see \eqref{eq:inclus2}),  $\Sigma^n_\eps(A)$, for $n>1$, cannot be characterised as the pseudospectrum of a single submatrix of $A$. 
In particular, for $n>1$, the inclusion set $\Sigma^n_{\eps}(L_N)$ is strictly larger than $\Specn_{\eps_n+\eps} L_n$. But, importantly, $\Specn_{\eps_n+\eps} L_n$ is also an inclusion set for $\Speps L_N$, for every $N\in \N$. This follows by applying Theorem \ref{thm:taubi} to the bi-infinite version of $L_N$, i.e.\ to the operator $E=V+V'$, where $V$ and its adjoint $V'$ are the left and right shift, respectively. To see this, note that $\Spec E=[-2,2]\supset \Spec L_N$ and $\Sigma^n_{\eps}(E) = \Specn_{\eps_n+\eps} L_n$, by \eqref{eq:Speps} and \eqref{eq:sigmaAAlt}, so that Theorem \ref{thm:taubi} implies that 
$\Specn_{\eps_n+\eps} L_n= \Sigma^n_{\eps}(E)\supset [-2,2]+\eps{\overline \D} \supset \Speps L_N$.

This last observation, in particular that $[-2,2]\subset \Sigma_0^n(L_N)$, leads to an explicit characterisation of  $\R\cap \Sigma_0^n(L_N)$, which, since $L_N$ is Hermitian, is also an inclusion set for $\Spec L_N$. Since, for $m\in \N$ and $\eps\geq 0$,
$
\R \cap \Specn_{\eps}L_m = \Spec L_m + [-\eps,\eps], 
$
with $\Spec L_m$ given by \eqref{eq:DLS}, it follows that
\begin{equation} \label{eq:22incl}
[-2,2] \cup (\R \cap \Specn_{\eps}L_m) = [-S_{\max}(m)-\eps,S_{\max}(m)+\eps],
\end{equation}
with $S_{\max}(m):= \max \Spec L_m = - \min \Spec L_m =2\cos(\pi/(m+1))$, so that, by \eqref{eq:SpepsLN},
\begin{equation} \label{eq:RSig}
\Spec L_N \subset \R\cap \Sigma_0^n(L_N) =  [-2-\eps^*_n,2 + \eps^*_n], 
\end{equation}
where, for $n\in \N$, 
\begin{eqnarray*}
\eps^*_n&:=&S_{\max}(n)-2+\eps_n=4(\sin(\theta^*_n/2) - \sin^2(\pi/(2n+2)))\\
&\geq &4(\sin(\pi/(4n+2))  - \sin^2(\pi/(2n+2)))\geq 0,
\end{eqnarray*}
with equality in the above inequalities if and only if $n=1$. Note that $\eps^*_n\sim \eps_n\sim  2\pi/n$ as $n\to\infty$.

Let us turn now to the $\tau_1$ inclusion sets.
Let 
$L_1^+ := (1,0,1)^T$, $\widetilde L_1^+ := (1,0,0)^T$, $\widehat L_1^+ := (0,0,1)^T$, and, for $n>1$, let
\begin{eqnarray*}
L^+_{n} &:=& \begin{pmatrix}
1 & 0&\cdots&0&0\\\hline
& & V_{n}\\\hline
0& 0 &\cdots&0&1
\end{pmatrix}_{(n+2)\times n}, \quad
\widetilde L^+_{n} := \begin{pmatrix}
1 & 0&\cdots&0\\\hline
& & V_{n}\\\hline
0& 0 &\cdots&0
\end{pmatrix}_{(n+2)\times n}, \\
& & \hspace{15ex} \widehat L^+_{n} := \begin{pmatrix}
0&\cdots&0&0\\\hline
 & V_{n}\\\hline
 0 &\cdots&0&1
\end{pmatrix}_{(n+2)\times n}.
\end{eqnarray*}
Then $\{B^{+}_{n,k}: 0\leq k\leq N-n\} = \{\widetilde L^+_n,\widehat L^+_n\}$, if $n=N-1$, $= \{L^+_n,\widetilde L^+_n,\widehat L^+_n\}$, if $n<N-1$. Recall that a permutation of rows of a matrix leaves its singular values unchanged. Thus $s_{\min}(\widetilde L^+_n-\lambda I^+_N) = s_{\min}(\widehat L^+_n-\lambda I^+_N)$, for $\lambda\in \C$, so that $\Speps \widetilde L^+_n = \Speps \widehat L^+_n$, for $\eps \geq 0$. Thus, for $\eps\geq 0$,
by application of  Theorem \ref{thm:tau1},
\begin{equation} \label{SpepsLN2}
\Speps L_N \subset \Gamma^n_\eps(L_N) = \left\{\begin{array}{ll}\Specn_{\eps+\eps'_n} \widetilde L^{+}_n,& n=N-1,\\
\Specn_{\eps+\eps'_n} L^+_n \cup \Specn_{\eps+\eps'_n} \widetilde L^{+}_n, &n<N-1,\end{array}\right. 
\end{equation}
and also $\Gamma^n_\eps(L_N)   \subset \Specn_{\eps+\eps'_n} L_N$, where
$$
\eps'_n := \eps'_n(L_N)  = 4\sin(\pi/(2n+2)). 
$$
We note that, for this example, because $A=L_N$ is Hermitian, the computation of $\Sigma^n_\eps(L_N)$ is straightforward in view of \eqref{eq:SpepsLn}, while determining $\Gamma^n_\eps(L_N)$ requires the computation of pseudospectra of  rectangular matrices.

Equations \eqref{eq:DLS} and \eqref{eq:RSig} suggest that $\Sigma^n_0(L_N)$ is an increasingly good approximation of $\Spec L_N$ as $n$ and $N$ increase. This, and the analogous result for $\Gamma^n_\eps(L_N)$, is confirmed by the following lemma (cf.\ Corollary \ref{cor:convVN}).

\begin{lemma} \label{lem:conv}
Suppose that $\eps \geq 0$. As $N\to\infty$, $\Speps L_N\Hto [-2,2]+\eps\overline{\D}$. Further, where $S_n$ denotes  $\Sigma_\eps^n(L_N)$ or $\Gamma_\eps^n(L_N)$,  $S_n\Hto [-2,2]+\eps\overline{\D}$, as $n,N\to\infty$ with $n< N$, so that $d_H(S_n,\Speps L_N)\to 0$ as $n,N\to\infty$ with $n<N$. 
\end{lemma}
\begin{proof} It is clear from \eqref{eq:DLS} that $\Spec L_N \Hto [-2,2]$ as $N\to\infty$, so that $\Speps L_n = \Spec L_N + \eps{\overline \D} \Hto [-2,2]+\eps\overline{\D}$. By \eqref{eq:SpepsLN}, $\Speps L_N \subset \Sigma_\eps^n(L_N) \subset \Specn_{\eps_{n-2}+\eps} L_n = \Spec L_n + (\eps+\eps_{n-2})\overline{\D}$, so that also $\Sigma_\eps^n(L_N)\Hto [-2,2]+\eps\overline{\D}$ as $n,N\to\infty$. Similarly, by the discussion around \eqref{SpepsLN2}, $\Speps L_N \subset \Gamma_\eps^n(L_N) \subset \Specn_{\eps''_n+\eps} L_N$, so also  $\Gamma_\eps^n(L_N)\Hto [-2,2]+\eps\overline{\D}$ as $n,N\to\infty$.
\end{proof}

The above lemma tells us that the inclusion set $\Sigma^n_0(L_N)$ is an increasingly good approximation to $\Spec L_N$ as $n,N\to\infty$, whereas the inclusion set provided by the standard Gershgorin theorem \eqref{eq:ger} is $G=2\overline{\D}$ if $N>2$. Of course, where the matrix is Hermitian so the spectrum is real, it makes sense to use this fact, taking the intersection of inclusion sets with $\R$. While $G$ is a poor approximation to $\Spec L_N$, $\R\cap G = [-2,2]$  is a sharp bound when $N$ is large, by  Lemma \ref{lem:conv}, and, except when $n=1$, sharper than $\R\cap\Sigma_0^n(L_N)$ given by \eqref{eq:RSig}.  Indeed, where $G$ is the Gershgorin inclusion set \eqref{eq:ger}, we will see in Remark \ref{rem:BandG} that $G\cap\R$ is a sharp inclusion set for the eigenvalues of large tridiagonal Hermitian Toeplitz matrices in general; our $\tau$ and $\tau_1$ sets offer no improvement.

Thus the example $L_N$ makes the point that our new inclusion sets need not improve on Gershgorin; that, for some Hermitian matrices, $\R\cap G$ already provides sharp eigenvalue enclosures. But we will see in \S\ref{sec:bandT} that, for large Hermitian Toeplitz matrices in general, $\R\cap G$ overestimates the spectrum while our $\tau$- and $\tau_1$-inclusion sets continue to provide sharp approximations (see Theorem \ref{thm:Herm}, Remark \ref{rem:BandG}, and Theorem \ref{thm:final}). Similarly, we will see  in \S\ref{sec:2T} that, within the class of  tridiagonal matrices, if $L_N$ is perturbed by addition of a real diagonal matrix, $\R\cap G$ can significantly overestimate the spectrum while $\Sigma_0^n(L_N)$ can, again, provide a sharp inclusion set.

\begin{remark}{\bf(Comparison with existing block-matrix Gershgorin)} \label{rem:bmG}
We have just made comparison with the classical Gershgorin theorem \eqref{eq:ger}. But how does our new inclusion set family compare with the existing block matrix versions of the Gershgorin theorem, discussed in Remark \ref{rem:case1}? Suppose that $A=L_M$ for some $M>1$ and write $A$ in the block form \eqref{eq:block} for some $1<N\leq M$, so that  $a_{ij}\in \C^{m_i\times m_j}$, for $i,j=1,\ldots, N$. Let us compute the inclusion set \eqref{eq:salas2} given by the block-matrix version of Gershgorin in the case that the matrix norms are the $2$-norm. It is clear that $A$ is block-tridiagonal, that $a_{ii} = L_{m_i}$, $i=1,\ldots,N$, and  that each non-zero off-diagonal block $a_{ij}$ has a single non-zero entry (taking the value 1), so that $\|a_{i,i+1}\|=\|a_{i+1,i}\|=1$, $i=1,\ldots,N-1$. Thus \eqref{eq:salas2} reads in this case
\begin{eqnarray*}
\Spec L_M\ \subset\ G^{N} &:= &\Specn_1 L_{m_1} \cup \Specn_1 L_{m_N} \cup \bigcup_{k=2}^{N-1} \Specn_2 L_{m_k}\\
&   = &\left(\overline{\D}+(\Spec L_{m_1} \cup \Spec L_{m_N})\right) \cup \left(2\overline{\D} +\bigcup_{k=2}^{N-1} \Spec L_{m_k}\right).
\end{eqnarray*}
This is no better than the inclusion set provided by the standard Gershgorin theorem, which it reduces to if $N=M$. Indeed, if $M>2$ it is immediate from the above and \eqref{eq:DLS} that $[-2,2]\subset G^{N}$. Thus, arguing as we did to get \eqref{eq:RSig}, using \eqref{eq:22incl}, we see that, where $m_- := \max(m_1,m_N)$ and $m_+:= \max_{k=2,\ldots,N-1}m_k$, 
\begin{equation} \label{eq:RcapGN}
\R\cap G^{N} =  [-S_{\max}(m_-)-1,S_{\max}(m_-)+1] \cup [-S_{\max}(m_+)-2,S_{\max}(m_+)+2],
\end{equation}
with the second interval present only for $N>2$. Thus $\R\cap G^{N}$ is strictly larger than $\R\cap G = [-2,2]$ if $m_->2$, or if $N>2$ and $m_+>1$.
  Indeed, as  $M\to\infty$ with $N>2$ and  $m_+\to\infty$, $G^{N}\Hto [-2,2] + 2\overline{\D}$  and $\R\cap G^{N}\Hto [-4,4]$.
\end{remark}

\subsection{General Toeplitz matrices} \label{sec:bandT}

The matrices in \S\ref{sec:disLap} are both examples of tridiagonal Toeplitz matrices. In this section we consider general Toeplitz matrices, i.e., the case $A=A_M$, for some $M>1$, where, for some complex coefficients $(a_j)_{j\in \Z}$,
$$
A_M := \left(\begin{array}{cccc} a_0 &a_{-1} &\ldots & a_{1-M}\\
a_1 & a_0 &\ldots &a_{2-M}\\
\vdots &\vdots & & \vdots\\
a_{M-2} & a_{M-3} & \ldots & a_{-1}\\
a_{M-1} & a_{M-2} & \ldots & a_0
\end{array}\right)_{M\times M}, \quad M\in \N.
$$

We are particularly interested in the case when $M$ is large, seeking to generalise Corollary \ref{cor:convVN} and Lemma \ref{lem:conv} as they relate to the $\tau$ and $\tau_1$ methods. The study of large Toeplitz  matrices and the asymptotics of their spectra and pseudospectra has a substantial literature that we will draw on; see, for example, \cite{ReichelTref} and the monographs \cite{BoeGru,BoeSi2,TrefEmbBook}.

Let $a$ denote the symbol associated to the coefficients $(a_j)_{j\in \Z}$, defined formally by
\begin{equation} \label{eq:sym}
a(t) := \sum_{j\in \Z} a_j t^j, \qquad t\in \T.
\end{equation}
As a substantial illustration of Remark \ref{rem:band}, we will first study, focussing on the $\tau$ method,  the case that, for some $w\in \N$, $a_j=0$ if $|j|>w$, so that $(A_M)_{M>1}$ is a family of banded Toeplitz matrices, each with band-width $\leq w$. In that case the sum \eqref{eq:sym} is finite and $a$ is rational. As our only example in this paper where $A$ does not coincide with its tridiagonal part $B$, so that the remaining part $C=A-B$ is non-zero, we will also consider, focussing on the $\tau_1$ method, the case where infinitely many of the coefficients $a_j$ are non-zero, but $a$ is in the so-called {\emph Wiener algebra} (e.g., \cite[Example 1.5]{BoeSi2}), meaning that
\begin{equation} \label{eq:wiener}
\|a\|_{\mathcal{W}} := \sum_{j\in \Z} |a_j| <\infty.
\end{equation}
In that case the sum \eqref{eq:sym} is absolutely and uniformly convergent, and $a\in C(\T)$, the symbol is continuous. A key result is that if $a$ is in the Wiener algebra then\footnote{This result from \cite{BoeSi2} needs, in fact, only that $a$ is piecewise continuous. The notion of convergence of sets used in \cite[Corollary 3.18]{BoeSi2} coincides with Hausdorff convergence by standard arguments, e.g., \cite[Proposition 3.6]{HaRoSi2}.} \cite[Corollary 3.18]{BoeSi2}, for all $\eps>0$,
\begin{equation} \label{eq:Toep}
\Speps A_M\ \Hto\ \Speps A_\infty^{+}, \qquad \mbox{as} \qquad M\to\infty.
\end{equation}
Here $A_\infty^{+}$ is the infinite Toeplitz matrix given by
$$
A_\infty^{+} := \left(\begin{array}{ccc} a_0 &a_{-1} &\ldots \\
a_1 & a_0 &\ldots\\
\vdots &\vdots & \vdots
\end{array}\right),
$$
which is a bounded operator on $\ell^2(\N)$ for $a$ in the Wiener algebra, with
$$
\|A_\infty^{+}\| \leq \|a\|_{\mathcal{W}}.
$$ 
$\Speps A_\infty^{+}$, the $\eps$-pseudospectrum of $A_\infty^{+}$ as an operator on $\ell^2(\N)$, is characterised in \cite[\S3.6]{BoeSi2}. 

If $A_\infty^{+}$ is self-adjoint, so that $a_{-j}=\bar a_j$, $j\in \Z$, then $A_M$ is Hermitian, so that, for $\eps>0$, $\Speps A_\infty^{+} = \Spec A_\infty^{+} + \eps {\overline \D}$, $\Speps A_M = \Spec A_M + \eps {\overline \D}$, and $d_H(\Spec A_M,\Spec A_\infty^{+}) \leq 2\eps + 
d_H(\Speps A_M,\Speps A_\infty^{+})$. Thus, if $A_\infty^{+}$ is self-adjoint and $a$ is in the Wiener algebra, \eqref{eq:Toep} holds also for $\eps=0$, i.e., $\Spec A_M$ $\Hto$ $\Spec A_\infty^{+}$; further, $a$ is real-valued and (e.g., \cite[Theorem 1.17]{BoeSi2}) $\Spec A_\infty^+ = a(\T)$, i.e.
\begin{equation} \label{eq:SpAinf1}
\Spec A_\infty^{+} = [a_{\min},a_{\max}], \quad \mbox{where} \quad a_{\min}:= \min_{t\in \T} a(t),  \quad a_{\max}:= \max_{t\in \T} a(t).
\end{equation}

\vspace{-3ex}
 
\subsubsection{The banded case} \label{sec:banded}

Suppose that, for some $w\in \N$, $a_j=0$ if $|j|>w$, so that, for each $M>1$, $A_M$ is banded  with band-width $w_M\leq w$.
Let us proceed as in Remark \ref{rem:band}, writing $A=A_M$ in the block form \eqref{eq:block}, with $m_i\geq w_M$, for $i=1,\ldots,N$. Then $A$, written in the block form \eqref{eq:block}, is tridiagonal, coinciding with its tridiagonal part $B$. Further, for $n=1,\ldots, N$ and $k=0,\ldots,N-n$,
$$
B_{n,k} = A_{M_{n,k}}, \qquad \mbox{with} \qquad M_{n,k} := \sum_{i=k+1}^{k+n} m_i,
$$
so that, by Theorem \ref{thm:tau},
\begin{equation} \label{eq:SpepsT}
\Speps A\subset \Sigma_\eps^n(A) = \left\{\begin{array}{ll}\sigma_\eps^n(A), & n=1,2,\\
\sigma_\eps^n(A) \cap \widehat \sigma_\eps^n(A), & n>2,\end{array}\right. 
\end{equation}
for $\eps\geq 0$, where
\begin{eqnarray} \nonumber
\widehat \sigma_\eps^n(A) &:= &
\bigcup_{k=0}^{N-n}  \Specn_{\eps+ \eps_{n-2}(A)}A_{M_{n,k}},\\ \nonumber
\sigma_\eps^n(A) &:=& \bigcup_{k=0}^{N-n} \Specn_{\eps+\eps_n(A)} A_{M_{n,k}} \cup\\ \label{eq:sigT}
& & \hspace{4ex} \bigcup_{m=1}^{n-1} \left(\Specn_{\eps+\eps_n(A)} A_{M_{m,0}}\cup \Specn_{\eps+\eps_n(A)} A_{M_{m,N-m}}\right),
\end{eqnarray}
and  with $\eps_n(A)$ given by \eqref{eq:epsdef} (with $C=0$).

The following result is essentially a corollary of \eqref{eq:Toep}. Its proof uses standard properties of Hausdorff convergence and pseudospectra (see, e.g., \cite[\S1.3, Eqn.~(2.8)]{CW.Heng.ML:SpecIncl1}), namely: i) that, if $L$, $U$, $S$, and $T$ are compact subsets of $\C$ and $L\subset S\subset U$, then 
$$
d_H(S,T) \leq \max\{d_H(L,T),d_H(U,T)\},
$$
in particular $d_H(S,U)\leq d_H(L,U)$; and ii) that if $F\in L(H)$, for some Hilbert space $H$, then, for $\eps\geq 0$, $\Specn_\eta F\Hto \Speps F$ as $\eta \to \eps^+$.   

\begin{theorem} \label{thm:Tgen}
Suppose $\eps>0$ and $A=A_M$ and $m_i\geq w_M$, for $i=1,\ldots,N$. Then $\Speps A$ $\Hto$ $\Speps A_\infty^{+}$ as $M\to\infty$ and $\Sigma_\eps^n(A)\Hto  \Speps A_\infty^{+}$ as $n\to\infty$ (with $M\geq w_M N$ and $N>n$), so that $d_H(\Sigma_\eps^n(A),\Speps A)\to 0$ as  $M,n\to\infty$. If $a_{-j} = \bar a_j$, $j\in \Z$, then these results hold also for $\eps=0$, in particular $d_H(\Sigma_0^n(A),\Spec A)\to 0$ as  $M,n\to\infty$.
\end{theorem}
\begin{proof}
Suppose $\eps>0$. That $\Speps A\Hto \Speps A_\infty^{+}$ as $M\to\infty$ is \eqref{eq:Toep}. To see the rest of the theorem, note that $r_L(A)\leq \|A\|_\infty \leq \|a\|_{\mathcal{W}}$ and $r_U(A)\leq \|a\|_{\mathcal{W}}$, so that $r(A)\leq 2\|a\|_{\mathcal{W}}$ and, by \eqref{eq:epsdef}, 
\begin{equation} \label{eq:epsb}
\eps_{n-2}(A) \leq \frac{2\pi \|a\|_{\mathcal{W}}}{n} \to 0 \quad \mbox{as}\quad n\to\infty.
\end{equation}
Suppose now that $K>1$. Then, for every $\eta>\eps$, provided $n$ is sufficiently large so that $\eps+\eps_{n-2}(A) \leq \eta$,
\begin{eqnarray*}
d_H(\Specn_{\eps+\eps_{n-2}(A)} A_{K},\Speps A_{K}) &\leq& d_H(\Specn_{\eta} A_{K},\Speps A_{K}) \\
& \leq & d_H(\Specn_{\eta} A_{K},\Specn_{\eta} A_\infty^{+})\\
& & \hspace{2ex} +\ d_H(\Specn_{\eta} A_\infty^{+},\Speps A_\infty^{+})\\
& & \hspace{2ex} +\ d_H(\Speps A_\infty^{+},\Speps A_{K}).
\end{eqnarray*}
The second term on the right-hand side of this last expression can be made arbitrarily small by choosing $\eta$ sufficiently close to $\eps$, and, for every $\eta>\eps$, the other terms $\to 0$ as $K\to\infty$ by \eqref{eq:Toep}. Thus, and by \eqref{eq:Toep}, for every $\delta>0$ there exists $n_0,K_0\in \N$ such that
$$
d_H(\Specn_{\eps+\eps_{n-2}(A)} A_{K},\Speps A_\infty^{+}) \leq \delta
$$
if $n> n_0$ and $K>K_0$. Since $M_{n,k}\geq n$, for each $k\in\{0,\ldots,N-n\}$, it follows that $\widehat \sigma_\eps^n(A)$ $\Hto$  $\Speps A_\infty^{+}$ as $n\to\infty$ with $N>n$. Since $\Speps A \subset \Sigma_\eps^n(A) \subset \widehat \sigma_\eps^n(A)$, also $\Sigma_\eps^n(A)\Hto  \Speps A_\infty^{+}$ as $n\to\infty$ with $N>n$.
If $A_\infty^{+}$ is self-adjoint then \eqref{eq:Toep} holds also for $\eps=0$, so the above arguments work also in that case.
\end{proof}

Where $B$ is a square matrix or $B\in L(H)$, for a Hilbert space $H$, let $W(B)$ denote the numerical range of $B$.  Where $\conv(S)$ denotes the convex hull of a set $S\subset \C$, recall that $\conv(\Spec B)\subset \overline{W(B)}$, with equality if $B$ is normal (e.g., \cite{GusRao97}), and note  that $W(PB|_{\widetilde H})\subset W(B)$ if $B\in L(H)$, $\widetilde H\subset H$ is a closed subspace, and $P:H\to\widetilde H$ is orthogonal projection. Thus, and recalling \eqref{eq:SpAinf1}, we see that, for $M$, $M'\in \N$ with $M<M'$,
\begin{equation} \label{eq:conv}
\conv(\Spec A_M) \subset \conv(\Spec A_{M'}) = W(A_{M'}) \subset W(A_\infty^{+}) = \Spec A_\infty^{+},
\end{equation}
 in the Hermitian case $a_{-j} = \bar a_j$, $j\in \Z$, so that 
\begin{equation} \label{eq:inclH}
\Specn_\eps A = \Spec A + \eps \overline{\D} \subset \Spec A^+_\infty + \eps\overline{\D}=\Specn_\eps A^+_\infty, \qquad \eps\geq 0.
\end{equation}
In the Hermitian case it is convenient to introduce
$\rho(a) := \|a-a_0\|_{\mathcal{W}} = 2\sum_{j=1}^\infty |a_j|$,
noting that (cf.~the proof of Theorem \ref{thm:Tgen})
\begin{equation}\label{eq:rAb}
r(A) \leq \rho(a).
\end{equation}
Recall also $\theta_n^*$ introduced above \eqref{equi1s}, so that $\theta_n^*\in (\pi/(n+3),\pi/(n+2]$ is the solution of \eqref{eq:theta} in the Hermitian case that $r_L(A)=r_U(A)=r(A)/2$.
\begin{theorem} \label{thm:Herm}
Suppose $a_{-j} = \bar a_j$, $j\in \Z$, $\eps\geq 0$, $A=A_M$, and $m_i\geq w$, for $i=1,\ldots,N$. Then 
\begin{equation} \label{eq:dHSS}
d_H(\Sigma_\eps^n(A),\Speps A) \leq \Theta_{n,N}\,\rho(a),
\end{equation}
for $1\leq n<N$ and $\eps\geq 0$, where 
$$
\Theta_{n,N} :=  \Theta_n+\Theta_N \; \mbox{and} \; \Theta_m := 2\sin(\theta^*_m/2) \leq 2\sin\left(\frac{\pi}{2m+4}\right)\leq\frac{\pi}{m+2}, \; m\in \N.
$$ 
Further, 
\begin{equation} \label{eq:RcS}
\R\cap \Sigma_0^n(A) \subset [a_{\min}-\Theta_n \rho(a), a_{\max}+\Theta_n\rho(A)]
\end{equation}
and
\begin{equation} \label{eq:Rincl}
\Spec A \subset \R\cap \Sigma_0^n(A) \subset \Spec A + [-\Theta_{n,N}\,\rho(a),\Theta_{n,N}\,\rho(a)].
\end{equation}
\end{theorem}
\begin{proof} Let $A_\infty$ denote the bi-infinite Laurent matrix with symbol $a$, defined as in \cite[\S1.2]{BoeSi2}, and recall that $\Spec A_\infty = a(\T)$ (e.g., \cite[Theorem 1.2]{BoeSi2}), so that, by \eqref{eq:SpAinf1}, 
\begin{equation} \label{eq:AA+}
\Speps A^+_\infty= \Spec A^+_\infty+\eps\overline{\D} = \Spec A_\infty+\eps\overline{\D}=\Speps A_\infty, \quad \eps\geq 0, 
\end{equation}
in this Hermitian case. Write $A_\infty$ in block-tridiagonal form as $A_\infty = [e_{ij}]_{i,j\in \Z}$, choosing the blocks so that $e_{ij}=a_{ij}$, $1\leq i,j\leq N$, and so that $e_{i+N,j+N}=e_{ij}$, $i,j\in \Z$. Then,
where $E_{n,k}$ is as defined below \eqref{eq:rdef2} and $\eta_N(A_\infty)$ by \eqref{eq:wtauinf}, $E_{N,k}=A$, for $k\in \Z$. Thus, and by Theorem \ref{thm:taubi}, $\Specn_\eps A_\infty\subset \Sigma_\eps^N(A_\infty)=\Specn_{\eps+\eta_N(A_\infty)} A$, for $\eps\geq 0$, with $\eta_N(A_\infty)=2r(A_\infty)\sin(\theta^*_N/2)\leq 2\rho(a)\sin(\theta^*_N/2)$ (cf.~\eqref{eq:rAb}). Thus, and by \eqref{eq:inclH} and \eqref{eq:AA+},
\begin{equation} \label{eq:inC1}
\Speps A \subset \Speps A^+_\infty \subset \Specn_{\eps+\eta_N(A_\infty)} A= \Speps A + \eta_N(A_\infty)\overline{\D}, \qquad \eps\geq 0.
\end{equation}
By \eqref{eq:SpepsT}, $\Speps A \subset \Sigma_\eps^n(A) \subset \sigma_\eps^n(A)$. Further, $\sigma_\eps^n(A)\subset$ $\Spec_{\eps+\eps_n(A)} A^+_\infty$ by \eqref{eq:inclH}, where $\eps_n(A) = 2r(A)\sin(\theta^*_n/2)\leq 2\rho(a)\sin(\theta^*_n/2)$, by \eqref{eq:rAb}, so that, applying \eqref{eq:inC1},
$$
\Speps A \subset \Sigma_\eps^n(A) \subset \Speps A + (\eps_n(A)+\eta_N(A_\infty))\overline{\D} \subset \Speps A + \Theta_{n,N}\, \rho(a)\overline{\D},
$$
and \eqref{eq:dHSS} and \eqref{eq:Rincl} follow.
The inclusion \eqref{eq:RcS} holds since, by \eqref{eq:inclH} and \eqref{eq:SpAinf1}, $\Spec A_J \subset [a_{\min},a_{\max}]$ and $\R \cap \Spec_\eps A_J =\Spec A_J + [-\eps,\eps]$, for $\eps\geq 0$ and $J\in \N$. 
\end{proof}

It follows from \cite[Theorem 3.19]{BoeSi2} that $\Speps A_\infty^{+} = (1+\eps)\overline{\D}$ and $\Speps A_\infty^{+} = [-2,2]+\eps \overline{\D}$ for the examples in \S\ref{sec:jordan} and \S\ref{sec:disLap2}, respectively. Thus Theorem \ref{thm:Tgen}, applied with $m_i=1$, $i=1,\ldots, N$, so that $N=M$, implies Corollary \ref{cor:convVN} and Lemma \ref{lem:conv}, as they relate to the $\tau$ method, as special cases.
\vspace{0.7ex}
\begin{remark}{\bf (Approximating $\Speps A$ using the above results)} \label{rem:AppSp}
Given a large Toeplitz matrix $A$ of order $M$ and band-width $w$ one might proceed as follows, informed by the above results, to obtain an inclusion set for $\Speps A$ that is also a good approximation to $\Speps A$. \begin{enumerate}
\item Assuming that $M\geq w^2$, set $N:=\lfloor M/w \rfloor$, the largest integer $\leq M/w$, set $r:=M-wN\in \{0,\ldots,w-1\}$, and note that $N-r\geq 1$.
\item Set $m_i := w$, for $i=1,\ldots,N-r$, $m_i:= w+1$, if $N-r <i\leq N$. Then $\sum_{i=1}^N m_i = M$ and also $nw\leq M_{n,k}\leq n(w+1)$, for $1\leq n<N$, $0\leq k\leq N-n$, so that there are at most $n+1$ different values of $M_{n,k}$. 
\item For $i,j=1,\ldots,N$, let $a_{ij}$  be the appropriate $m_i\times m_j$ submatrix of $A$ such that $A$ is \eqref{eq:block} in block form. Note that this block matrix is tridiagonal by Remark \ref{rem:band}.
\item Define $\Sigma_\eps^n(A)$ by \eqref{eq:SpepsT}, for $n=1,\ldots,N-1$. The observations in the above steps regarding the size of $M_{n,k}$ imply that, no matter how large the value of $M$,  computing $\Sigma_\eps^n(A)$ requires the computation of at most $4n$ pseudospectra of square matrices, and each of these matrices has order $\leq n(w+1)$. Of course,  $\Speps A \subset \Sigma_\eps^n(A)$,  for each $n$. Further, given that $M$ is large, Theorem \ref{thm:Tgen} suggests that, if $\eps>0$, $\Sigma_\eps^n(A)$ will be an increasingly good approximation to $\Speps A$ as $n$ increases. Indeed, if $A$ is Hermitian, this is true also for $\eps=0$ and the error is quantified, for all $\eps\geq 0$, in Theorem \ref{thm:Herm}.
\end{enumerate}
The above algorithm is most straightforward when $M$ is a multiple of the band-width $w$, i.e., $M=Nw$, so that $m_i=w$, $i=1,\ldots,N$. The expression for $\Sigma_\eps^N(A)$ simplifies in that case to
\begin{equation} \label{eq:Sigw}
\Sigma_\eps^n(A) := \Specn_{\eps+ \eps_{n-2}(A)}A_{nw} \cap  \bigcup_{m=1}^{n} \Specn_{\eps+\eps_n(A)} A_{mw},
\end{equation}
with the term  $\Specn_{\eps+ \eps_{n-2}(A)}A_{nw}$ absent for $n=1,2$.
This is an inclusion set for $\Speps A= \Speps A_{Nw}$, for $N>n$. Since $\Speps A_{Nw}\Hto$ $\Speps A_\infty^{+}$ as $N\to\infty$, by \eqref{eq:Toep}, $\Sigma_\eps^n(A)$, defined by \eqref{eq:Sigw}, is also an inclusion set for $\Speps A_\infty^{+}$, i.e.,
\begin{equation} \label{eq:inclInf}
\Speps A_\infty^{+} \subset \Sigma_\eps^n(A), \qquad n\in \N.
\end{equation}
This holds, because it uses \eqref{eq:Toep},  only for  $\eps>0$ in the first instance, but then, by taking the intersection of \eqref{eq:inclInf} over all $\eps>0$, also for $\eps=0$. This in turn implies, in the Hermitian case $a_{-j}=\bar a_j$, $j\in \Z$, that
\begin{equation} \label{eq:Rcap}
\R\cap \Sigma_0^n(A) = [a_{-}(nw)-\eps_n(A), a_+(nw)+\eps_n(A)] \supset [a_{\min},a_{\max}],
\end{equation}
where, for $J\in \N$, $a_-(J) := \min \Spec A_J$ and $a_+(J):= \max \Spec A_J$. For \eqref{eq:inclInf} and \eqref{eq:SpAinf1} imply in the Hermitian case that $[a_{\min},a_{\max}]\subset \Sigma_0^n(A)$, and then \eqref{eq:Rcap} follows by arguing as we did to obtain \eqref{eq:RSig}, using \eqref{eq:conv}.
\end{remark}
\vspace{0.7ex}
\begin{figure}[t]
\begin{center}
\includegraphics[width=80mm]{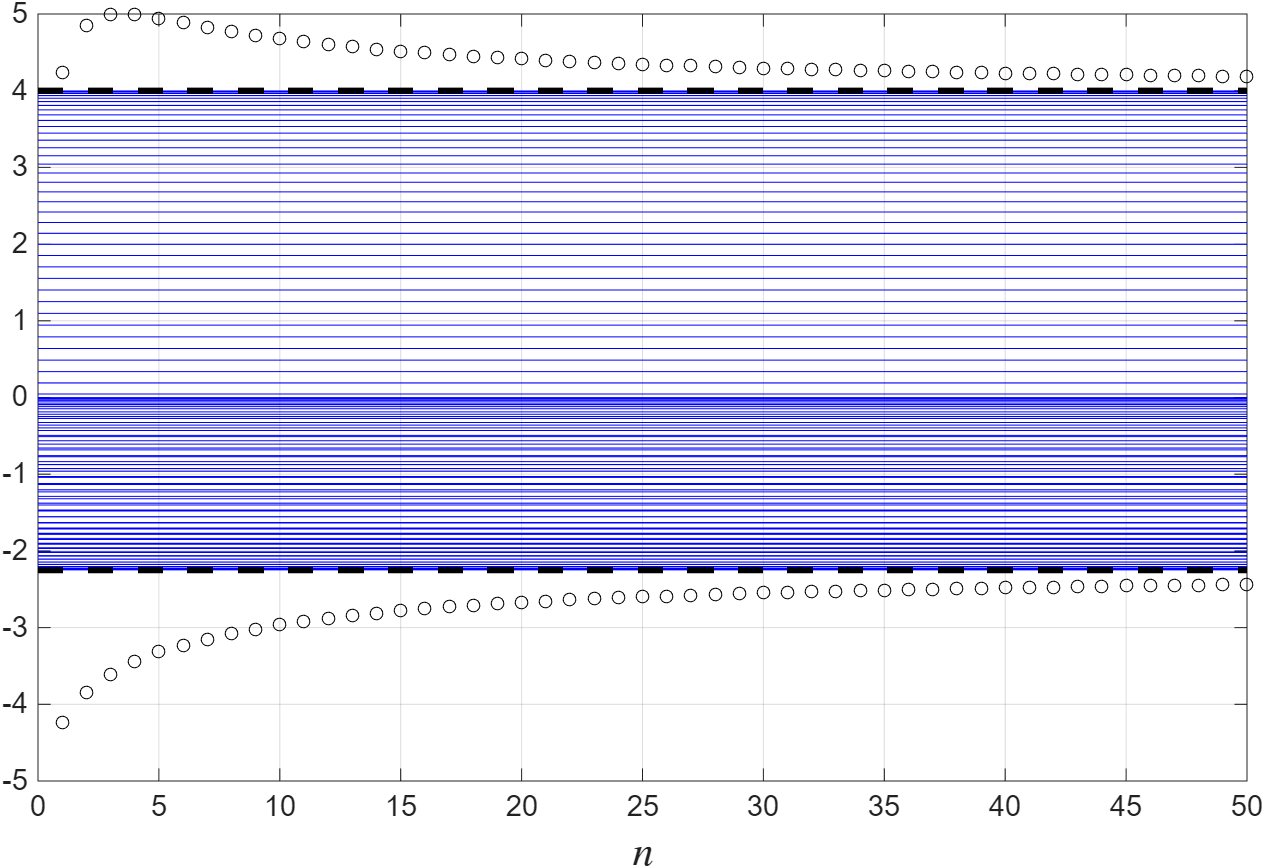} 
\end{center}

\vspace{-3ex}

\caption{Plot of $\Spec A$, where $A=A_M$ is the order $M$ pentadiagonal Toeplitz example at the end of Remark \ref{rem:BandG}. Also plotted are the $\tau$-method inclusions sets, $\R\cap \Sigma_0^n(A)$. The horizontal solid lines indicate the eigenvalues of $A$ when $M=110$, and the dashed lines the endpoints of $\Spec A^+_\infty=[-9/4,4]$. The circles are the upper and lower limits of the interval $\R\cap \Sigma_0^n(A)$, given by \eqref{eq:Rcap}. Where $G$ is the standard Gershgorin inclusion set, $\R\cap G=[-4,4]$; see Remark \ref{rem:BandG}.} \label{fig:penta}
\end{figure}
\begin{remark}{\bf (Gershgorin versus $\tau$ method: Hermitian Toeplitz)} \label{rem:BandG}
As noted in and above \eqref{eq:SpAinf1} and in \eqref{eq:inclH}, $\Specn A \subset \Specn A^+_\infty$ and $\Specn A\Hto \Specn A^+_\infty =[a_{\min}, a_{\max}]$ as $M\to\infty$, so that the eigenvalues of $A$ are an increasingly dense subset of $[a_{\min},a_{\max}]$ as $M$ increases. Thus, and by \eqref{eq:RcS}, as quantified already in \eqref{eq:Rincl}, $\R\cap \Sigma_0^n(A)$ is a sharp inclusion set for $\Spec A$ when $n$ and  $N$ are large. 

Where $G$ is the Gershgorin inclusion set given by \eqref{eq:ger}, we have $G=a_0+\rho(a)\overline{\D}$ if $M>w$, so that
$$
\Specn A \subset \R\cap G = [a_0-\rho(a),a_0+\rho(a)]\supset [a_{\min},a_{\max}].
$$
Thus $\R\cap G$ is a sharp inclusion set, asymptotically in the limit $M\to\infty$, if and only if $[a_0-\rho(a),a_0+\rho(a)]= [a_{\min},a_{\max}]$. This holds if $A$ has only three non-zero diagonals, in particular if $A$ is tridiagonal, and
 in certain other cases. (As an example, suppose that $a_2, a_6>0$, but $a_j=0$ if $j\in \N\setminus \{2,6\}$. Then $a_{\max}=a(0)=a_0+\rho(a)$ and $a_{\min} = a(\ri)=a_0-\rho(a)$.) But generically, when $A$ has more than three non-zero diagonals, $\R\cap G$ is strictly larger than $[a_{\min},a_{\max}]$. 
 
 As a concrete example, suppose that $A$ is pentadiagonal, with $a_0=0$, $a_1=a_2=1$. Then $\R\cap G = [-\rho(a),\rho(a)]=[-4,4]$ while elementary calculations give that 
 $
 [a_{\min},a_{\max}]=[a(-1/4\pm\ri \sqrt{15}/4),a(0)]=[-9/4,4].
 $
 In the case $M=110$ we plot $\Spec A$ and $\R\cap \Sigma_0^n(A)$, given by \eqref{eq:Rcap} with $w=2$, in Figure \ref{fig:penta}. As $n$ increases $\R\cap \Sigma_0^n(A)\Hto [a_{\min}, a_{\max}]$, as predicted. In this example a sharper bound than $\R\cap\Sigma_0^n(A)$ is $\R\cap \Sigma_0^n(A)\cap G$, taking the intersection of our $\tau$-method inclusions sets with classical Gershgorin.
\end{remark}

\subsubsection{The Wiener algebra case}

Suppose now that $a$ is in the Wiener algebra, so that \eqref{eq:wiener} holds. As in the previous subsection we suppose that $A=A_M$, for some $M>1$, and write $A$ in block form as \eqref{eq:block}. In contrast to the previous section $A$ does not, in general, coincide with its tridiagonal part, so its remaining part $C=A-B$ is non-zero. If, for some $w\in \N$, we choose the block form of $A$ so that $m_i\geq w$, $i=1,\ldots, N$, then $C=[c_{ij}]_{i,j=1}^M$ satisfies $c_{ij}=0$ if $|i-j|\leq w$, so that
\begin{equation} \label{eq:Cbound}
\|C\| \leq \max\{\|C\|_1,\|C\|_\infty\} \leq a_w := \sum_{j=w+1}^\infty (|a_j|+|a_{-j}|).
\end{equation}

The argument we made in the banded case to show $\Sigma_\eps^n(A)\Hto \Speps A_\infty^{+}$ does not apply here because $\widehat \sigma_\eps^n(A)$ is no longer a union over $k$ of pseudospectra of Toeplitz matrices. 
Instead, we prove a version of Theorem \ref{thm:Tgen} for the $\tau_1$ inclusion sets $\Gamma_\eps^n(A)\supset \Speps A$ given by \eqref{eq:Gamdef}. Note that the two-sided inclusions for the $\tau_1$ set $\Gamma_\eps^n(A)$, captured in Theorem \ref{thm:tau1}, play a crucial role in the proof.

\begin{theorem} \label{thm:final}
Suppose $\eps>0$ and $A=A_M$. Then $\Speps A$ $\Hto$ $\Speps A_\infty^{+}$ as $M\to\infty$. Further, $\Gamma_\eps^n(A)\Hto  \Speps A_\infty^{+}$ as $n,M,w\to\infty$, where
\begin{equation} \nonumber 
w:= \min_{1\leq i\leq N} m_i.
\end{equation}
Thus $d_H(\Gamma_\eps^n(A),\Speps A)\to 0$ as  $M,n,w\to\infty$.  If $a_{-j} = \bar a_j$, $j\in \Z$, then these results hold also for $\eps=0$, in particular $d_H(\Gamma_0^n(A),\Spec A)\to 0$ as  $M,n,w\to\infty$.
\end{theorem}
\begin{proof} Suppose that $\eps>0$. Again, the first result is just \eqref{eq:Toep}. Analogously to \eqref{eq:epsb}, we have by \eqref{eq:epstau1def} that
$$
\eps''_{n}(A) \leq \frac{2\pi \|a\|_{\mathcal{W}}}{n+1} + \|C\|,
$$
so that $\eps''_{n}(A) \to 0$ as $n,w\to\infty$, by \eqref{eq:Cbound}.
By Theorem \ref{thm:tau1}, 
$$\Speps A \subset \Gamma_{\eps}^n(A) \subset \Specn_{\eps+\eps'_n(A) + 2\|C\|} A.$$ 
Thus, arguing as in the proof of Theorem \ref{thm:Tgen}, using \eqref{eq:Toep} and \eqref{eq:Cbound}, we see that, as $n,M,w\to\infty$, $\Specn_{\eps+\eps'_n(A) + 2\|C\|} A\Hto \Speps A_\infty^{+}$  so that also $\Gamma_\eps^n(A)\Hto  \Speps A_\infty^{+}$.
If $A_\infty^{+}$ is self-adjoint then \eqref{eq:Toep} holds also for $\eps=0$, so the above arguments work also in that case. 
\end{proof}

\vspace{-2ex}

\subsection{A non-Toeplitz example} \label{sec:2T}
We finish by applying our $\tau$ method to a tridiagonal Hermitian matrix $A$ that is a perturbation, by a real diagonal matrix $D_M= \diag(d_1,\ldots,d_M)$, of the discrete Laplacian $L_M$ introduced in \S\ref{sec:disLap}, i.e., $A=D_M+L_M$, for some $M\in \N$. We choose $N=M$ so that the block structure \eqref{eq:block}, used in the construction of our $\tau$ method inclusion set $\Sigma_0^n(A)$, is trivial, each block a single complex number, as in \S\ref{sec:disLap2}. To obtain an example that has features beyond the Toeplitz case, but is simple enough that everything is known explicitly, we assume that the diagonal is $2$-periodic, i.e., $d_{i+2}=d_i$, $i=1,\ldots,M-2$, so that $A$ is $r$-Toeplitz with $r=2$ in the sense of \cite{GoverBarnett85}, i.e., $a_{i+2,j+2}=a_{ij}$, $i,j=1,\ldots,M-2$. In the case that $M$ is even, this means that $A$ is block Toeplitz, with $2\times 2$ blocks. Noting that 
$$
\Spec(\lambda I + A) = \lambda + \Specn A \quad  \mbox{and} \quad \Sigma_0^n(\lambda I + A) = \lambda + \Sigma_0^n(A), \quad n=1,\ldots, N-1,
$$ 
without loss of generality we focus on the case that, for some $\Delta\in \R$, 
$$
a_{ii}=d_i = (-1)^{i-1}\Delta, \quad i=1,\ldots,M,
$$
denoting the matrix $A$ in that case by $A_M(\Delta)$ when we want to make explicit the dependence on $M$ and $\Delta$.

The study of the eigenvalues of tridiagonal $2$-Toeplitz matrices was initiated in \cite{Gover94}. An explicit formula, for the effect of a 2-periodic diagonal perturbation on the eigenvalues of any tridiagonal matrix with zero main diagonal, was presented recently in \cite{DyachenkoTyaglov24}. To apply the results of \cite{DyachenkoTyaglov24} we note that the eigenvalues of $L_M$ are given by \eqref{eq:DLS} and, for $M\in \N$ with $M\geq 2$, let $M^+:=\lfloor M/2\rfloor$,
$$
\Specn^+ L_M := \{\lambda\in \Spec L_M:\lambda>0\}=\{2 \cos(j\pi/(
M+ 1)) : j \in \{1, \ldots, M^+\}\},
$$
and
\begin{equation} \label{eq:SN+}
S_M^+(\Delta):= \{\sqrt{\lambda^2+\Delta^2}:\lambda\in \mathrm{Spec}^+L_M \}\subset \left(|\Delta|,\sqrt{4+\Delta^2}\right).
\end{equation}
Then  \cite[Theorem 3.1]{DyachenkoTyaglov24} gives that, for $\Delta\in \R$ and $M\in \N$, with $S^+_1(\Delta):= \emptyset$,
\begin{equation} \label{eq:spec2p}
\Spec A = \Spec A_M(\Delta) = \left\{\begin{array}{ll}-S^+_M(\Delta) \cup S^+_M(\Delta), & \mbox{if $M$ is even},\\
-S^+_M(\Delta)\cup S^+_M(\Delta) \cup\{\Delta\}, & \mbox{if $M$ is odd}.\end{array}\right.
\end{equation}

It is informative to relate $\Spec A$ to the spectra of the block Laurent matrix $A_\infty(\Delta)=[a_{ij}]_{i,j\in \Z}$ and the block Toeplitz matrix $A^+_\infty(\Delta)=[a_{ij}]_{i,j\in \N}$, where $a_{ij}:=0$ if $|i-j|>1$, $:=1$ if $|i-j|=1$, and $a_{ii} := (-1)^{i-1}\Delta$, $i\in \Z$.  $A_\infty(\Delta)$ and $A^+_\infty(\Delta)$ are examples of discrete Schr\"odinger operators, so that $A_N(\Delta)$ is a finite section of a discrete Schr\"odinger operator, e.g.~\cite{Gabeletal23}. We have (see \cite[Examples 3.5(c), 3.10(a)]{Gabeletal23}) that, for $\Delta\in \R$,
\begin{equation} \label{eq:SpAinf}
\Spec A^+_\infty(\Delta)=\Spec A_\infty(\Delta) = \left[-\sqrt{4+\Delta^2},-|\Delta|\right]\cup\left[|\Delta|,\sqrt{4+\Delta^2}\right].
\end{equation}
Thus, for $\Delta\in \R$,  $\Spec A_M(\Delta) \subset \Spec A^+_\infty(\Delta)$, for $M\in \N$, and $\Spec A_M(\Delta)\Hto$ $\Spec A^+_\infty(\Delta)$ as $M\to\infty$ (cf.~Lemma \ref{lem:conv}, for the special case $\Delta=0$).

\subsubsection{The $\tau$ method inclusion sets} \label{sec:tauToep2} Let us write down the $\tau$ method inclusion sets, as defined in \S\ref{sec:tau}, for this case. For $n=1,\ldots,N-1$ and $k=0,\ldots,N-n$, we have that $B_{n,k} = A_n(\Delta)$ if $k$ is even, $=A_n(-\Delta)$ if $k$ is odd. Thus, for $N>n\geq 3$ and $\eps\geq 0$,
$$
\widehat \sigma^n_{\varepsilon}(A) = \Specn_{\varepsilon+\varepsilon_{n-2}} A_n(\Delta) \cup \Specn_{\varepsilon+\varepsilon_{n-2}} A_n(-\Delta),
$$
where  $\theta_n^*$ and $\eps_n := 4\sin(\theta_n^{*}/2)$ are defined as in the discrete Laplacian case of \S\ref{sec:disLap2}.
Since $A_n(\pm \Delta)$ are Hermitian, $\Specn_{\eps} A_n(\pm \Delta) = \Spec A_n(\pm \Delta) +\eps\overline{\D}$, 
for $n\in \N$ and $\eps\geq 0$,
and note that, by \eqref{eq:spec2p},
$$
\mathrm{Spec}\, A_n(\Delta) \cup \mathrm{Spec}\, A_n(-\Delta) = \left\{\begin{array}{ll}-S^+_n(\Delta) \cup S^+_n(\Delta), & \mbox{if $n$ is even},\\
-S^+_n(\Delta)\cup S^+_n(\Delta) \cup\{-\Delta,\Delta\}, & \mbox{if $n$ is odd},\end{array}\right.
$$
so that 
\begin{equation} \label{eq:wsig}
\widehat \sigma^n_{\varepsilon}(A) = \left\{\begin{array}{ll}-S^+_n(\Delta) \cup S^+_n(\Delta)+ (\eps+\eps_{n-2})\overline{\D}, & \mbox{if $n$ is even},\\
-S^+_n(\Delta)\cup S^+_n(\Delta) \cup\{-\Delta,\Delta\} + (\eps+\eps_{n-2})\overline{\D}, & \mbox{if $n$ is odd}.\end{array}\right.
\end{equation}
Similarly, for $N>n\geq 1$ and $\eps\geq 0$,
\begin{equation} \label{eq:sig2p}
\sigma^n_{\varepsilon}(A) = \left\{\hspace*{-1ex}\begin{array}{ll} \{\Delta\} \cup \bigcup_{m=1}^n\left(-S^+_m(\Delta) \cup S^+_m(\Delta)\right)+ (\eps+\eps_{n})\overline{\D}, & \mbox{\hspace{-1ex}$n$ even and $N$ odd},\\
 \{-\Delta,\Delta\} \cup \bigcup_{m=1}^n\left(-S^+_m(\Delta) \cup S^+_m(\Delta)\right)+ (\eps+\eps_{n})\overline{\D}, & \mbox{\hspace{-1ex}otherwise.}
\end{array}\right.
\end{equation}
Recall that, for $\eps\geq 0$, the $\tau$-method inclusion sets for $\Speps A$ are $\Sigma_\eps^n(A)$, where $\Sigma_\eps^n(A) := \sigma^n_{\varepsilon}(A)$, for $n=1,2$, $:= \widehat \sigma^n_{\varepsilon}(A) \cap \sigma^n_{\varepsilon}(A)$, for $n>2$.

Similarly to \eqref{eq:wsig}, it follows from Theorem \ref{thm:taubi} that 
$$
\Speps A_\infty \subset \Sigma_\eps^n(A_\infty) = \left\{\begin{array}{ll}-S^+_n(\Delta) \cup S^+_n(\Delta)+ (\eps+\eps_{n})\overline{\D}, & \mbox{$n$ even},\\
-S^+_n(\Delta)\cup S^+_n(\Delta) \cup\{-\Delta,\Delta\} + (\eps+\eps_{n})\overline{\D}, & \mbox{$n$ odd}.\end{array}\right.
$$
Note that $\Sigma_\eps^n(A_\infty)\subset \Sigma_\eps^n(A)$, for $\eps\geq 0$ and $N>n\geq 2$, so that, recalling the characterisation above of $\Spec A_\infty$ and \eqref{eq:SN+},
$$
-S^+_m(\Delta) \cup S^+_m(\Delta) \subset \left[-\sqrt{4+\Delta^2},-|\Delta|\right]\cup\left[|\Delta|,\sqrt{4+\Delta^2}\right] \subset \Sigma_0^n(A),
$$
for $1\leq m\leq n$. From these inclusions, arguing as in the proof of \eqref{eq:RSig}, we get that, for $N>n\geq 1$ and $\Delta\in \R$,
\begin{equation} \label{eq:Rcap2p}
\R \cap \Sigma_0^n(A) = -\widehat S^+_{n,N}(-\Delta) \cup \widehat S^+_{n,N}(\Delta),
\end{equation}
where
$$
\widehat S^+_{n,N}(\Delta) := \left[c^-_{n,N}(\Delta),c^+_n(\Delta)\right] \supset \left[|\Delta|,\sqrt{4+\Delta^2}\right],
$$
with $c_1^+(\Delta):= |\Delta|+\eps_1$, $c_n^+(\Delta) := \max S_n^+(|\Delta|) +\eps_n$, for $n\geq 2$, and
$$
c_{n,N}^-(\Delta)  :=  \left\{\hspace{-1ex}\begin{array}{ll} |\Delta|-\eps_n, & \hspace{-1ex}\mbox{if $n$ is odd or}\\ & \hspace{-1ex}\mbox{$n=2$ and ($N$ even or $\Delta>0$),}\\ 
\min S^+_n(|\Delta|)-\eps_{n}, & \hspace{-1ex}\mbox{if $n$ even, $N$ odd, and $\Delta\leq 0$,}\\
\max(\min S^+_n(|\Delta|)-\eps_{n-2},|\Delta|-\eps_n), & \hspace{-1ex}\mbox{otherwise.}\\
\end{array}\right.
$$
Note that
$\max S_n^+(\Delta)= \sqrt{4\cos^2(\pi/(n+1))+\Delta^2}$  and, in the case that $n$ is even, $\min S^+_n(\Delta)= \sqrt{4\cos^2(n\pi/(2n+2))+\Delta^2}$, so that, as $n\to\infty$,
\begin{equation} \label{eq:cnasymp}
c_n^+(\Delta) = \sqrt{4+\Delta^2} + \frac{2\pi}{n} + O(n^{-2}), \quad c_{n,N}^-(\Delta) = |\Delta| - \frac{2\pi}{n} + O(n^{-2}).
\end{equation}

\vspace{-2ex}

\subsubsection{Comparison with Gershgorin} \label{sec:T2G} $\R\cap \Sigma_0^n(A)$, given by \eqref{eq:Rcap2p}, is the family of $\tau$-method inclusion sets. Since, as discussed above, $\Spec A \Hto -\widehat S^+_\infty(|\Delta|)\cup \widehat S^+_\infty(|\Delta|)$ as $M\to\infty$, where 
$$
\widehat S^+_\infty(\Delta) := [\Delta, \sqrt{4+\Delta^2}], \qquad \Delta \geq 0,
$$ 
the above asymptotics for $c_n^+(\Delta)$ and $c_{n,N}^-(\Delta)$ make clear that $\R\cap \Sigma_0^n(A)$ is a sharp inclusion set for $\Spec A$ in the limit $M,n\to\infty$. By contrast, where $G$ is the Gershgorin theorem inclusion set  \eqref{eq:ger},
\begin{equation} \label{eq:ger2p}
\R\cap G = \{-\Delta,\Delta\} + [-2,2],
\end{equation}
if $M\geq 3$.
This is a sharp bound for $\Spec A$ only in the discrete Laplacian case $\Delta=0$. In particular, $\Spec A \cap (-|\Delta|,|\Delta|)=\emptyset$, for $\Delta\neq 0$, while $(-|\Delta|,|\Delta|)\subset \R\cap G$ for $|\Delta|\leq 2$, $(-|\Delta|+2] \cup [-2+|\Delta|)\subset \R\cap G$ for $|\Delta|\geq 2$.

As we did in Remark \ref{rem:bmG} for $\Delta=0$, let us make comparison also with the block matrix Gershgorin theorem, discussed in Remark \ref{rem:case1}. Recall that $A=A_M(\Delta)$, for some $M>1$ and $\Delta\in \R$, and write $A$ in the block form \eqref{eq:block} for some $1<N\leq M$, so that  $a_{ij}\in \C^{m_i\times m_j}$, for $i,j=1,\ldots, N$. As in Remark \ref{rem:bmG}, we compute the inclusion set \eqref{eq:salas2} given by the block-matrix Gershgorin when the matrix norms are the $2$-norm. $A$ is block-tridiagonal, with $a_{ii} = A_{m_i}((-1)^{r_i}\Delta)$, for $i=1,\ldots,N$, where $r_1=0$ and $r_i= \sum_{j=1}^{i-1}m_j$, $i>1$. Each non-zero off-diagonal block $a_{ij}$ has a single non-zero entry with value 1, so $\|a_{i,i+1}\|=\|a_{i+1,i}\|=1$, $i=1,\ldots,N-1$. Thus \eqref{eq:salas2} reads in this case (cf.~Remark \ref{rem:bmG})
\begin{eqnarray*}
\Spec A\ \subset\ G^{N} 
&   = &\left(\overline{\D}+(\Spec a_{1,1} \cup \Spec a_{N,N})\right) \cup \left(2\overline{\D} +\bigcup_{k=2}^{N-1} \Spec a_{k,k}\right).
\end{eqnarray*}
As noted for the case $\Delta=0$ in Remark \ref{rem:bmG}, this seems no better than the inclusion set $G$ provided by the standard Gershgorin theorem, which it reduces to if $N=M$. In particular, as  $M\to\infty$ with $N>2$ and  $\min_{1\leq k\leq N} m_k\to\infty$, it is clear from \eqref{eq:SN+} and \eqref{eq:spec2p} that 
$$
G^{N}\Hto -\widehat S^+_\infty(|\Delta|)\cup \widehat S^+_\infty(|\Delta|) + 2\overline{\D} \quad \mbox{so} \quad \R\cap G^{N}\Hto -\widehat S^+_\infty(|\Delta|)\cup \widehat S^+_\infty(|\Delta|) + [-2,2].
$$
For example, if $\Delta=1.5$, $\Spec A \Hto [-2.5,-1.5]\cup [1.5,2.5]$ as $M\to\infty$, the $\tau$ method inclusion set $\Sigma_0^n(A)$ tends to the same limit as $n\to\infty$, while the $\min_{1\leq k\leq N} m_k\to\infty$ limit of the block-Gershgorin inclusion set is $-\widehat S^+_\infty(|\Delta|)\cup \widehat S^+_\infty(|\Delta|) + [-2,2]=[-4.5,4.5]$ and the standard Gershgorin inclusion is
$\R\cap G = [-3.5,3.5]$.
\begin{figure}[t]
\begin{center}
\includegraphics[width=97mm]{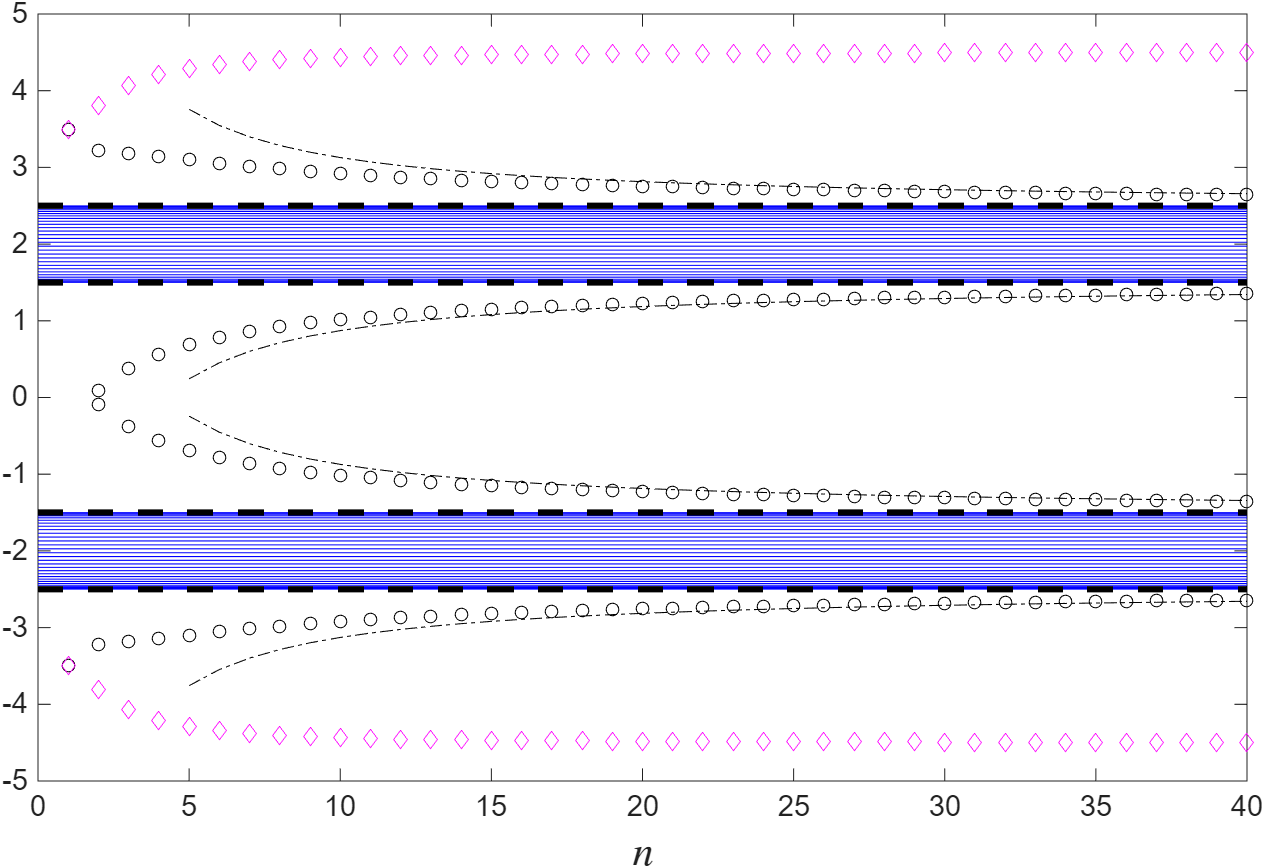} 
\end{center}

\vspace{-4ex}

\caption{$\Spec A$, for $A=A_M(\Delta)$ with $\Delta=1.5$ and $M=60$; the Gershgorin inclusion $\R\cap G$ and its block matrix versions; and the $\tau$-method inclusions sets, $\R\cap \Sigma_0^n(A)$. The horizontal solid lines  are the eigenvalues of $A$ and the dashed lines the boundary of $\Spec A^+_\infty(\Delta)=[-3.5,-1.5]\cup [1.5,3.5] $. The circles are $\pm c_n^+(\pm\Delta)$ (top and bottom lines) and $\pm c_{n,M}^-(\pm\Delta)$ (middle two lines); for each $M>n$, $\Spec A \subset\R \cap \Sigma_0^n(A) = \left[-c^+_{n}(-\Delta),-c^-_{n,M}(-\Delta)\right] \cup \left[c^-_{n,M}(\Delta),c^+_n(\Delta)\right]$; see \S\ref{sec:tauToep2}. The values of $c^-_{1,M}(\pm\Delta)$ (not plotted) are negative, so  $\R \cap \Sigma_0^1(A)=[-c^+_1(-\Delta),c^+_1(\Delta)]=[-3.5,3.5]=\R\cap G$, where $G$ is the Gershgorin inclusion set \eqref{eq:ger}.   The magenta diamonds are $\pm B_n$, where $B_n := \max S_n^+(\Delta)+2$. By \eqref{eq:GNR}, $\R\cap G^{N}=[-B_n,B_n]$ for this value of $\Delta$, where $G^{N}$ is the block matrix Gershgorin inclusion set \eqref{eq:salas} with blocks of size $n$. The dash-dot lines are the first two terms of the asymptotic expansions \eqref{eq:cnasymp} for $\pm c_n^{+}(\pm\Delta)$ and $\pm c_{n,M}^{-}(\pm\Delta)$.}  \label{fig:Toep2}
\end{figure}

In more detail, consider the case that $n\in \N$ divides $M$ and $M\geq 3$, and set $N=M/n$ and $m_i=n$, $i=1,\ldots, N$. Then $\R\cap G^{N}=\R\cap G=\{-\Delta,\Delta\}+[-2,2]$ if $n=1$. If $n\geq 2$ and $N>2$ it follows from the above representation for $\R\cap G^{N}$ and \eqref{eq:spec2p} that, where $\Spec A^+_\infty$ is given by \eqref{eq:SpAinf}, $\Spec A^+_\infty\subset G^{N}$. Thus, arguing as we did to get \eqref{eq:RcapGN}, and restricting attention to the case $\Delta\geq 0$ and $N>3$, we see that
\begin{equation} \label{eq:GNR}
\R\cap G^{N} = -G_n^r \cup G_n^r
\end{equation}
where
$$
G^r_n := \left\{ 
\begin{array}{ll}  \left[\min S_n^+(\Delta)-2,\max S_n^+(\Delta)+2\right], & n \mbox{ even},\\
\left[\Delta-2,\max S_n^+(\Delta)+2\right], & n \mbox{ odd},
\end{array}
\right.
$$

In Figure \ref{fig:Toep2} we show these inclusion sets for $\Spec A=\Spec A_M(\Delta)$ in the case $\Delta=1.5$, plotting $\R\cap \Sigma_0^n(A)$, given by \eqref{eq:Rcap2p}, and $\R\cap G^{N}$, given by \eqref{eq:GNR}, against $n$. These are inclusion sets for $\Spec A$ provided $n<M$ for $\R\cap \Sigma_0^n(A)$ (Theorem \ref{thm:tau}), provided $N=M/n$ is an integer $\geq 4$ for $\R\cap G^{N}$ (see above discussion). For fixed $n$,  $\R\cap G^{N}=-G_n^r\cap G_n^r$ and $\R\cap\Sigma_0^n(A)$ are both based on computing eigenvalues of principal submatrices of $A$ of order $n$. But, while  $\R\cap\Sigma_0^n(A)$ is monotonically decreasing, converging to a sharp inclusion set for $\Spec A$, $\R\cap G^{N}$ becomes larger as $n$ increases, approaching the limit $[-4.5,4.5]$.

\section{Conclusion and possible extensions} \label{sec:concl}

In this paper we have, extending recent bi-infinite matrix results \cite{CW.Heng.ML:SpecIncl1}, derived sequences of inclusion sets (our $\tau$ and $\tau_1$ methods) for the spectra and pseudospectra of a finite matrix $A$ (Theorems \ref{thm:tau} and \ref{thm:tau1}). Each inclusion set in each sequence is expressed as a union of pseudospectra of finite submatrices of what we term the tridiagonal part of $A$ when $A$ is written in the block form \eqref{eq:block},  square submatrices for the $\tau$ method, rectangular for the $\tau_1$ method. In \S\ref{sec:examples} we have explored the sharpness of these inclusion sets for the case when the finite matrix is a large Toeplitz matrix, showing that both sequences produce sharp inclusion sets for the pseudospectrum of $A$ 
  in the case that $A$ is large (if $A$ is banded for the $\tau$ method, has symbol in the Wiener class for the $\tau_1$ method). If $A$ is Hermitian, then our inclusion sets also provide sharp inclusions for the spectrum. Importantly, the inclusions in Theorem \ref{thm:tau1} are two-sided, associated with the use of rectangular rather than square submatrices in the $\tau_1$ method. This is a key proof ingredient for the $\tau_1$-method results for general Toeplitz matrices. 
  
In \S\ref{sec:2T} we have extended our convergence results for the $\tau$-method to a class of tridiagonal 2-Toeplitz Hermitian matrices, and used this class and a simple Toeplitz example (Remark \ref{rem:bmG}) to demonstrate that our $\tau$-method sequence can be superior to existing block-matrix Gershgorin generalisations. Where $G$ is the classical Gershgorin set, we have also shown for the classes of large Hermitian matrices studied that, while $\R\cap G$ can be a sharp spectrum bound in particular cases, our $\tau$ and $\tau_1$ sequences generate sharp inclusions  in every case (Remark \ref{rem:BandG}, \S\ref{sec:T2G}).

There are many directions for further work.  Theorems \ref{thm:tau} and \ref{thm:tau1}, and Theorem \ref{thm:taubi} and \cite{CW.Heng.ML:SpecIncl1}, provide inclusion sets for finite matrices and for bi-infinite matrices, respectively. In work in progress we are constructing similar inclusion set sequences  for the semi-infinite case (exemplified by the infinite Toeplitz matrix $A_\infty^{+}$ in \S\ref{sec:bandT}). 
Our results for that case, like the results in this paper for the finite matrix case, depend on inclusion sets from \cite{CW.Heng.ML:SpecIncl1} for the bi-infinite case,  extended in \S\ref{sec:bi}. The arguments  for the bi-infinite case use, implicitly, that $\Z$ is a group under addition. In other work in progress, with Christian Seifert, we are extending these arguments to general groups;  while our results in \cite{CW.Heng.ML:SpecIncl1} apply for bi-infinite matrices acting on $\ell^2(\Z)$, our new results apply to matrices acting on $\ell^2(G)$, for some Abelian group $G$. When $G$ is finite, this provides a route to spectral inclusion sets for certain finite matrices, complementing the work in this paper. 

Other possible directions for future research include:
\begin{enumerate}
\item Work on the (efficient) implementation of our inclusion sets for general finite matrices, the development of associated open source codes, and large-scale computational experiments investigating the sharpness of our inclusion sets, and the relative merits of our $\tau$ and $\tau_1$ families.
\item Complementing 1, further theoretical studies related to the sharpness of our inclusion sets for large matrices. A key tool in our \S\ref{sec:bandT}  investigations for large Toeplitz matrices was \eqref{eq:Toep}, capturing the asymptotics of pseudospectra for large matrix size. Another matrix class, of interest for applications in mathematical physics (see, e.g., \cite[\S VIII]{TrefEmbBook} and \cite[\S8.4]{CW.Heng.ML:SpecIncl1}) and where we understand the asymptotics of pseudospectra, is the class of random tridiagonal matrices (see \cite[Remark 4.17]{CWLi2016:Coburn}). 
\item Our results in \cite{CW.Heng.ML:SpecIncl1} and in Theorem \ref{thm:taubi} for bi-infinite matrices allow operator-valued matrix entries. Such matrices arise, for example,  in the study of integral operators on $L^2(\R)$ via a standard {\em discretisation} (e.g., \cite[{\S}1.2.3]{LiBook}) which replaces the operator on $L^2(\R)$ by a unitarily equivalent operator on $\ell^2(\Z,L^2[0,1])$, to which Theorem \ref{thm:taubi} applies.  Similarly, for any $N\in \N$, an integral operator on $L^2[a,b]$, for some finite interval $[a,b]$, is unitarily equivalent to an operator on $\ell^2(\I,L^2[0,h])$, where $\I=\{1,\ldots,N\}$ and $h=(b-a)/N$, to which versions of Theorems \ref{thm:tau} and \ref{thm:tau1} could be applied that allowed operator-valued  entries in \eqref{eq:block}. Extensions to such cases  should be relatively straightforward, starting from Theorem \ref{thm:taubi} and the results in \cite{CW.Heng.ML:SpecIncl1}. Similarly, the block-matrix versions of Gershgorin's theorem, \eqref{eq:salas} and \eqref{eq:salas2}, have been extended to matrices with operator-valued entries in \cite{Salas}.
\end{enumerate}

\paragraph{Data availability} The Matlab codes used to produce the data displayed in Figures \ref{fig:jordan}-\ref{fig:Toep2}, and to produce the figures, are provided in the supplementary materials.

\newpage
\setcounter{figure}{0}
\renewcommand{\thefigure}{\arabic{figure}}
\setcounter{page}{1}
\setcounter{section}{0}
\renewcommand{\thesection}{S\arabic{section}}
\renewcommand{\theHsection}{S\arabic{section}}
\begin{center}
{\bf \large Supplementary materials to}\\
\vspace{1ex}
{\bf  \Large `Gershgorin-Type Spectral Inclusions for Matrices'}\\
\vspace{2ex}

{\sc \large Simon N. Chandler-Wilde\symbolfootnote{Department of Mathematics and Statistics, University of Reading, Reading, RG6 6AX, UK. Email: \href{mailto:s.n.chandler-wilde@reading.ac.uk}{\tt s.n.chandler-wilde@reading.ac.uk}}}
\quad and\quad {\sc \large Marko Lindner\symbolfootnote{Institut Mathematik, TU Hamburg (TUHH), D-21073 Hamburg, Germany. Email: \href{mailto:lindner@tuhh.de}{\tt lindner@tuhh.de}}}\\
\vspace{2ex}
August 2025
\end{center}

\noindent {\bf \large Table of Contents}\\
\\
\hyperref[sec:intros]{\bf \bl{S1. Introduction}}\\
\hyperref[sec:script]{\bf \bl{S2. Matlab files used to produce the figures in \cite{CWL}}}\\
\hspace*{3ex} \hyperref[sec:Jordans]{\bl{\it S2.1. The Jordan block example}}\\
\hspace*{6ex} \hyperref[code:JB]{\bl{\it S2.1.1. JordanBlock.m}}\\
\hspace*{3ex} \hyperref[sec:penm]{\bl{\it S2.2. A Toeplitz pentadiagonal example}}\\
\hspace*{6ex} \hyperref[code:fig2]{\bl{\it S2.2.1. figure\_2\_in\_section\_5.m}}\\
\hspace*{6ex} \hyperref[code:pentau]{\bl{\it S2.2.2. penta\_tau.m and its test script file test\_penta\_tau.m}}\\
\hspace*{6ex} \hyperref[code:ptm]{\bl{\it S2.2.3. penta\_tau\_method.m}}\\
\hspace*{6ex} \hyperref[code:penta]{\bl{\it S2.2.4. penta.m}}\\
\hspace*{6ex} \hyperref[code:theta]{\bl{\it S2.2.5. theta.m}}\\
\hspace*{3ex} \hyperref[sec:Toep2]{\bl{\it S2.3. A 2-Toeplitz tridiagonal example}}\\
\hspace*{6ex} \hyperref[code:fig3]{\bl{\it S2.3.1. figure\_3\_in\_section\_5.m}}\\
\hspace*{6ex} \hyperref[code:climits]{\bl{\it S2.3.2. climits.m}}\\
\hspace*{6ex} \hyperref[code:Toep2e]{\bl{\it S2.3.3. Toep2eigs.m}}\\
\hyperref[sec:Herm]{\bf \bl{S3. A Matlab function to compute the $\tau$-method inclusion sets in the}}\\
\hspace*{4ex} \hyperref[sec:Herm]{\bf \bl{Hermitian case}}\\
\hspace*{3ex} \hyperref[sec:Mfts]{\bl{\it S3.1. Matlab functions and test scripts}}\\
\hspace*{6ex} \hyperref[code:tm]{\bl{\it S3.1.1. tau\_method.m}}\\
\hspace*{6ex} \hyperref[code:es]{\bl{\it S3.1.2. eigs\_sigma.m}}\\
\hspace*{6ex} \hyperref[code:int]{\bl{\it S3.1.3. mergeIntervals.m and intersectIntervals.m}}\\
\hspace*{6ex} \hyperref[code:Toep3]{\bl{\it S3.1.4. ThreeToeplitzExample.m and Toep3\_tau\_method.m}}\\
\hyperref[sec:biblios]{\bl{\bf References}}

\section{Introduction} \label{sec:intros}

This pdf document is the main part of the supplementary materials\symbolfootnote{The other part of the supplementary materials is a zipped folder containing the various Matlab functions and script files listed in this pdf.} to our recent paper \cite{CWL}. It comprises, in \S\ref{sec:script}, printouts and documentation for the Matlab script files and associated functions that were used to produce the figures in \cite{CWL}, and examples of using these script files to produce figures for other parameter values. In \S\ref{sec:Herm} we present, and give examples of the use of, Matlab code to compute $\R\cap \Sigma_0^n(A)$ in the case that $A$ is Hermitian.  Here $\Sigma_0^n(A)$ is the $\tau$-method family of inclusions sets for $\Spec A$, presented in \cite[\S1.1]{CWL}. In the case that  $A$ is Hermitian, $\Spec A$ is real so that $\Spec A\subset \R\cap \Sigma_0^n(A)$. 

Throughout, we include Matlab script files that test our implementations and the algorithms behind these implementations. In particular, we compare our hand calculations of $\Spec A\cap \Sigma_0^n(A)$ for particular Hermitian matrices in \cite{CWL} (see \cite[(5.40), (5.47)]{CWL}) with the results of our general purpose code from \S\ref{sec:Herm}, achieving agreement to within machine precision; see \S\ref{sec:penm} and \S\ref{sec:Toep2}. This, of course, is suggestive that our calculations leading to  \cite[(5.40), (5.47)]{CWL} are accurate, and is suggestive that our general code is correct.

\section{Matlab files used to produce the figures in \cite{CWL}} \label{sec:script}

\subsection{The Jordan block example} \label{sec:Jordans}

Figure 5.1 from \cite{CWL} is reproduced here as Figure \ref{fig:jordans}. Running the Matlab script file \verb+JordanBlock+, reproduced in \S\ref{code:JB}, produces as outputs the figure files \verb+Sigma.png+ and \verb+Gamma.png+ that are the left and right panels in Figure \ref{fig:jordans}. It also produces as output the radii of the three  circles referenced in the caption of Figure \ref{fig:jordans}. As explained in the comments to \verb+JordanBlock.m+, results for other values of the parameters $n\in \N$ and $\eps\geq 0$ can be obtained by changing the values assigned to \verb+n+ and \verb+epsilon+. For example, setting  \verb+n = 2+ and \verb+epsilon = 0+ (so that $n=2$ and $\eps=0$) gives the results in Figure \ref{fig:jordanE}.

\begin{figure}[t]
\[
\begin{tabular}{cc}
\includegraphics[width=55mm]{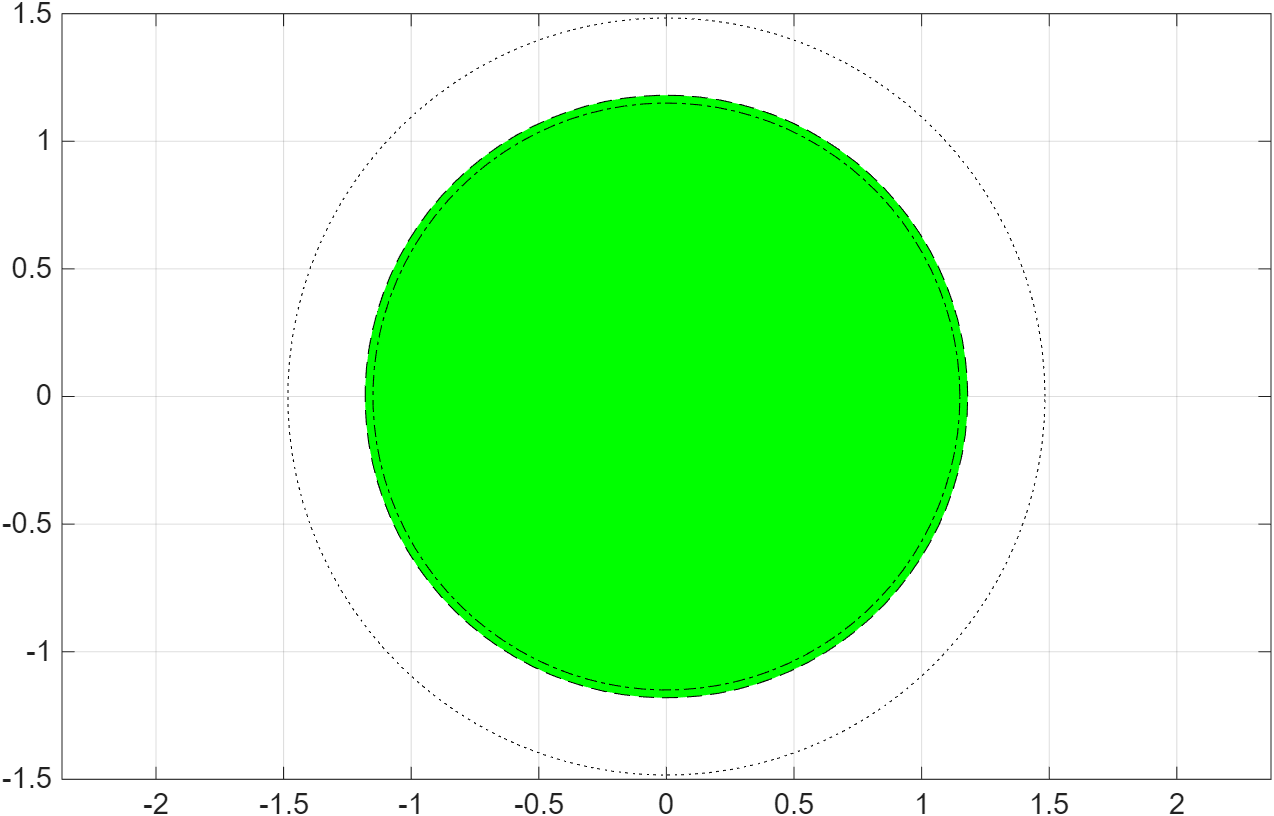} &
\includegraphics[width=55mm]{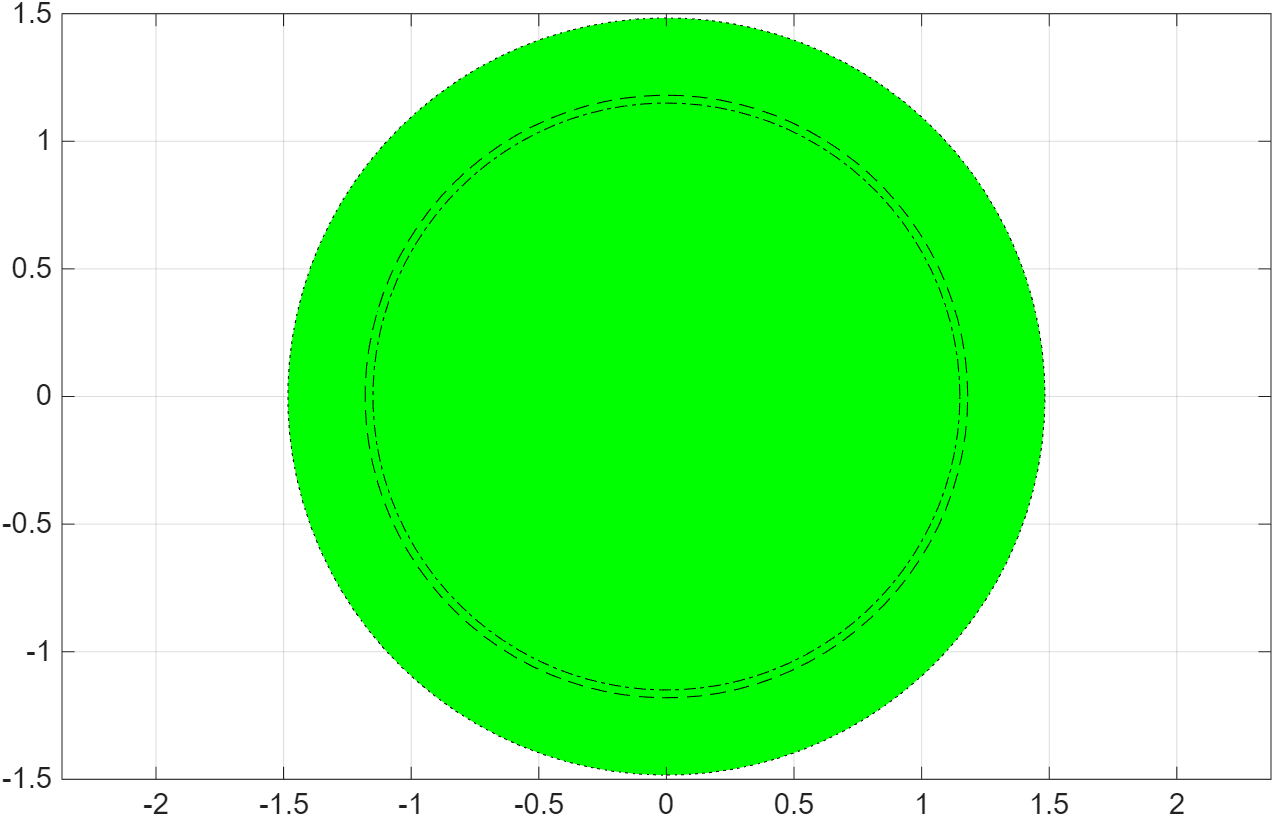}\\
$\Sigma^n_{\eps}(V_N)$ & $\Gamma^n_{\eps}(V_N)$
\end{tabular}
\]
\caption{The $\tau$ and $\tau_1$ inclusion sets, $\Sigma^n_{\eps}(V_N)$ and $\Gamma^n_{\eps}(V_N)$, for $n=4$ and $\eps=0.15$.
Each is an inclusion set for $\Speps V_N$ if $N\geq n$. Also shown in each panel are circles of radii $1+\eps=1.15$ ($- .$), $\alpha_n(\eps_n+\eps)\approx 1.18$ ($-\,-$), and $\alpha_n(\bl{\eps'_n}+\eps)\approx 1.48$ ($\cdots$). The circles of radii $\alpha_n(\eps_n+\eps)$ and $\alpha_n(\bl{\eps'_n}+\eps)$ are the boundaries of $\Sigma^n_{\eps}(V_N)$ and $\Gamma^n_{\eps}(V_N)$. } \label{fig:jordans}
\end{figure}

\begin{figure}[h]
\[
\begin{tabular}{cc}
\includegraphics[width=55mm]{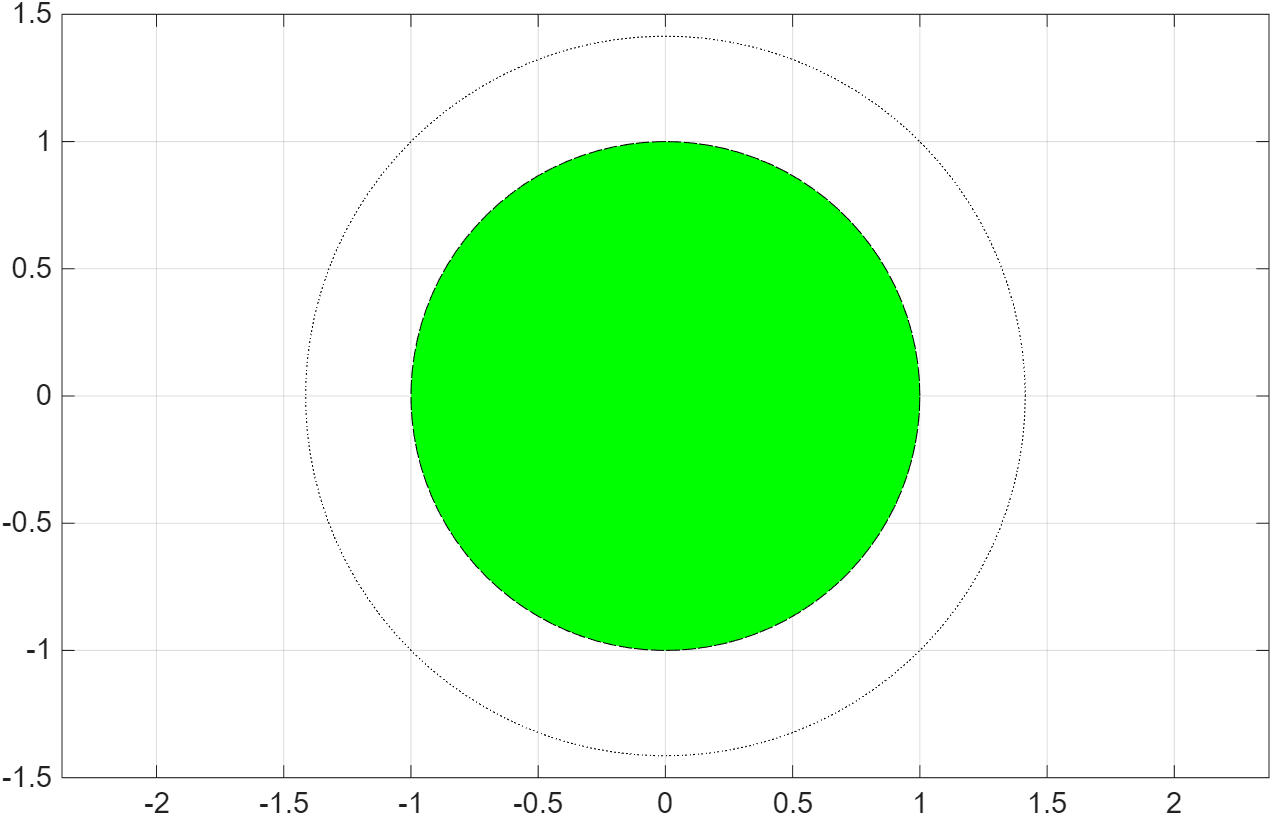} &
\includegraphics[width=55mm]{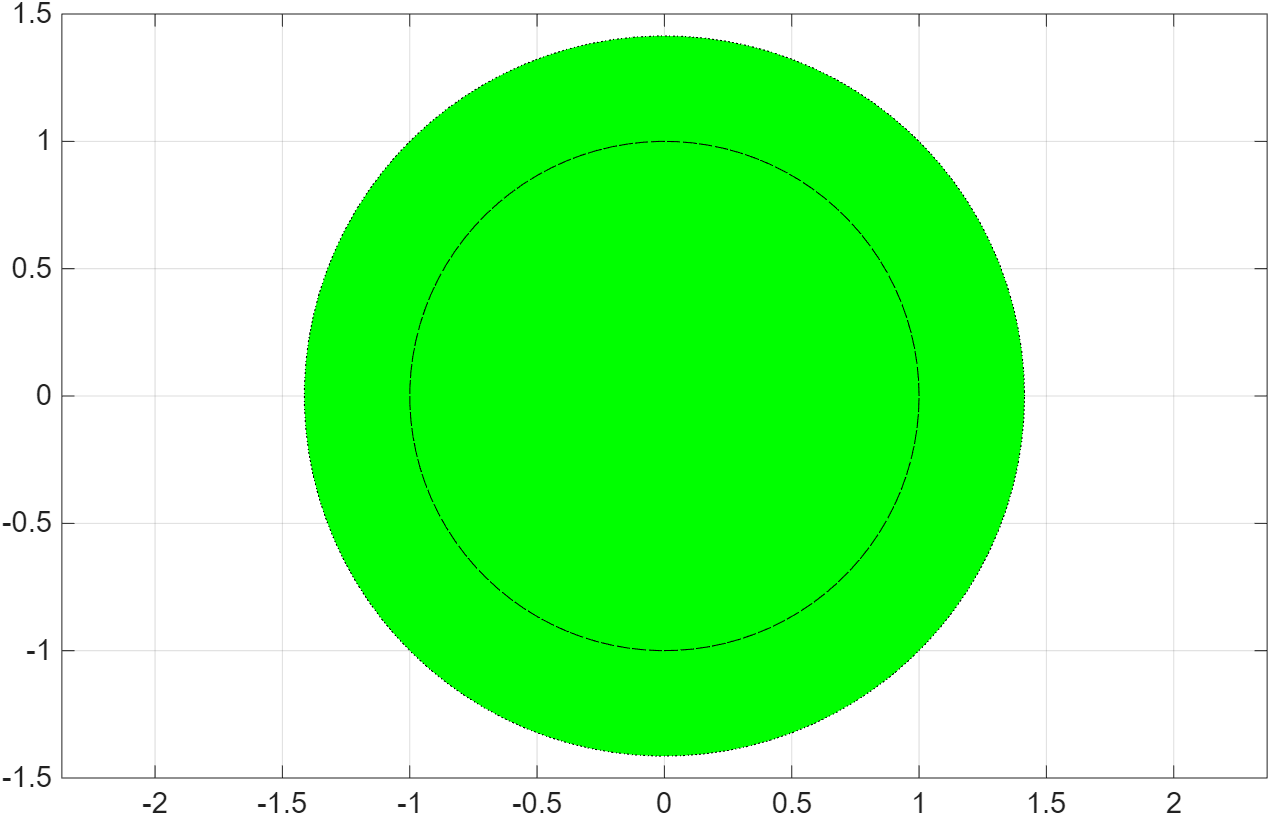}\\
$\Sigma^n_{\eps}(V_N)$ & $\Gamma^n_{\eps}(V_N)$
\end{tabular}
\]
\caption{The $\tau$, $\tau_1$, and $\pi$ inclusion sets, $\Sigma^n_{\eps}(V_N)$, $\Gamma^n_{\eps}(V_N)$, and $\Pi^{n,t}_{\eps}(V_N)$, respectively, for $n=2$ and $\eps=0$.
Each is an inclusion set for $\Speps V_N$ if $N\geq n$. Also shown in each panel are circles of radii $1+\eps=1$ ($- .$), $\alpha_n(\eps_n+\eps)=1$ ($-\,-$), and $\alpha_n(\eps'_n+\eps)\approx 1.41$ ($\cdots$). The circles of radii $\alpha_n(\eps_n+\eps)$ and $\alpha_n(\eps'_n+\eps)$ are the boundaries of $\Sigma^n_{\eps}(V_N)$ and $\Gamma^n_{\eps}(V_N)$, respectively, and the two innermost circles coincide for this value of $\eps$.} \label{fig:jordanE}
\end{figure}

\subsubsection{JordanBlock.m} \label{code:JB}

\begin{verbatim}
% Script file JordanBlock.m, that produces Figure 5.1 in [1].
%
% [1] S. N. Chandler-Wilde and M. Lindner, Gershgorin-type inclusion sets
%     for matrices. Submitted for publication, 2024.
%
% The file produces plots of the tau and tau_1 inclusion sets of 
% [1] for a matrix A that is a Jordan block of order N. See the caption
% of [1,Figure 5.1] for more details.
%
% The following lines set values for the parameters n and epsilon
% that define these inclusion sets. The parameter values below, 
% i.e., n = 4 and epsilon = 0.15, are the parameter values used 
% in [1,Figure 5.1]. Figures for different choices of these 
% parameter values can be obtained by changing the values 4 and 0.15 
% in the next two lines (any positive integer n and value
% of epsilon are allowable). Depending on the choices of n and epsilon 
% the value of axiswidth, set and described below, may need to change.
disp('Display the values of the parameters n and epsilon:')
n = 4
epsilon = 0.15
% Compute epsilon_n and epsilon'_n, as defined in [1,Section 5.1.1].
epsn = 2*sin(pi/(4*n+2));
epsndash = 2*sin(pi/(2*n+2));
epsiSig = epsn + epsilon;
epsiGam = epsndash + epsilon;
% Where the function alpha_n is as defined in [1,Section 5.1.1] (and see
% function alphan below), compute the values of alpha_n(epsi) for epsi =
% epsiSig and epsiGam
alSig = alphan(n,epsiSig);
alGam = alphan(n,epsiGam);
% Compute the coordinates x and y of points on the unit circle.
theta = linspace(0,2*pi);
x = cos(theta); y = sin(theta);
% Set the axes: the figures will display at least the part of the complex
% plane with real and imaginary parts in the range [-axiswidth,axiswidth].
axiswidth = 1.5;
figure(1)
% This is the plot of the tau method inclusion set, the left-hand plot in
% [1, Figure 5.1].
h = fill(alSig*x,alSig*y,'g');
set(h,'EdgeColor','none');
hold on
% Add to the plot three circles, as per the caption of [1, Figure 5.1].
disp('Display the radii of the three circles plotted in [1, Figure 5.1],')
disp('namely rmin, alSig, and alGam. For detail see the caption of')
disp('[1, Figure 5.1].')
rmin = 1 + epsilon 
alSig
alGam
plot(rmin*x,rmin*y,'k-.')
plot(alSig*x,alSig*y,'k--')
plot(alGam*x,alGam*y,'k:')
hold off
axis([-axiswidth,axiswidth,-axiswidth,axiswidth])
axis equal
grid
ax = gca; ax.FontSize = 14;
% Export the plot to file Sigma.png in the current directory.
exportgraphics(gcf,'Sigma.png');

figure(2)
% This is the plot of the tau_1 method inclusion set, the right hand plot in
% [1, Figure 5.1].
h = fill(alGam*x,alGam*y,'g');
set(h,'EdgeColor','none');
hold on
% Add to the plot three circles, as per the caption of [1, Figure 5.1].
plot(rmin*x,rmin*y,'k-.')
plot(alSig*x,alSig*y,'k--')
plot(alGam*x,alGam*y,'k:')
hold off
axis([-axiswidth,axiswidth,-axiswidth,axiswidth])
axis equal
grid
ax = gca; ax.FontSize = 14;
% Export the plot to file Gamma.png in the current directory.
exportgraphics(gcf,'Gamma.png');

function p = phin(n,s)
%
% function p = phin(n,s)
%
% Given scalar inputs s >=1 and positive integer n, this function computes
% p, the unique solution in the range [pi/(2*n+1),pi/(n+1)] of the equation
% s*sin((n+1)*t) = sin(n*t), this the value of phi_n(s) as defined below
% Equation (5.6) in [1].
%
p = fzero(@(t) s*sin((n+1)*t)-sin(n*t),[pi/(2*n+1),pi/(n+1)]);
end

function v = vn(n,s)
%
% function v = vn(n,s)
%
% Given scalar inputs s>=1 and positive integer n, this function computes
% v > 0, the value of the function v_n(s), defined in [1, Section 5.1.1],
% notably [1, Equation (5.6)].
%
v = sqrt(1+s^2-2*s*cos(phin(n,s)));
end

function a = alphan(n,epsi)
%
% function a = alphan(n,epsi)
%
% Given scalar inputs, namely the positive integer n and epsi >= eps_n := 
% 2sin(pi/(4n+2)), this function computes a >= 1, the unique solution of
% v_n(s) = epsi, where v_n(s) is defined for s >= 1 by [1, Equation (5.6)],
% and see function vn.
%
epsn = 2*sin(pi/(4*n+2));
ep = epsi - epsn;
if ep < 0
    disp(' ')
    disp('Error: the input value of epsi to function alphan is < epsn')
    disp(' ')
end
lower = 1+ep;
upper = lower + 2*epsn*ep/(ep+sqrt(2*epsn*ep+ep^2));
a = lower;
if ep > 0
    a = fzero(@(s) vn(n,s)-epsi,[lower,upper]);
end
end
\end{verbatim}

\subsection{A Toeplitz pentadiagonal example} \label{sec:penm}

Figure 5.2 from \cite{CWL} is reproduced here as Figure \ref{fig:pentas}. Running  the Matlab script file \verb+figure_2_in_section_5.m+,  reproduced in \S\ref{code:fig2}, produces as output the figure file \verb+penta.png+ shown in this figure. This figure plots the intersection with the real line of the $\tau$-method inclusions sets $\Sigma_0^n(A)$, for $n=1,\ldots,50$, given by \cite[Equation (5.40)]{CWL}, in the case when $A$ is the order $M=110$ pentadiagonal Toeplitz matrix with a zero main diagonal and 1's on the first two sub- and superdiagonals. 

This script file depends on subsidiary Matlab functions, \verb+penta_tau.m+, \verb+penta.m+, and \verb+theta.m+, which are listed below, in \S\ref{code:pentau}, \S\ref{code:penta}, and \S\ref{code:theta}, respectively, the first of these functions carrying out the actual computation of $\R \cap \Sigma_0^n(A)$.  Section \ref{code:pentau} also contains a test script file \verb+test_penta_tau.m+. This checks the computation of \verb+penta_tau+, that uses the formula \cite[Equation (5.40)]{CWL}, by comparing its output with that of the function  \verb+penta_tau_method+ in \S\ref{code:ptm}, which computes using the general purpose Hermitian $\tau$-method code presented below in \S\ref{sec:Herm}. The results of this testing provide strong support for the accuracy of  \cite[Equation (5.40)]{CWL}.

\begin{figure}[t]
\begin{center}
\includegraphics[width=90mm]{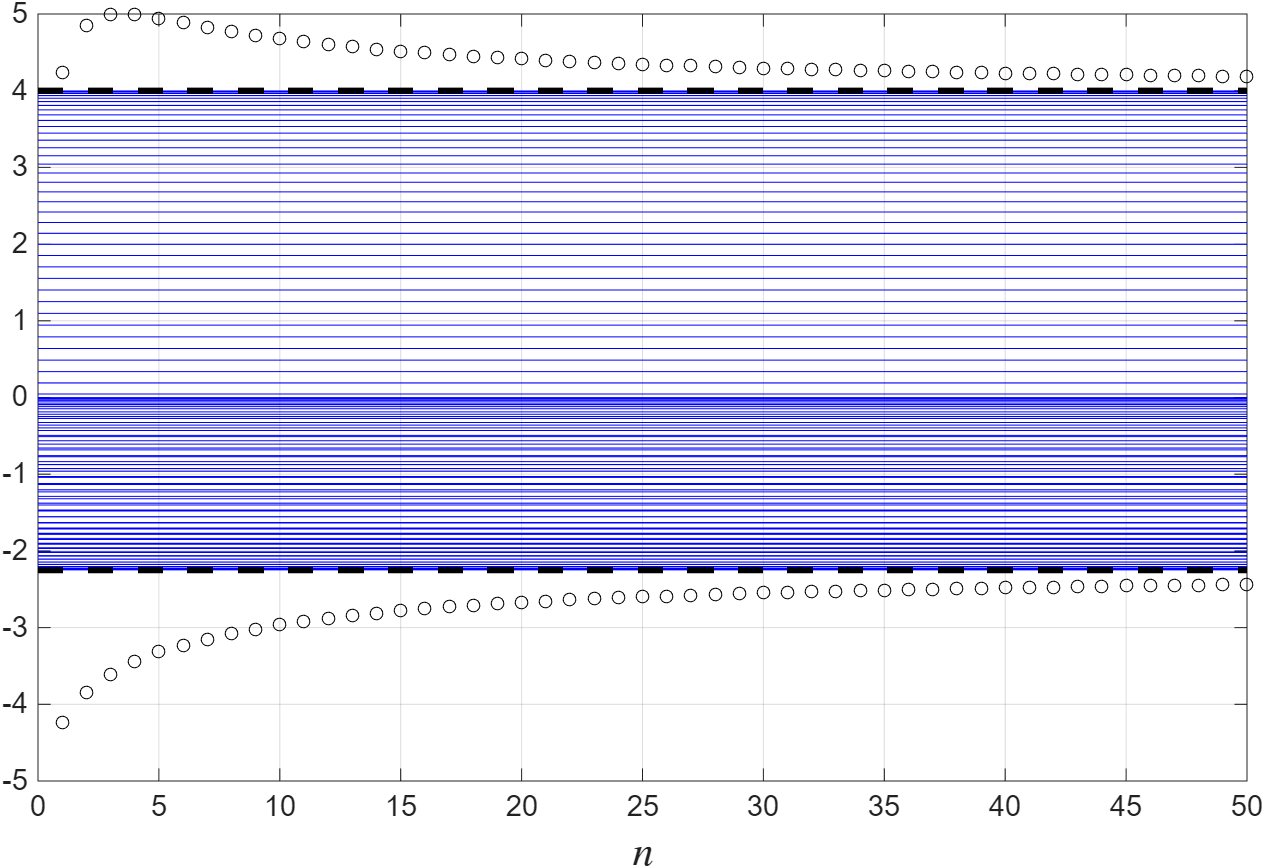} 
\end{center}
\caption{Plot of $\Spec A$, where $A=A_M$ is the order $M$ pentadiagonal Toeplitz example at the end of \cite[Remark 5.8]{CWL}. Also plotted are the $\tau$-method inclusions sets, $\R\cap \Sigma_0^n(A)$. The horizontal solid lines indicate the eigenvalues of $A$ when $M=110$, and the dashed lines the endpoints of $\Spec A^+_\infty=[-9/4,4]$. The circles are the upper and lower limits of the interval $\R\cap \Sigma_0^n(A)$, given by  \cite[Equation (5.40)]{CWL}. Where $G$ is the standard Gershgorin inclusion set, $\R\cap G=[-4,4]$; see  \cite[Remark 5.8]{CWL}.} \label{fig:pentas}
\end{figure} 

\subsubsection{figure\_2\_in\_section\_5.m} \label{code:fig2}
\begin{verbatim}
% Script file figure_2_in_section_5.m, that produces Figure 5.2 in [1].
%
% [1] S. N. Chandler-Wilde and M. Lindner, Gershgorin-type inclusion sets
%     for matrices. Submitted for publication, 2024.
%
% The file produces a plot of the intersection with the real line of the
% tau method inclusion sets of [1] for Spec A, the set of eigenvalues of A,
% in the case that A is an order M matrix A that is pentadiagonal and 
% Toeplitz with a zero main diagonal and 1's in the first two sub- and 
% super-diagonals. The plot also displays Spec A. See the caption of 
% [1,Figure 5.2] for more details.
%

% The following lines set values for the parameters M and n_len; note that
% R \cap \Sigma_0^n(A), the tau-method inclusion set intersected with the
% real line, is plotted for n = 1,...,n_len.
disp('Display the values of the parameters M and n_len:')
M = 110
n_len = 50
nn = 1:n_len;

% Compute R \cap \Sigma_0^n(A) = [tau_min(n),tau_max(n)], for n=1,..,n_len
tau_min = zeros(size(nn)); tau_max = tau_min;
Ones = [1 1];
for n = 1:n_len
    [tau_min(n),tau_max(n)] = penta_tau(n);
end

% Plot these inclusion sets in a figure
figure(1)
plot(nn,tau_min,'ko',nn,tau_max,'ko')

% Add a plot of Spec A, the eigenvalues of A
hold on
SpecA = penta(M);
xax = [0,n_len];
for i = 1:M
    plot(xax,SpecA(i)*Ones,'b-')
end

% Add a plot of the boundaries of the interval [-9/4,4], which is the
% large M limit of Spec A
SpecAinfMin = -(9/4)*Ones;
SpecAinfMax = 4*Ones;
plot(xax,SpecAinfMin,'k--',xax,SpecAinfMax,'k--','LineWidth',3);
hold off

ax = gca; ax.FontSize = 14;
xlabel('$n$',FontSize=20,Interpreter='latex')
grid
exportgraphics(gcf,'Penta.png');
\end{verbatim}

\subsubsection{penta\_tau.m and its test script file test\_penta\_tau.m} \label{code:pentau}

The following Matlab function \verb+penta_tau+ is used by the script file in \S\ref{code:fig2} to produce Figure \ref{fig:pentas}, computing the $\tau$-method inclusion set $\R\cap \Sigma_0^n(A)$, for the particular pentadiagonal Toeplitz matrix that is the focus of this section. This function in turn uses the function \verb+theta+, in \S\ref{code:theta}. Also listed below, in \S\ref{code:ptm}, is the Matlab function \verb+penta_tau_method+ that has the same inputs and outputs as \verb+penta_tau+, but computes $\R\cap \Sigma_0^n(A)$ via the general purpose $\tau$-method code \verb+tau_method+ that we present in \S\ref{sec:Herm} below. The test file \verb+test_penta_tau.m+ is listed below together with its output that shows that, over a range of parameter values, these two functions produce identical outputs.

\begin{verbatim}
function [tau_min,tau_max] = penta_tau(n)
%
% function [tau_min,tau_max] = penta_tau(n)
%
% This function evaluates the intersection of the tau-method inclusion set, 
% Sigma_0^n(A), with the real line, in the case that A is a pentadiagonal
% Toeplitz matrix with zero on the main diagonal and 1's on the fist two
% sub- and super-diagonals. Note that, for each positive integer n,
% Sigma_0^n(A) \cap R is a single closed interval [tau_min,tau_max], which,
% provided 2n < M, where M is the order of A, is an inclusion set for 
% Spec A. The function assumes that M is even so that Sigma_0^n(A) \cap R 
% is given by [1, Equation (5.40)].
%
% [1] S. N. Chandler-Wilde and M. Lindner, Gershgorin-type inclusion sets
%     for matrices. Submitted for publication, 2024.
%
% Input variable:
% n - positive integer
%
% Output variables:
% tau_min, tau_max - real, with tau_min < tau_max. [tau_min,tau_max] is the
%                    inclusion set for Spec A.

% First compute the value of theta_n (thetan) and the value of epsilon_n 
% (epsn), as defined below [1,Equation (5.16)],
% first theta_n (thetan)
thetan = theta(n);
% and now epsilon_n (epsn)
rA = 2*sqrt((3+sqrt(5))/2);
epsn = 2*rA*sin(thetan/2);

% Now compute the inclusion set, first calculating the eigenvalues of A
% in the case that A = 2n. 
Spec = penta(2*n);
tau_min = min(Spec) - epsn;
tau_max = max(Spec) + epsn;

end
\end{verbatim}

\paragraph{test\_penta\_tau.m}
\begin{verbatim}
n_len = 50;
nn = 1:n_len;
M = 110;
max_rel_err = 0;
for n = 1:n_len
    [a,b] = penta_tau(n); ptn = [a,b];
    rel_err = norm(ptn - penta_tau_method(M,n))/norm(ptn);
    max_rel_err = max([max_rel_err,rel_err]);
end
disp(' ')
disp('Test of consistency of functions penta_tau.m and penta_tau_method.m')
disp(' ')
disp('Both these functions compute the tau method inclusion set ')
disp('R\cap \Sigma_0^n(A), where A is a specific order M pentadiagonal')
disp('Toeplitz matrix. The first function uses a formula specific to this')
disp('case, the second computes the inclusion set by calling the generic')
disp('tau method function tau_method.m. Computation is carried out')
disp('for each n = 1,...,n_len. max_rel_err is the maximum, over this')
disp('range of n, of')
disp('norm(penta_tau(n) - penta_tau_method(M,n))/norm(penta_tau(n))')
disp(' ')
disp('The results of these computations are:')
disp(' ')
disp(['M = ',num2str(M),', n_len = ',num2str(n_len),', ' ...
    'max_rel_err = ',num2str(max_rel_err)])
\end{verbatim}

\paragraph{Results from test\_penta\_tau.m}

\begin{verbatim}
>> test_penta_tau
 
Test of consistency of functions penta_tau.m and penta_tau_method.m
 
Both these functions compute the tau method inclusion set 
R\cap \Sigma_0^n(A), where A is a specific order M pentadiagonal
Toeplitz matrix. The first function uses a formula specific to this
case, the second computes the inclusion set by calling the generic
tau method function tau_method.m. Computation is carried out
for each n = 1,...,n_len. max_rel_err is the maximum, over this
range of n, of
norm(penta_tau(n) - penta_tau_method(M,n))/norm(penta_tau(n))
 
The results of these computations are:
 
M = 110, n_len = 50, max_rel_err = 0
\end{verbatim}

\subsubsection{penta\_tau\_method.m} \label{code:ptm}

The following is the function \verb+penta_tau_method+ which has the same inputs and outputs as \verb+penta_tau+ above, but computes the $\tau$-method inclusion sets by a call to the general purpose $\tau$-method function \verb+tau_method+, discussed in \S\ref{sec:Herm} below.

\begin{verbatim}
function Sigma = penta_tau_method(M,n)
%
% function [tau_min,tau_max] = penta_tau_method(M,n)
%
% This function evaluates the intersection of the tau-method inclusion set, 
% Sigma_0^n(A), with the real line, in the case that A is a pentadiagonal
% Toeplitz matrix with zero on the main diagonal and 1's on the fist two
% sub- and super-diagonals. Note that, for each positive integer n,
% Sigma_0^n(A) \cap R is a single closed interval [tau_min,tau_max], which,
% provided 2n < M, where M is the order of A, is an inclusion set for 
% Spec A. 
% 
% The function assumes that M is even so that Sigma_0^n(A) \cap R 
% is given by [1, Equation (5.40)], but the computation does not use this
% formula but instead evaluates Sigma_0^n(A) \cap R directly from the
% definition of Sigma_0^n(A), which is [1,Equation (1.6)].
%
% [1] S. N. Chandler-Wilde and M. Lindner, Gershgorin-type inclusion sets
%     for matrices. Submitted for publication, 2024.
%
% Input variable:
% M - even integer > 2
% n - positive integer, with n < N := M/2
%
% Output variables:
% tau_min, tau_max - real, with tau_min < tau_max. [tau_min,tau_max] is the
%                    inclusion set for Spec A.

% Construct the matrix A, which is order M, pentadiagonal and Toeplitz, 
% as described above.
A = zeros(M);
for i = 1:M-1
    A(i, i+1) = 1;
    A(i+1, i) = 1;
end
for i = 1:M-2
    A(i, i+2) = 1;
    A(i+2, i) = 1;
end 

% Now compute the tau-method inclusion set for Spec A, thinking of A as
% having a tridiagonal block form with 2x2 blocks.
N = M/2;
m = 2*ones(1,N);
Sigma = tau_method(A,m,n);

end
\end{verbatim}

\subsubsection{penta.m} \label{code:penta}

\begin{verbatim}
function Spec = penta(M)
%
% function Spec = penta(M)
%
% This function evaluates Spec A, a row vector containing the eigenvalues
% of A, in the case that A is a an order M pentadiagonal Toeplitz matrix 
% with zeros on the main diagonal and 1's on the fist two
% sub- and super-diagonals. This Toeplitz example is discussed in Remark 
% 5.8 and Figure 5.2 in [1].
%
% [1] S. N. Chandler-Wilde and M. Lindner, Gershgorin-type inclusion sets
%     for matrices. Submitted for publication, 2024.
%
% Input variable:
% M - positive integer
%
% Output variable:
% Spec - row vector with the (real) eigenvalues of A
%
A = zeros(M);
for i = 1:M-1
    A(i, i+1) = 1;
    A(i+1, i) = 1;
end
for i = 1:M-2
    A(i, i+2) = 1;
    A(i+2, i) = 1;
end 

Spec = eig(A).';

end
\end{verbatim}

\subsubsection{theta.m} \label{code:theta}

The function \verb+theta+ solves a particular nonlinear equation, the solution of which is a component of all $\tau$-method computations in the Hermitian case. Also listed below are a script file \verb+test_theta.m+ that tests the accuracy of this function, together with the results from running that script file.  
\begin{verbatim}
function th = theta(n)
%
% function th = theta(n)
%
% Given a positive integer n, th is the unique solution in the interval
% [pi/(n+3),pi/(n+2)] of [1, Equation (5.17)], i.e. of the equation
% 2*cos((n+1)*t/2) = cos((n-1)*t/2).
%
% [1] S. N. Chandler-Wilde and M. Lindner, Gershgorin-type inclusion sets
%     for matrices. Submitted for publication, 2024.
if n == 1
    th = pi/3;
else
    th = fzero(@(t) 2*cos((n+1)*t/2)-cos((n-1)*t/2),[pi/(n+3),pi/(n+2)]);
end
end
\end{verbatim}

\paragraph{test\_theta.m}
\begin{verbatim}
% Script file test_theta to test the function theta.m
n_len = 100;
nn = 1:n_len; Resid = zeros(size(nn)); Resid2 = Resid;
for n = nn
    th = theta(n);
    Resid(n) = 8*sin(th/2)*cos((n+1/2)*th) + sin((n-1)*th);
    Resid2(n) = 2*cos((n+1)*th/2)-cos((n-1)*th/2);
end
disp('Test of function theta.m')
disp(' ')
disp('Error in satisfying 8sin(t/2)cos((n+1/2)t) + sin((n-1)t)=0,')
disp(['i.e., maximum value of LHS over n = 1,...,',num2str(n_len)])
disp(['with t = theta(n) is ',num2str(max(Resid))])
disp(' ')
disp('Error in satisfying 2cos((n+1)t/2)-cos((n-1)t/2)=0,')
disp(['i.e., maximum value of LHS over n = 1,...,',num2str(n_len)])
disp(['with t = theta(n) is ',num2str(max(Resid))])
disp(' ')
disp('Note that these two equations are eqivalent, have the same')
disp('unique zero in (pi/(n+3),pi/(n+2].')
\end{verbatim}

\paragraph{Results from test\_theta.m}
\begin{verbatim}
>> test_theta
Test of function theta.m
 
Error in satisfying 8sin(t/2)cos((n+1/2)t) + sin((n-1)t)=0,
i.e., maximum value of LHS over n = 1,...,100
with t = theta(n) is 9.3814e-15
 
Error in satisfying 2cos((n+1)t/2)-cos((n-1)t/2)=0,
i.e., maximum value of LHS over n = 1,...,100
with t = theta(n) is 9.3814e-15
 
Note that these two equations are eqivalent, have the same
unique zero in (pi/(n+3),pi/(n+2].
\end{verbatim}

\subsection{A 2-Toeplitz tridiagonal example} \label{sec:Toep2}

Figure 5.3 from \cite{CWL} is reproduced here as Figure \ref{fig:Toep2s}. Running the Matlab script file \verb+figure_3_in_section_5.m+, reproduced in \S\ref{code:fig3}, produces as output the figure file \verb+DS.png+ shown in Figure \ref{fig:Toep2s}. This figure plots the intersection with the real line of the $\tau$-method inclusions sets $\Sigma_0^n(A)$, for $n=1,\ldots,n_{\mathrm{len}}$, given by \cite[Equation (5.47)]{CWL}, in the case when $n_{\mathrm{len}}=40$ and $A$ is the order $M=60$ tridiagonal 2-Toeplitz matrix with 1's on the first sub- and superdiagonal and the 2-periodic sequence $(\Delta,-\Delta,\Delta,\ldots,(-1)^{M-1}\Delta)$ on the main diagonal, with $\Delta = 1.5$.  This script file also plots, on the same figure, $\R\cap G^N$, where $G^N$ are the block-matrix Gershgorin inclusion sets, given by \cite[Equation (5.50)]{CWL}, and the asymptotics of the boundaries of $\R\cap \Sigma_0^n(A)$, given by  \cite[Equation (5.48)]{CWL}. The figure also plots the eigenvalues of $A$. 

The parameters $\Delta\geq 0$, $M\geq 3$, and $n_{\mathrm{len}}<M$ can be altered in the script file to produce graphs for other parameter values, as made clear in the comment statements in \verb+figure_3_in_section_5+. Figure \ref{fig:Toep2E} is obtained by adjusting the parameter values so that $M=61$ and $\Delta=3$.

The script file  \verb+figure_3_in_section_5.m+ depends on subsidiary Matlab functions, \verb+climits.m+, \verb+Toep2eigs.m+, and \verb+theta.m+. The first two of these are listed below in \S\ref{code:climits} and  \S\ref{code:Toep2e}; the function \verb+theta+ is listed above in \S\ref{code:theta}. It is the first of these functions, \verb+climits+, that carries out the actual computation of $\R \cap \Sigma_0^n(A)$, computing the boundaries of the sets $\widehat S^+_{n,N}(\pm \Delta)$ in \cite[Equation (5.47)]{CWL}.  The function \verb+Toep2eigs.m+ computes $\Spec A$.

For the purpose of testing the implementation of the function \verb+climits+ and the formulae below \cite[Equation (5.47)]{CWL} that it depends on, we also include in \S\ref{code:climits} two functions that compute $\R\cap \Sigma_0^n(A)$ for the 2-Toeplitz tridiagonal matrix $A$ that is the focus of this section. The first of these, \verb+Toep2_tau+, uses the formula \cite[(5.47)]{CWL}, which expresses  $\R\cap \Sigma_0^n(A)$ as the union of two closed intervals, that may or may not overlap, and determines the boundaries of these two intervals by using the function \verb+climits+. Function  \verb+Toep2_tau_method+ has the same inputs and outputs, but determines $\R\cap \Sigma_0^n(A)$ by constructing the full matrix $A$ and then calling the general purpose $\tau$-method function \verb+tau_method+ that we discuss in \S\ref{sec:Herm} below. The script file \verb+test_Toep2_tau.m+, and the results from running this script file, both also included in \S\ref{code:climits}, checks that these two functions give, to within machine precision, the same output for $\R\cap \Sigma_0^n(A)$, i.e. the same number of intervals, and the same interval endpoints. This testing, indirectly, tests the implementation of \verb+climits+ and the formulae it is based on below  \cite[Equation (5.47)]{CWL}. Indeed, \verb+test_Toep2_tau+ makes comparison over a range of values of $n$, $\Delta$, and $M$ that is designed to test all cases in the formulae for $c_n^+(\Delta)$ and $c_{n,M}^-(\Delta)$ that  \verb+climits+ is based on.

\begin{figure}[t]
\begin{center}
\includegraphics[width=105mm]{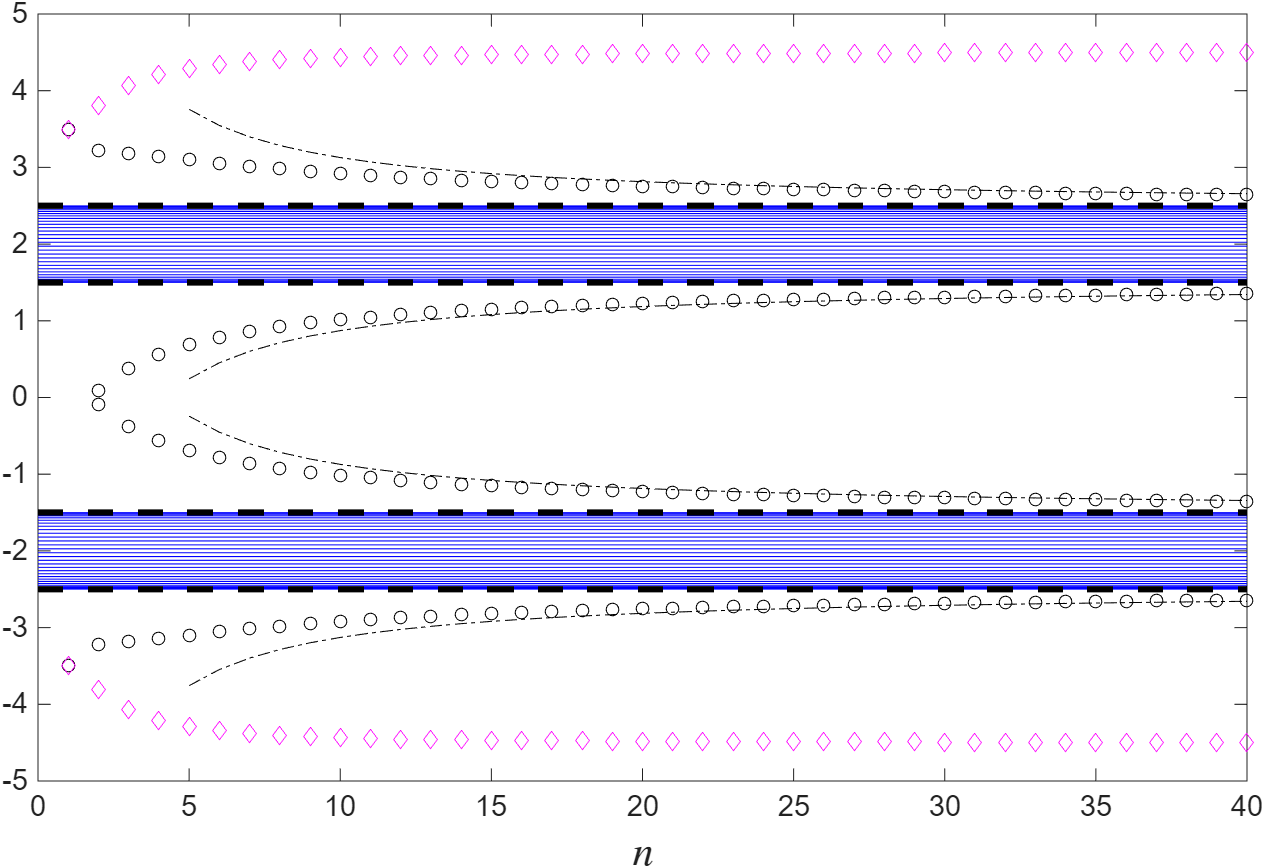} 
\end{center}
\caption{$\Spec A$, for $A=A_M(\Delta)$ with $\Delta=1.5$ and $M=60$; the Gershgorin inclusion $\R\cap G$ and its block matrix versions; and the $\tau$-method inclusions sets, $\R\cap \Sigma_0^n(A)$. The horizontal solid lines  are the eigenvalues of $A$ and the dashed lines the boundary of $\Spec A^+_\infty(\Delta)=[-3.5,-1.5]\cup [1.5,3.5]  $. The circles are $\pm c_n^+(\pm\Delta)$ (top and bottom lines) and $\pm c_{n,M}^-(\pm\Delta)$ (middle two lines); for each $M>n$, $\Spec A \subset\R \cap \Sigma_0^n(A) = \left[-c^+_{n}(-\Delta),-c^-_{n,M}(-\Delta)\right] \cup \left[c^-_{n,M}(\Delta),c^+_n(\Delta)\right]$; see \cite[\S5.3.1]{CWL}. The values of $c^-_{1,M}(\pm\Delta)$ (not plotted) are negative, so  $\R \cap \Sigma_0^1(A)=[-c^+_1(-\Delta),c^+_1(\Delta)]=[-3.5,3.5]=\R\cap G$, where $G$ is the Gershgorin inclusion set \cite[(1.10)]{CWL}.   The magenta diamonds are $\pm B_n$, where $B_n := \max S_n^+(\Delta)+2$. By \cite[(5.50)]{CWL}, $\R\cap G^{N}=[-B_n,B_n]$ for this value of $\Delta$, where $G^{N}$ is the block matrix Gershgorin inclusion set \cite[(1.8)]{CWL} with blocks of size $n$. The dash-dot lines are the first two terms of the asymptotic expansions \cite[(5.48)]{CWL} for $\pm c_n^{+}(\pm\Delta)$ and $\pm c_{n,M}^{-}(\pm\Delta)$.} \label{fig:Toep2s}
\end{figure}

\begin{figure}[t]
\begin{center}
\includegraphics[width=105mm]{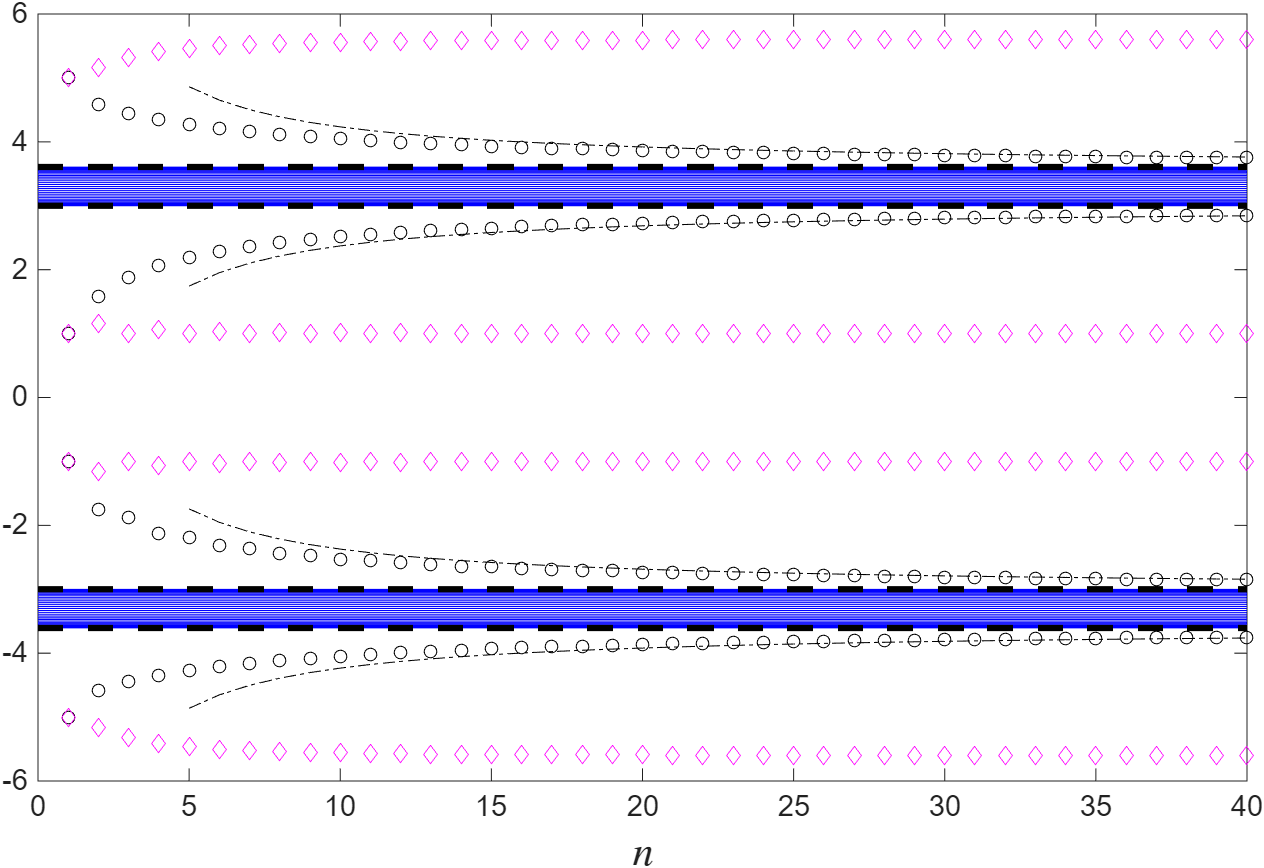} 
\end{center}
\caption{$\Spec A$, for $A=A_M(\Delta)$ with $\Delta=3$ and $M=61$; the Gershgorin inclusion $\R\cap G$ and its block matrix versions; and the $\tau$-method inclusions sets, $\R\cap \Sigma_0^n(A)$. The horizontal solid lines  are the eigenvalues of $A$ and the dashed lines the boundary of $\Spec A^+_\infty(\Delta)=[-\sqrt{13},-3]\cup [3,\sqrt{13}]$. The circles are $\pm c_n^+(\pm\Delta)$ (top and bottom lines) and $\pm c_{n,M}^-(\pm\Delta)$ (middle two lines); for each $M>n$, $\Spec A \subset\R \cap \Sigma_0^n(A) = \left[-c^+_{n}(-\Delta),-c^-_{n,M}(-\Delta)\right] \cup \left[c^-_{n,M}(\Delta),c^+_n(\Delta)\right]$; see \cite[\S5.3.1]{CWL}. Note that $\R \cap \Sigma_0^1(A)=\R\cap G$, where $G$ is the Gershgorin inclusion set \cite[(1.10)]{CWL}.   The magenta diamonds are $\pm B^\pm_n$, where $B^+_n := \max S_n^+(\Delta)+2$ and $B^-_n := \min S_n^+(\Delta)-2$. By \cite[(5.50)]{CWL}, $\R\cap G^{N}=[-B^+_n,-B^-_n]\cup [B^-_n,B^+_n]$, where $G^{N}$ is the block matrix Gershgorin inclusion set \cite[(1.8)]{CWL} with blocks of size $n$. The dash-dot lines are the first two terms of the asymptotic expansions \cite[(5.48)]{CWL} for $\pm c_n^{+}(\pm\Delta)$ and $\pm c_{n,M}^{-}(\pm\Delta)$.} \label{fig:Toep2E}
\end{figure}

\subsubsection{figure\_3\_in\_section\_5.m} \label{code:fig3}
\begin{verbatim}
% Script file figure_3_in_section_5.m, that produces Figure 5.3 in [1].
%
% [1] S. N. Chandler-Wilde and M. Lindner, Gershgorin-type inclusion sets
%     for matrices. Submitted for publication, 2024.
%
% The file produces a plot of the intersection with the real line of the
% tau method inclusion sets of [1] for Spec A, the set of eigenvalues of A,
% in the case that the order M matrix is a tridiagonal 2-Toeplitz matrix 
% with 1's on the first sub- and superdiagonal and main diagonal given by
% 
% a_{i,i} = (-1)^(i-1) Delta,   i=1,...,M.    (1)
% 
% See [1, Section 5.3] for more detail about this example. 
% The plot also displays Spec A, the large n asymptotics of the boundaries
% of these inclusion sets, and conventional block-matrix Gershgorin
% inclusion sets. See the caption of [1,Figure 5.3] for more details.
%

% The following lines set values for the parameters M, Delta, and n_len; 
% note that R \cap \Sigma_0^n(A), the tau-method inclusion set intersected
% with the real line, is plotted for n = 1,...,n_len. The user can alter
% these values to any values M >= 3, n_len < M, Delta >= 0 (the values
% currently are those of [1,Figure 5.3], i.e. M = 60, n_len = 40, Delta =
% 1.5).
disp('Display the values of the parameters M, Delta, and n_len:')
M = 60
n_len = 40
Delta = 1.5
nn = 1:n_len;

% Initialise cmv, cpv, cmvM, cpvM, Gnr_min, Gnr_max as zero vectors of the
% same size as nn
cmv = zeros(size(nn)); cpv = cmv; cmvM = cmv; cpvM = cmv;
Gnr_min = cmv; Gnr_max = cmv;

% For each n compute c_n^+(Delta) (cpv(n)), c_{n,M}^-(Delta) (cmv(n)),
% c_n^+(-Delta) (cpvM(n)), c_+n^-(-Delta) (cpvM(n)), as defined in 
% [1,Section 5.3.1]. The significance of these values is that
% R \cap \Sigma_0^n(A) = [-cpvM(n),-cmvM(n)] \cup [cmv(n),cpv(n)] if
% cmv(n) > -cmvM(n), otherwise R \cap \Sigma_0^n(A) = [-cpvM(n),cpv(n)].
%
% Also, compute the interval G_n^r = [Gnr_min(n),Gnr_max(n)], as defined
% in [1, Section 5.3.2]. In the case that n divides M, and where N = M/n
% and G^N is the block-matrix Gershgorin inclusion set using blocks of
% order n, it holds that R \cap G^N = -G_n^r \cup G_n^r, see [1,Equation
% (5.50)].
for n = 1:n_len
    [cmv(n),cpv(n)] = climits(n,M,Delta);
    [cmvM(n),cpvM(n)] = climits(n,M,-Delta);
    maxSn = sqrt(4*(cos(pi/(n+1)))^2 + Delta^2);
    Gnr_max(n) = maxSn +2;
    if rem(n,2) == 0
        minSn = sqrt(4*(sin(pi/(2*n+2)))^2 + Delta^2);
        Gnr_min(n) = minSn - 2;
    else
        Gnr_min(n) = Delta - 2;
    end
end

% Now plot the figure, first plotting the outer limits (as black circles)
% of R \cap Sigma_0^n(A) and the outer limits (as magenta diamonds) of
% \R \cap G^N.
figure(1)
plot(nn,cpv,'ko',nn,-cpvM,'ko',nn,Gnr_max,'md',nn,-Gnr_max,'md')
hold on

% Now add the inner limits (as black circles) of R \cap Sigma_0^n(A), but
% only for the values of n for which R \cap \Sigma_0^n(A) is two intervals
% rather than a single interval, i.e. where the two intervals do not
% overlap.
pos = cmv>-cmvM;
if any(pos) == true
    nn_pos = nn(pos);
    cmv_pos = cmv(pos);
    cmvM_pos = cmvM(pos);
    plot(nn_pos,cmv_pos,'ko',nn_pos,-cmvM_pos,'ko')
end

% Likewise, add the inner limits (as magenta diamonds) of R \cap G^N, but
% only for the values of n for which R \cap G^N is two intervals rather
% than a single interval.
pos = Gnr_min>0;
if any(pos) == true
    nn_pos = nn(pos);
    Gnr_pos = Gnr_min(pos);
    plot(nn_pos,Gnr_pos,'md',nn_pos,-Gnr_pos,'md')
end

% Now compute Spec A (SpecA) and plot each of these eigenvalues as a
% horizontal line.
SpecA = Toep2eigs(M,Delta);
Ones = [1 1];
xax = [0,n_len];
for i = 1:M
    plot(xax,SpecA(i)*Ones,'b-')
end

% Now add to the plot, as dashed lines, the boundaries of 
% \lim_{M\to\infty} Spec A, given by [1,Equation (5.44)].
smv = Delta*Ones;
spv = sqrt(4+Delta^2)*Ones;
plot(xax,spv,'k--',xax,-spv,'k--','LineWidth',3);
if Delta > 0
    plot(xax,smv,'k--',xax,-smv,'k--','LineWidth',3);
end

% Finally, add to the plot the first two terms of the asymptotic
% expansions, given by [1, Equation (5.48)], of \pm c_n^+(\pm Delta) and 
% \pm c_{n,M}^-(\pm Delta), the boundaries of R \cap Sigma_0^n(A). But
% only plot the asymptotic expansions of c_{n,M}^-(Delta) and
% -c_{n,M}^-(-Delta), respectively, for the values of n for which those
% expansions are positive/negative, respectively.
nn10 = 5:n_len;
cpa = sqrt(4+Delta^2)+2*pi./nn10;
cma = Delta-2*pi./nn10;
plot(nn10,cpa,'k-.',nn10,-cpa,'k-.')
pos = cma>0;
if any(pos) == true
    nn10_pos = nn10(pos);
    cma_pos = cma(pos);
    plot(nn10_pos,cma_pos,'k-.',nn10_pos,-cma_pos,'k-.')
end
hold off

% Add an x-axis label
ax = gca; ax.FontSize = 14;
xlabel('$n$',FontSize=20,Interpreter='latex')

% Export the plot to file DS.png in the current directory.
exportgraphics(gcf,'DS.png');
\end{verbatim}

\subsubsection{climits.m} \label{code:climits}
\begin{verbatim}
function [cm,cp] = climits(n,N,Delta)
%
% function [cm,cp] = climits(n,N,Delta)
%
% This function evaluates the functions c^+_n(Delta) and c^-_{n,N}(Delta)
% defined in section 5.3.1 of [1].
% 
% [1] S. N. Chandler-Wilde and M. Lindner, Gershgorin-type inclusion sets
%     for matrices. Submitted for publication, 2024.
%
% Inputs:
% N > n > 1 are integers
% Delta is real
%
% Outputs:
% cm, cp are the values of c^-_{n,N}(Delta) and c^+_n(Delta), respectively,
% with cm < cp and cp > 0.

% Compute the values of theta_n (thetan) and the value of epsilon_n 
% (epsn), as defined below [1,Equation (5.16)]. First theta_n (thetan)
thetan = theta(n);
% Now compute epsilon_n (epsn)
epsn = 4*sin(thetan/2);
% Now compute theta_{n-2} (thetan2) and epsilon_{n-2} (epsn2), if n > 2.
if n > 2
    thetan2 = theta(n-2);
    epsn2 = 4*sin(thetan2/2);
end

% Now compute the maximum (maxSn) and minimum (minSn) values of 
% S_n^+(Delta), as defined in section 5.3.1 of [1]. Note that the minimum 
% value is correct only for n even (which is all we need).
maxSn = sqrt(4*(cos(pi/(n+1)))^2 + Delta^2);
minSn = sqrt(4*(sin(pi/(2*n+2)))^2 + Delta^2);

% Finally, compute cm and cp
cp = maxSn + epsn;
if rem(n,2) == 1 | n == 2 & (rem(N,2) == 0 | Delta>0)
    cm = abs(Delta) - epsn;
else
    if rem(N,2) == 1 & Delta <= 0
        cm = minSn - epsn;
    else
        cm = max(minSn - epsn2, abs(Delta) -epsn);
    end
end

end
\end{verbatim}

\paragraph{Toep2\_tau.m}

\begin{verbatim}
function Sigma = Toep2_tau(M,Delta,n)
%
% function Sigma = Toep2_tau(M,Delta,n)
%
% This function evaluates the intersection of the tau-method inclusion set, 
% Sigma_0^n(A), with the real line, in the case that the order M natrix
% A is a tridiagonal 2-Toeplitz matrix with 1's on the first sub- and 
% superdiagonal and main diagonal given by
% 
% a_{i,i} = (-1)^(i-1) Delta,   i=1,...,M.    (1)
% 
% See [1, Section 5.3] for more detail about this example. Note that, 
% for each positive integer n < M, Sigma_0^n(A) \cap R is an inclusion set 
% for Spec A, the set of eigenvalues of A. 
% 
% This function evaluates Sigma_0^n(A) \cap R using [1,Equation (5.47)].
%
% [1] S. N. Chandler-Wilde and M. Lindner, Gershgorin-type inclusion sets
%     for matrices. Submitted for publication, 2024.
%
% Input variables:
% M - positive integer, the order of A
% Delta - real number, see equation (1) above for its meaning
% n - positive integer, with n < M
%
% Output variables:
% Sigma - real Ix2 array, for some positive integer I. The inclusion set
%         R \cap Sigma_0^n(A) consists of I disjoint closed intervals,
%         [a_i,b_i], i = 1,...,I, with I = 1 or 2 and a_i <= b_i, i = 1,
%         ...,I. The ith row of Sigma is [a_i,b_i].
%
    
    [cm_plus,cp_plus] = climits(n,M,Delta);
    [cm_minus,cp_minus] = climits(n,M,-Delta);
    if cm_plus > -cm_minus
        Sigma = [-cp_minus, -cm_minus; cm_plus, cp_plus];
    else
        Sigma = [-cp_minus, cp_plus];
    end

end
\end{verbatim}

\paragraph{Toep2\_tau\_method.m}

\begin{verbatim}
function Sigma = Toep2_tau_method(M,Delta,n)
%
% function Sigma = Toep2_tau_method(M,Delta,n)
%
% This function evaluates the intersection of the tau-method inclusion set, 
% Sigma_0^n(A), with the real line, in the case that the order M natrix
% A is a tridiagonal 2-Toeplitz matrix with 1's on the first sub- and 
% superdiagonal and main diagonal given by
% 
% a_{i,i} = (-1)^(i-1) Delta,   i=1,...,M.    (1)
% 
% See [1, Section 5.3] for more detail about this example. Note that, 
% for each positive integer n < M, Sigma_0^n(A) \cap R is an inclusion set 
% for Spec A, the set of eigenvalues of A. 
% 
% This function evaluates Sigma_0^n(A) \cap R directly from the
% definition of Sigma_0^n(A), which is [1,Euqtion (1.6)].
%
% [1] S. N. Chandler-Wilde and M. Lindner, Gershgorin-type inclusion sets
%     for matrices. Submitted for publication, 2024.
%
% Input variables:
% M - positive integer, the order of A
% Delta - real number, see equation (1) above for its meaning
% n - positive integer, with n < M
%
% Output variables:
% Sigma - real Ix2 array, for some positive integer I. The inclusion set
%         R \cap Sigma_0^n(A) consists of I disjoint closed intervals,
%         [a_i,b_i], i = 1,...,I, with I = 1 or 2 and a_i <= b_i, i = 1,
%         ...,I. The ith row of Sigma is [a_i,b_i].
%

    % Construct the matrix A, which is order M, tridiagonal and 2-Toeplitz, 
    % as described above.
    A = zeros(M);
    for i = 1:M-1
        A(i,i+1)= 1;
        A(i+1,i) = 1;
    end
    for i = 1:M
        A(i,i) = (-1)^(i-1)*Delta;
    end

    % Now compute R \cap Sigma_0^n(A)
    m = ones(1,M);
    Sigma = tau_method(A,m,n);
end
\end{verbatim}

\paragraph{test\_Toep2\_tau.m}

\begin{verbatim}
n_len = 40;
nn = 1:n_len;
Mrange = 59:62;
Delta_range = [-1.5,-1,-0.1,0,0.1,0.5,1,1.5,2,3];
max_rel_err = 0;
for n = 1:n_len
    for M = Mrange
        for Delta = Delta_range
            T2f = Toep2_tau(M,Delta,n);
            rel_err = norm(T2f -Toep2_tau_method(M,Delta,n))/norm(T2f);
            max_rel_err = max([max_rel_err,rel_err]);
        end
    end
end
disp(' ')
disp('Test of consistency of functions Toep2_tau.m and Toep2_tau_method.m')
disp(' ')
disp('Both these functions compute the tau method inclusion set ')
disp('R\cap \Sigma_0^n(A), where A is a specific order M tridiagonal')
disp('2-Toeplitz matrix with 1''s on the first sub- and superdiagonal')
disp('and \pm Delta on the main diagonal. The first function uses a')
disp('formula specific to this case, the second computes the inclusion')
disp('set by calling the generic tau method function tau_method.m.') 
disp('Computation is carried out for n = 1,...,n_len, for M = M_1,...')
disp(['M_2, and for ',num2str(length(Delta_range)),' values of Delta,'...
    ' ranging from ',num2str(Delta_range(1)),' to ',...
    num2str(Delta_range(end)),'.'])
disp(' ')
disp('max_rel_err is the maximum, over these n, M, and Delta values, of')
disp(['norm(Toep2_tau(M,Delta,n) - Toep2_tau_method(M,Delta,n))/', ...
    'norm(Toep2_tau(M,Delta,n))'])
disp(' ')
disp('The results of these computations are:')
disp(' ')
disp(['M_1 = ',num2str(Mrange(1)),', M_2 = ',num2str(Mrange(end)), ...
    ', n_len = ',num2str(n_len),', max_rel_err = ',num2str(max_rel_err)])
\end{verbatim}

\paragraph{Results from test\_Toep2\_tau.m}

\begin{verbatim}
>> test_Toep2_tau
 
Test of consistency of functions Toep2_tau.m and Toep2_tau_method.m
 
Both these functions compute the tau method inclusion set 
R\cap \Sigma_0^n(A), where A is a specific order M tridiagonal
2-Toeplitz matrix with 1's on the first sub- and superdiagonal
and \pm Delta on the main diagonal. The first function uses a
formula specific to this case, the second computes the inclusion
set by calling the generic tau method function tau_method.m.
Computation is carried out for n = 1,...,n_len, for M = M_1,...
M_2, and for 10 values of Delta, ranging from -1.5 to 3.
 
max_rel_err is the maximum, over these n, M, and Delta values, of
norm(Toep2_tau(M,Delta,n) - Toep2_tau_method(M,Delta,n))/norm(Toep2_tau(M,Delta,n))
 
The results of these computations are:
 
M_1 = 59, M_2 = 62, n_len = 40, max_rel_err = 6.5308e-16
\end{verbatim}

\subsubsection{Toep2eigs.m}  \label{code:Toep2e}

This subsection contains the function \verb+Toep2eigs+ that computes the eigenvalues of the particular 2-Toeplitz tridiagonal matrix $A_M(\Delta)$ that is discussed in \cite[\S5.3]{CWL}, using the explicit formula \cite[(5.43)]{CWL} for these eigenvalues. Also listed below are the function  \verb+Toep2eigsBF+ that does the same calculation, with the same inputs and output, using a brute force calculation, constructing the full matrix and computing its eigenvalues using \verb+eig+, and the test script file \verb+test_Toep2eigs.m+ and its output. These validate  \verb+Toep2eigs+ and the formula \cite[(5.43)]{CWL} that it is based on by demonstrating that \verb+Toep2eigs+ and \verb+Toep2eigsBF+ produce the same output, to within rounding errors, over a range of values of the parameters $\Delta$ and $M$.

\begin{verbatim}
function Spec = Toep2eigs(M,Delta)
%
% function Spec = Toep2eigs(M,Delta)
%
% Computes the eigenvalues of A, the order M tridiagonal 2-Toeplitz matrix
% with 1's on the first sub- and superdiagonal and main diagonal given by
% 
% a_{i,i} = (-1)^(i-1) Delta,   i=1,...,M,    (1)
% 
% using the formula [1, Equation (5.43)]. See [1, Section 5.3] for more 
% detail about this example.
%
% [1] S. N. Chandler-Wilde and M. Lindner, Gershgorin-type inclusion sets
%     for matrices. Submitted for publication, 2024.
%
% Input variables:
% M - positive integer, order of the matrix A
% Delta - real number, determining the leading diagonal of A through (1)
%
% Output variable:
% SpecSN - real row vector of length M, the eigenvalues 
%          of A in ascending order 
%

M2 = floor(M/2);
% Compute the M2 positive eigenvalues of A in the case Delta = 0, when
% A = LM, the discrete Laplacian of order M.
SpecLM = 2*cos((1:M2)*pi/(M+1));

% Now compute the eigenvalues of A, using [1, Equation (5.43)].
Spec = sqrt(SpecLM.^2+Delta^2);
if rem(M,2)==0
    Spec = [-Spec,flip(Spec)];
else
    Spec = [-Spec,Delta,flip(Spec)];
end
\end{verbatim}

\paragraph{Toep2eigsBF.m}

\begin{verbatim}
function Spec = Toep2eigsBF(M,Delta)
%
% function Spec = Toep2eigsBF(M,Delta)
%
% Computes the eigenvalues of A, the order M tridiagonal 2-Toeplitz matrix
% with 1's on the first sub- and superdiagonal and main diagonal given by
% 
% a_{i,i} = (-1)^(i-1) Delta,   i=1,...,M,    (1)
% 
% by a brute force calculation. See [1, Section 5.3] for more detail
% about this example.
%
% [1] S. N. Chandler-Wilde and M. Lindner, Gershgorin-type inclusion sets
%     for matrices. Submitted for publication, 2024.
%
% Input variables:
% M - positive integer, order of the matrix A
% Delta - real number, determining the leading diagonal of A through (1)
%
% Output variable:
% SpecSN - real row vector of length M, the eigenvalues 
%          of A in ascending order 
%
S_N = zeros(M);
for i = 1:M
    S_N(i,i) = (-1)^(i-1)*Delta;
end
for i=1:M-1
    S_N(i,i+1) = 1;
    S_N(i+1,i) = 1;
end
Spec = eig(S_N).';
end
\end{verbatim}

\paragraph{test\_Toep2eigs.m}
\begin{verbatim}
M_len = 40;
Delta_range = [-1.5,-1,-0.1,0,0.1,0.5,1,1.5,2,3];
max_rel_err = 0;
for M = 1:M_len
    for Delta = Delta_range
        T2e = Toep2eigs(M,Delta);
        rel_err = norm(T2e -Toep2eigsBF(M,Delta))/norm(T2e);
        max_rel_err = max([max_rel_err,rel_err]);
    end
end
disp(' ')
disp('Test of consistency of functions Toep2eigs.m and Toep2eigsBF.m')
disp(' ')
disp('Both these functions compute the eigenvalues of A, where A is a')
disp('specific order M tridiagonal 2-Toeplitz matrix with 1''s on the') 
disp('first sub- and superdiagonal and \pm Delta on the main diagonal.')
disp('The first function uses a formula specific to this case, the second')
disp('computes the eigenvalues by Matlab''s generic eig function.')
disp('Computation is carried out for M = 1,...,M_len, and for')
disp([num2str(length(Delta_range)),' values of Delta,'...
    'ranging from ',num2str(Delta_range(1)),' to ',...
    num2str(Delta_range(end)),'.'])
disp(' ')
disp('max_rel_err is the maximum, over these M and Delta values, of')
disp(['norm(Toep2eigs(M,Delta) - Toep2eigsBF(M,Delta))/', ...
    'norm(Toep2eigs(M,Delta))'])
disp(' ')
disp('The results of these computations are:')
disp(' ')
disp(['M_len = ',num2str(M_len),', max_rel_err = ',num2str(max_rel_err)])
\end{verbatim}

\paragraph{Output from test\_Toep2eigs.m}
\begin{verbatim}
>> test_Toep2eigs
 
Test of consistency of functions Toep2eigs.m and Toep2eigsBF.m
 
Both these functions compute the eigenvalues of A, where A is a
specific order M tridiagonal 2-Toeplitz matrix with 1's on the
first sub- and superdiagonal and \pm Delta on the main diagonal.
The first function uses a formula specific to this case, the second
computes the eigenvalues by Matlab's generic eig function.
Computation is carried out for M = 1,...,M_len, and for
10 values of Delta,ranging from -1.5 to 3.
 
max_rel_err is the maximum, over these M and Delta values, of
norm(Toep2eigs(M,Delta) - Toep2eigsBF(M,Delta))/norm(Toep2eigs(M,Delta))
 
The results of these computations are:
 
M_len = 40, max_rel_err = 5.2478e-16
\end{verbatim}

\section{A Matlab function to compute the $\tau$-method inclusion sets in the Hermitian case} \label{sec:Herm}

In this section we provide a Matlab function \verb+tau_method+ that computes $\R\cap \Sigma_0^n(A)$, i.e.~the intersection with the real line of the $\tau$-method inclusion set $\Sigma_0^n(A)$, as defined in \cite[\S1.1]{CWL}, in the case that $A$ is Hermitian.  Given the $M\times M$ real or complex Hermitian matrix $A$, a finite sequence $(m_1,\ldots, m_N)$ of length $N$ satisfying $1<N\leq M$, and such that
\begin{equation} \label{eq:msum}
\sum_{i=1}^N m_i = M,
\end{equation}
and  an integer $n$ with $1\leq n< N$,
the function \verb+tau_method+ computes $\R\cap \Sigma_0^n(A)$ based on a partition of $A$ in the block form
\begin{equation} \label{eq:blocks}
A = [a_{ij}]_{i,j=1}^N
\end{equation}
where each $a_{ij}\in \C^{m_i \times m_j}$ is a particular sub-block of $A$. Assuming that $A$ is Hermitian so that $\Spec A$ is real,  
\begin{equation} \label{eq:incl}
\Spec A \subset \R\cap \Sigma_0^n(A), \qquad 1\leq n < N.
\end{equation} 

The computational method is as follows. As in \cite[\S1]{CWL},
let $B:=[b_{ij}]_{i,j=1}^N$, where
$$
b_{ij} := \left\{\begin{array}{ll} a_{ij}, & |i-j|\leq 1,\\ 0, & \mbox{otherwise}.\end{array}\right.
$$
As in \cite[\S1]{CWL} we term $B$ the {\it tridiagonal} part of $A$.
Let  $C:=A-B$, which we term the {\em remaining part} of $A$, and set
\begin{eqnarray} \nonumber
r_L(A) &:=& \max\{\|a_{2,1}\|,\|a_{3,2}\|,\ldots,\|a_{N,N-1}\|\},\\ \label{eq:rdefs}
 r_U(A)& :=& \max\{\|a_{1,2}\|,\|a_{2,3}\|,\ldots,\|a_{N-1,N}\|\},\\ \nonumber
r(A) &:=& r_L(A)+r_U(A),
\end{eqnarray}
noting that $r(A)=2r_L(A)=2r_U(A)$ if $A$ is Hermitian and that all matrix norms are $2$-norms.
For $n=1,\ldots,N$ and $k=0,\ldots,N-n$, let 
$$
B_{n,k}:=[b_{ij}]_{i,j=k+1}^{k+n},
$$ 
so that $B_{n,k}$ is the $n\times n$ block-tridiagonal matrix consisting of the blocks of $B$ whose rows and columns are in the range $k+1$ through $k+n$. For $n=1,\ldots,N-1$, let
\begin{eqnarray} \label{eq:sigmahat}
\widehat \sigma^n(A) &:=&\bigcup_{k=0}^{N-n}  \Spec B_{n,k}, \\ \label{eq:sigma}
\sigma^n(A) &:=& \widehat \sigma^n_{\eps}(A) \cup \bigcup_{m=1}^{n-1} \left(\Spec  B_{m,0}\cup \Specn  B_{m,N-m}\right).
\end{eqnarray}
Then, in the case that $A$ is Hermitian, so that each $B_{n,k}$ is Hermitian, so that $\Speps B_{n,k} = B_{n,k} + \eps \overline{\D}$, for every $\eps\geq 0$, it follows, using the definitions in and above \cite[(1.6)]{CWL}, that
\begin{equation} \label{eq:hatsig}
\left.\begin{array}{ll}
\R\cap \sigma_0^n(A) = \sigma^n(A) + [-\eps_n,\eps_n], & n\in \N,\\
\R\cap \widehat \sigma_0^n(A) = \widehat \sigma^n(A) + [-\eps_{n-2},\eps_{n-2}], & n>2,\end{array} \right\}
\end{equation}
where $\eps_n := 4r_L(A) \sin(\theta_n^*/2) + \|C\|$ and $\theta_n^*$ is the unique solution of \cite[(1.5)]{CWL} in the range $\pi/(2n+1),\pi/(n+2)]$ (note that $r_L(A)r_U(A)/(r(A))^2=1/4$ in this Hermitian case). Further,
\begin{equation} \label{eq:Sigma}
\R\cap \Sigma_0^n(A) =\left\{\begin{array}{ll}\R\cap \sigma^n_{\eps}(A), & \mbox{if } n = 1,2,\\
\R\cap \sigma^n_{\eps}(A) \cap \widehat \sigma^n_{\eps}(A),& \mbox{if } n> 2.\end{array} \right.
\end{equation}  

Regarding implementation, the function \verb+tau_method+, presented in \S\ref{code:tm}, computes $\R\cap \Sigma_0^n(A)$ from $A$, $m=(m_1,m_2,\ldots,m_N)$, and $n$, using \eqref{eq:Sigma} and \eqref{eq:hatsig}, and using the function \verb+eigs_sigma+, presented in \S\ref{code:es}, to compute $\widehat \sigma^n(A)$, $\sigma^n(A)$, $r_L(A)$, and $\|C\|$. The function \verb+eigs_sigma+, largely written by ChatGPT based on a specification similar to the above, computes  $\widehat \sigma^n(A)$, $\sigma^n(A)$, $r_L(A)$, and $\|C\|$, using \eqref{eq:rdefs}, \eqref{eq:sigmahat}, and \eqref{eq:sigma}. The function \verb+theta+, presented above in \S\ref{code:theta}, computes $\theta^*_n$. 

Note that $\R\cap \sigma_0^n(A)$, $\R\cap \widehat \sigma_0^n(A)$, and $\R\cap \widehat \Sigma_0^n(A)$ are each finite unions of closed intervals of the form $[a_i,b_i]$ with $a_i\leq b_i$. These are represented in \verb+tau_method+ and \verb+eigs_sigma+ as $I\times 2$ real arrays, for some $I\in\N$, each interval $[a_i,b_i]$ stored as one of the rows. The function \verb+tau_method+ returns as output the set $\R\cap \widehat \Sigma_0^n(A)$ represented as such an array with minimal number of rows, i.e., if  $\R\cap \widehat \Sigma_0^n(A)$ is a union of $I$ disjoint closed intervals, then the array returned will have $I$ rows. The necessary merging of intervals to achieve this minimal representation, and the operations needed to perform the intersection of   $\R\cap \sigma_0^n(A)$ and $\R\cap \widehat \sigma_0^n(A)$ in \eqref{eq:Sigma}, are carried out by auxiliary functions \verb+mergeIntervals+ and \verb+intersectIntervals+, written by ChatGPT and presented in \S\ref{code:int}.

We note that we have already carried out tests that serve as a check of the implementation of the function  \verb+tau_method+ and of the functions  \verb+eigs_sigma+, \verb+mergeIntervals+, and \verb+intersectIntervals+ that it uses to carry out its computations. Specifically, we have compared in \S\ref{code:pentau} outputs computed with \verb+tau_method+ with an explicit calculation of $\R\cap \Sigma_0^n(A)$ for a particular pentadiagonal Hermitian matrix $A$, using the function \verb+penta_tau+ based on the formulae \cite[(5.40)]{CWL}. This is an example where \verb+tau_method+ predicts correctly that $\R\cap \Sigma_0^n(A)$ is a single closed interval for each $n$. Similarly, in \S\ref{sec:Toep2} and \S\ref{code:climits} we have used \verb+tau_method+, called from the function \verb+Toep2_tau_method+, to compute  $\R\cap \Sigma_0^n(A)$ in the case that $A$ is the tridiagonal 2-Toeplitz matrix of order $M$ with $\pm \Delta$ on the main diagonal, discussed in \S\ref{sec:Toep2}. In this case we have compared the results computed by \verb+tau_method+ with those computed by \verb+Toep2_tau+ in \S\ref{code:climits} by use of the explicit formula \cite[(5.47)]{CWL}. As shown in the results in \S\ref{code:climits} from \verb+test_Toep2_tau+, the function \verb+tau_method+, for the wide range of values of $n$, $M$, and $\Delta$ tested, produces the same results as \cite[(5.47)]{CWL} to within machine precision, in particular predicting the correct number of subintervals (1 or 2), depending on the values of $n$, $A$, and $M$, of which  $\R\cap \Sigma_0^n(A)$ is comprised.

As a further illustration of the use of \verb+tau_method+, we provide in \S\ref{code:Toep3} a function to compute $\R\cap \Sigma_0^n((A)$ for a similar case to that of \S\ref{sec:Toep2}, the case in which $A$ is Hermitian and tridiagonal of order $M$, with 1's on the first sub- and superdiagonal, and a length $M$ 3-periodic sequence $(0,1,0,0,1,0,\ldots)$ on the main diagonal. This function \verb+Toep3_tau_method+ takes as inputs the values $n\in \N$ and $M>n$ and computes as outputs the sets $\R\cap \Sigma_0^n(A)$ and $\Spec A$, the former encoded as a finite list of disjoint closed intervals, and computed by a call to \verb+tau_method+, choosing $m_1=\dots=m_N=1$, i.e.~$1\times 1$ blocks in the representation \eqref{eq:blocks}. The script file \verb+ThreeToeplitzExample.m+ makes calls to \verb+Toep3_tau_method+ to produce the graph in Figure \ref{fig:Toep3}. Note that the magenta solid lines at the bottom of this Figure \ref{fig:Toep3} are the spectrum of the corresponding semi-infinite matrix $A^+_\infty$, given in \cite[Examples 3.5(d) and 3.10(b)]{Gab} as
\begin{equation} \label{eq:specinf}
\Spec A^+_\infty = [-\sqrt{3},-1]\cup [1-\sqrt{2},1]\cup [\sqrt{3},1+\sqrt{2}]\cup \left\{-\frac{\sqrt{5}-1}{2}\right\}.
\end{equation}

\begin{figure}[t]
\begin{center}
\includegraphics[width=105mm]{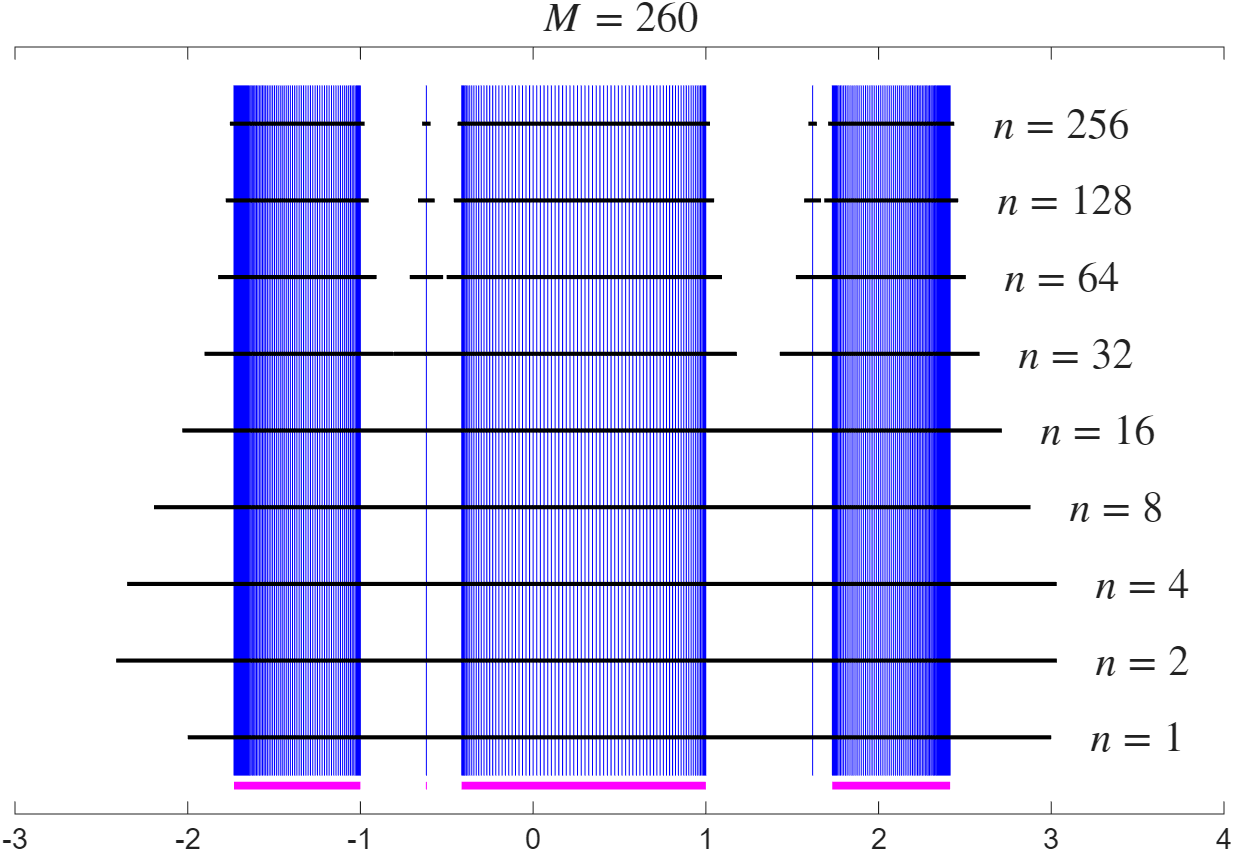} 
\end{center}
\caption{Plot of $\Spec A$ and the $\tau$-method inclusions sets, $\R\cap \Sigma_0^n(A)$, in the case that the order $M$ tridiagonal Hermitian matrix $M$ has 1's on the first sub- and superdiagonal, and the length $M$ periodic sequence $(0,1,0,0,1,0,\ldots)$ on the main diagonal. The vertical solid lines  are the eigenvalues of $A$, in the case $M=260$.  The thick solid horizontal lines are $\R \cap \Sigma_0^n(A)$ for the indicated values of $n$, each a union of finitely many disjoint closed intervals, and each an inclusion set for $\Spec A$. Note that $\R \cap \Sigma_0^1(A)=\R\cap G$, where $G$ is the Gershgorin inclusion set \cite[(1.10)]{CWL}, so that the lowest line for $n=1$ is also a plot of $\R\cap G = [-2,3]$. The magenta lines at the bottom indicate the spectrum of the corresponding semi-infinite matrix $A^+_\infty$, given by \eqref{eq:specinf}. } \label{fig:Toep3}
\end{figure}

\subsection{Matlab functions and test scripts} \label{sec:Mfts}

\subsubsection{tau\_method.m} \label{code:tm}

\begin{verbatim}
function Sigma = tau_method(A, m, n)
%
% function Sigma = tau_method(A, m, n)
%
% This function evaluates the intersection of the tau-method inclusion set, 
% Sigma_0^n(A), as defined in [1, Equation (1.6)], with the real line in
% the case that A is Hermitian. 
%
% [1] S. N. Chandler-Wilde and M. Lindner, Gershgorin-type inclusion sets
%     for matrices. Submitted for publication, 2024.
%
% Input variables:
% A         - MxM Hermitian complex matrix
% m         - 1xN vector of positive integers, such that sum(m) = M
%             In the tau method A is written in block form as A = 
%             [a_{i,j}]_{i,j=1,...,N}, where N = length(m), with a_{i,j}
%             an m_i x m_j block.
% n         - integer with 1 <= n < N = length(m)
%
% Output variable:
%   Sigma   - Ix2 real matrix encoding the tau method inclusion set
%             \Sigma_0^n(A) intersected with the real line. Precisely, 
%             \Sigma_0^n(A) \cap R is a union of I disjoint intervals, 
%             [a_i,b_i], with a_i <= b_i. Each row of Sigma is one of 
%             these intervals.

    N = length(m);
    M = sum(m);

    % Now call eigs_sigma to compute r_L(A), defined by [1, (1.3)], the 
    % maximum of the 2-norms of the blocks on the first subdiagonal, and
    % normC, the 2-norm of C, the remaining part of A, which means A with
    % the blocks on the main diagonal and on the first sub- and
    % superdiagonal set to zero. Finally compute sigmahat and sigma, both
    % unions of eigenvalues of principal square submatrices of the block
    % tridiagonal matrix B = A - C.
    [rL, normC, sigmaHat, sigma] = eigs_sigma(A, m, n);

    % Now compute the value of theta_n (thetan) and the value of epsilon_n 
    % (epsn), as defined below [1,Equation (5.16)]. First compute thetan
    thetan = theta(n);
    % Now compute epsilon_n (epsn)
    epsn = 4*rL*sin(thetan/2) + normC;
    % Now compute theta_{n-2} (thetan2) and epsilon_{n-2} (epsn2), if n > 2
    if n > 2
        thetan2 = theta(n-2);
        epsn2 = 4*rL*sin(thetan2/2) + normC;
    end
    
    % Compute sigma_0^n(A) (sigma_0n) = sigma + [-eps_n,eps_n], which is a
    % finite union of closed intervals, stored in a Jx2 real matrix, each
    % interval in one row.
    sigma_len = length(sigma);
    sigma_0n = zeros(sigma_len,2);
    for j = 1:sigma_len
        sigma_0n(j,:) = [sigma(j)-epsn,sigma(j)+epsn];
    end
    sigma_0n = mergeIntervals(sigma_0n);

    if n > 2
        % If n > 2 compute \hat sigma_0^n(A) (sigmaHat_0n) = sigmaHat + 
        % [-eps_{n-2},eps_{n-2}], which is a finite union of closed 
        % intervals, stored in a Kx2 real matrix, each interval in one row.
        sigmaHat_len = length(sigmaHat);
        sigmaHat_0n = zeros(sigmaHat_len,2);
        for j = 1:sigmaHat_len
            sigmaHat_0n(j,:) = [sigmaHat(j)-epsn2,sigmaHat(j)+epsn2];
        end
        sigmaHat_0n = mergeIntervals(sigmaHat_0n);
        % Now define Sigma as the intersection of sigma_0n nd sigmaHat_0n
        Sigma = intersectIntervals(sigma_0n, sigmaHat_0n);
    else
        Sigma = sigma_0n;
    end
end
\end{verbatim}

\subsubsection{eigs\_sigma.m} \label{code:es}

\begin{verbatim}
function [rL_val, normC, sigmaHat, sigma] = eigs_sigma(A, m, n)
%
% function [rL_val, normC, sigmaHat, sigma] = eigs_sigma(A, m, n)
%
% This function evaluates quantities needed to compute the intersection of
% the tau-method inclusion set, Sigma_0^n(A), as defined in [1, (1.6)], 
% with the real line in the case that A is Hermitian. 
%
% [1] S. N. Chandler-Wilde and M. Lindner, Gershgorin-type inclusion sets
%     for matrices. Submitted for publication, 2024.
%
% Input variables:
% A         - MxM Hermitian real or complex matrix
% m         - 1xN vector of positive integers, such that sum(m) = M
%             In the tau method A is written in block form as A = 
%             [a_{i,j}]_{i,j=1,...,N}, where N = length(m), with a_{i,j}
%             an m_i x m_j block.
% n         - integer with 1 <= n < N = length(m)
%
% Output variable:
% rL_val    - non-negative scalar, the value of r_L(A), defined by [1, (1.3)],
%             the maximum of the 2-norms of the blocks on the first 
%             subdiagonal of A.,
% normC     - non-negative scalar, the 2-norm of C, the remaining part of
%             A, which means A with the blocks on the main diagonal and on
%             the first sub- and superdiagonal set to zero.
% sigmaHat  - real row vector, containing the union of the eigenvalues
%             of all the n x n block principal square submatrices of the 
%             block tridiagonal matrix B = A - C. (This is the union over
%             k = 0,...,N-n of Spec B_{n,k}, in the notation of 
%             [1, Section 1.1].
% sigma     - real row vector, the union of sigmaHat with the union of the
%             eigenvalues of B_{m,0} and B_{m,N-n}, in the notation of 
%             [1, Section 1.1], each a submatrix of B, over m = 1,...,n-1. 


    N = length(m);
    M = sum(m);
    assert(isequal(size(A), [M M]), 'A must be an MxM matrix with M = sum(m)');
    assert(ishermitian(A), 'Matrix A must be Hermitian');
    assert(n >= 1 && n < N, 'n must be in the range 1 <= n < N');

    % Compute block indices
    blockIdx = [0, cumsum(m)];

    % Extract blocks a_{ij} in the notation of [1,(1.1)]
    a = cell(N, N);
    for i = 1:N
        for j = 1:N
            rows = (blockIdx(i)+1):blockIdx(i+1);
            cols = (blockIdx(j)+1):blockIdx(j+1);
            a{i,j} = A(rows, cols);
        end
    end

    % Construct B (tridiagonal part of A)
    B = cell(N,N);
    for i = 1:N
        for j = 1:N
            if abs(i-j) <= 1
                B{i,j} = a{i,j};
            else
                B{i,j} = zeros(m(i), m(j));
            end
        end
    end

    % Construct full B matrix
    Bfull = zeros(M, M);
    for i = 1:N
        for j = 1:N
            rows = (blockIdx(i)+1):blockIdx(i+1);
            cols = (blockIdx(j)+1):blockIdx(j+1);
            Bfull(rows, cols) = B{i,j};
        end
    end

    % Compute r_L(A), in the notation of [1,(1.3)]
    rL = zeros(1, N-1);
    for i = 2:N
        rL(i-1) = norm(a{i, i-1}, 2);
    end
    rL_val = max(rL);

    % Compute ||C||, where C is the so-called remaining part of A, see 
    % [1, Section 1.1]
    C = A - Bfull;
    normC = norm(C);

    % Compute the output sigmaHat
    sigmaHat = [];
    for k = 0:N-n
        rowRange = (blockIdx(k+1)+1):(blockIdx(k+n+1));
        Bnk = Bfull(rowRange, rowRange);
        sigmaHat = [sigmaHat, eig(Bnk)'];
    end
    sigmaHat = sort(real(sigmaHat)); % Real row vector

    % Compute the output sigma
    sigma = sigmaHat;
    for mVal = 1:(n-1)
        rows1 = (blockIdx(1)+1):(blockIdx(mVal+1));
        rows2 = (blockIdx(N-mVal+1)+1):(blockIdx(N+1));
        B1 = Bfull(rows1, rows1);
        B2 = Bfull(rows2, rows2);
        sigma = [sigma, eig(B1)', eig(B2)'];
    end
    sigma = sort(real(sigma)); % Real row vector

end
\end{verbatim}

\subsubsection{mergeIntervals.m and intersectIntervals.m} \label{code:int}

\begin{verbatim}
function outputIntervals = mergeIntervals(inputIntervals)
%
% function outputIntervals = mergeIntervals(inputIntervals)
%
% This function, given finitely many intervals [a_i,b_i], i = 1,...,I,
% computes a minimal set of intervals [A_i,B_i], i = 1,...,J 
% (minimal in the sense that J is minimal), such that 
% \cup_{i=1,..,J} [A_i,B_i] = \cup_{i=1,...,I} [a_i,b_i].
%
% Input variable:
% inputIntervals - real Ix2 array, with [a_i,b_i] in the ith row, and   
%                  a_i <= b_i for each i.
%
% Output variable:
% outputIntervals - real Jx2 array, with [A_i,B_i] in the ith row, and
%                   A_i <= B_i for each i.
%
% Apart from these initial comments, this function is written by ChatGPT.
%
    % Sort intervals by their start points
    inputIntervals = sortrows(inputIntervals, 1);

    % Initialize the merged intervals list
    merged = [];
    
    for i = 1:size(inputIntervals, 1)
        if isempty(merged) || inputIntervals(i, 1) > merged(end, 2)
            % No overlap, add new interval
            merged = [merged; inputIntervals(i, :)];
        else
            % Overlap, merge intervals
            merged(end, 2) = max(merged(end, 2), inputIntervals(i, 2));
        end
    end

    outputIntervals = merged;
end
\end{verbatim}

\begin{verbatim}
function output = intersectIntervals(A, B)
%
% function output = intersectIntervals(A,B)
%
% Function to compute the intersection of two sets, each the union of
% finitely many disjoint intervals. In more detail, this function,
% given two sets of finitely many disjoint intervals, 
% [a_i,A_i], i = 1,...,I, and [b_i,B_i], i = 1,...,J, computes the 
% minimal set of intervals [c_i,C_i], i = 1,...,K (minimal in the sense 
% that K is minimal), such that
%
% \cup_{i=1,..,K} [c_i,C_i] = 
%             \cup_{i=1,...,I} [a_i,A_i] \cap \cup_{i=1,...,J} [b_i,B_i].
%
% Input variable:
% A - real Ix2 array, with [a_i,A_i] in ith row, and a_i <= A_i for each i,
%     and A_i < a_{i+1}, i = 1,...,I-1.
% B - real Jx2 array, with [b_i,B_i] in ith row, and b_i <= B_i for each i,
%     and B_i < b_{i+1}, i = 1,...,J-1.
%
% Output variable:
% output - real Kx2 array, with [c_i,C_i] in the ith row, and
%          c_i <= C_i for each i, and C_i < c_{i+1}, i = 1,...,K-1.
%
% Apart from these initial comments, and the final line that we have added,
% this function is written by ChatGPT.
%

i = 1; j = 1;
output = [];

while i <= size(A,1) && j <= size(B,1)
    a1 = A(i,1); a2 = A(i,2);
    b1 = B(j,1); b2 = B(j,2);
   
    % Compute the overlap
    left  = max(a1, b1);
    right = min(a2, b2);
   
    if left <= right
        output = [output; left, right];
    end
   
    % Move to next interval
    if a2 < b2
        i = i + 1;
    else
        j = j + 1;
    end
end

% Add a final merge of the intervals
output = mergeIntervals(output);

end
\end{verbatim}

\subsubsection{ThreeToeplitzExample.m and Toep3\_tau\_method.m} \label{code:Toep3}

\begin{verbatim}
% Script file ThreeToeplitzExample.m, that produces additional figure
% for the supplementary materials to [1].
%
% [1] S. N. Chandler-Wilde and M. Lindner, Gershgorin-type inclusion sets
%     for matrices. Submitted for publication, 2024.
%
% The file is an example of using tau_method.m to compute the tau method
% inclusion sets of [1] for a Hermitian matrix A. It produces a plot of the 
% intersection with the real line of the tau method inclusion sets of [1] 
% for Spec A, the set of eigenvalues of A, in the case that A is an order 
% M matrix A that is tridiagonal and 3-Toeplitz with the finite 3-periodic
% sequence (0,1,0,0,1,0,...) on the main diagonal and 1's in the first 
% sub- and super-diagonal. The plot also displays Spec A. 
%

% The following lines set values for the parameters M and nn; the tau
% method inclusion set will be produced for each value of n in the vector
% nn. These values can be altered by the user. nn can be any row vector of
% positive integers. M can be any positive integer > max(nn), subject to
% system memory constraints, that the code constructs a full matrix of order
% M.
disp('Display the values of the parameters M and nn:')
M = 260; disp(['M = ',num2str(M)])
disp('n values are ')
nn = [1 2 4 8 16 32 64 128 256]; disp(nn)



% Now compute R \cap \Sigma_0^n(A), the nth tau method inclusion
% set (intersected with the real line), and Spec A, by a call to
% Toep3_tau_method, which in turn calls tau_method to compute R \cap 
% \Sigma_0^n(A). First use n = nn(1)
figure(1)
n = nn(1);
[Sigma, SpecA] = Toep3_tau_method(M,n);

% Plot Spec A, the eigenvalues of A
yax = [0.5,length(nn)+0.5];
for i = 1:M
    plot(SpecA(i)*ones,yax,'b-')
    hold on
end

% Now plot R \cap \Sigma_0^n(A) for n = nn(1)
ones = [1,1];
plot(Sigma(1,:), ones, 'k-', 'LineWidth', 2)
hold on
for m = 2:size(Sigma,1)
    plot(Sigma(m,:), ones, 'k-', 'LineWidth', 2);
end
text(Sigma(end,2)+0.2, 1, ['$n = $',num2str(nn(1))], 'Interpreter', ...
    'latex', 'FontSize', 20)
% Now repeat for n = nn(j), j = 2,...,length(nn)
for j = 2:length(nn)
    n = nn(j);
    Sigma = Toep3_tau_method(M,n);
    for m = 1:size(Sigma,1)
        plot(Sigma(m,:), j*ones, 'k-', 'LineWidth', 2);
    end
    text(Sigma(end,2)+0.2, j, ['$n = $',num2str(nn(j))], 'Interpreter', ...
    'latex', 'FontSize', 20)
end

% Now add at the bottom, in magenta, a plot of the spectrum of the 
% corresponding semi-infinite matrix.
thick = 0.05;
shift = 0.37;
fill([-sqrt(3),-1,-1,-sqrt(3)], shift+[-thick,-thick,thick,thick], ...
    'm-', 'EdgeColor', 'none');
fill([1-sqrt(2),1,1, 1-sqrt(2)], shift+[-thick,-thick,thick,thick], ...
    'm-', 'EdgeColor', 'none');
fill([1+sqrt(2), sqrt(3), sqrt(3), 1+sqrt(2)], ...
    shift+[-thick,-thick,thick,thick], 'm-', 'EdgeColor', 'none');
plot([-(sqrt(5)-1)/2,-(sqrt(5)-1)/2], shift+[-thick,thick], 'm-');
hold off

ax = gca; ax.FontSize = 14;
axis([-3, 4, 0, length(nn)+1])
ax.YTick = [];
ax.YColor = 'none';
title(['$M=$',num2str(M)],'Interpreter','latex','FontSize',20);
exportgraphics(gcf,'Toep3.png');
\end{verbatim}

\begin{verbatim}
function [Sigma, SpecA] = Toep3_tau_method(M,n)
%
% function [Sigma, SpecA] = Toep3_tau_method(M,n)
%
% This function evaluates the intersection of the tau-method inclusion set, 
% Sigma_0^n(A), with the real line, in the case that the order M natrix
% A is a tridiagonal 3-Toeplitz matrix with 1's on the first sub- and 
% superdiagonal and the 3-periodic main diagonal given by
% 
% (0,1,0,0,1,0,...)                      (1)
% 
% For each positive integer n < M, Sigma_0^n(A) \cap R is an inclusion set 
% for Spec A, the set of eigenvalues of A, which is also produced as an
% output. 
% 
% This function evaluates Sigma_0^n(A) \cap R directly from the
% definition of Sigma_0^n(A), which is [1,Equation (1.6)].
%
% [1] S. N. Chandler-Wilde and M. Lindner, Gershgorin-type inclusion sets
%     for matrices. Submitted for publication, 2024.
%
% Input variables:
% M - positive integer, the order of A
% n - positive integer, with n < M
%
% Output variables:
% Sigma - real Ix2 array, for some positive integer I. The inclusion set
%         R \cap Sigma_0^n(A) consists of I disjoint closed intervals,
%         [a_i,b_i], i = 1,...,I, with I = 1 or 2 and a_i <= b_i, i = 1,
%         ...,I. The ith row of Sigma is [a_i,b_i].
% SpecA - a row vector containing the real eigenvalues of A
%

    % Construct the matrix A, which is order M, tridiagonal and 3-Toeplitz, 
    % as described above.
    A = zeros(M);
    for i = 1:M-1
        A(i,i+1)= 1;
        A(i+1,i) = 1;
    end
    for i = 1:floor((M+1)/3)
        A(3*i-1,3*i-1) = 1;
    end

    % Now compute R \cap Sigma_0^n(A)
    m = ones(1,M);
    Sigma = tau_method(A,m,n);

    % Now compute Spec A
    SpecA = eig(A)';
end
\end{verbatim}

\end{document}